\documentclass[11pt,a4paper]{amsart}

\usepackage{graphicx}
\graphicspath{{./Figures/}}
\usepackage{subfigure}
\usepackage{caption}
\usepackage[inline]{enumitem}
\usepackage{amssymb}
\usepackage{xcolor}
\usepackage{stmaryrd}
\usepackage{appendix}

\theoremstyle{plain}
\newtheorem{theorem}{Theorem}[section]
\newtheorem{lemma}[theorem]{Lemma}

\newtheorem{proposition}[theorem]{Proposition}

\theoremstyle{definition}
\newtheorem{definition}[theorem]{Definition}

\theoremstyle{remark}
\newtheorem{remark}[theorem]{Remark}

\numberwithin{equation}{section}

\def\bold#1{\mbox{\boldmath $#1$}}
\newcommand{\uu}[1]{\bold{#1}}

\newcommand{\abs}[1]{\lvert#1\rvert}
\newcommand{\D}{\partial}
\newcommand{\dd}{\mathrm{d}}
\newcommand{\dive}{\mathrm{div}\,}
\newcommand{\bdive}{\uu{\mathrm{div}}\,}
\newcommand{\divM}{\mathrm{div}_{\mathcal{M}}}
\newcommand{\diam}{\mathrm{diam}}

\newcommand{\Dt}{\partial_t}

\newcommand{\gm}{\gamma}

\newcommand{\grd}{\nabla}
\newcommand{\bgrd}{\uu{\nabla}}

\newcommand{\mbb}{\mathbb}

\newcommand{\mcal}{\mathcal}
\newcommand{\CFL}{\mathrm{CFL}}

\newcommand{\norm}[1]{\lVert#1\rVert}
\newcommand{\Norm}[1]{{\left\vert\kern-0.25ex\left\vert\kern-0.25ex\left\vert #1 
    \right\vert\kern-0.25ex\right\vert\kern-0.25ex\right\vert}}

\newcommand{\sign}{\mathrm{sign}}

\newcommand{\half}{\frac{1}{2}}
\newcommand{\veps}{\varepsilon}
\newcommand{\eps}{\epsilon}

\newcommand{\s}{\sigma}

\newcommand{\M}{\mcal{M}}
\newcommand{\E}{\mcal{E}}

\newcommand{\Ds}{D_\sigma}
\newcommand{\Eds}{\tilde{\mcal{E}}(D_\sigma)}
\newcommand{\Eint}{\mcal{E}_{\mathrm{int}}}

\newcommand{\Lm}{L_{\M}(\Omega)}

\newcommand{\Hez}{\uu{H}_{\E,0}(\Omega)}

\newcommand{\dt}{\delta t}
\newcommand{\chark}{\mcal{X}_K}

\newcommand{\fesig}{F_{\epsilon , \sigma}}

\newcommand{\intr}{\mathrm{int}}
\newcommand{\extr}{\mathrm{ext}}
\newcommand{\absq}[1]{\abs{#1}^2}
\def\bold#1{\mbox{\boldmath $#1$}}

\makeindex             

\begin{document}

\title[Asymptotic Preserving and Energy Stable Scheme]{An Asymptotic
  Preserving and Energy Stable Scheme for the Barotropic Euler System
  in the Incompressible Limit}         
 
\author[Arun]{K.~R.~Arun}
\thanks{K.~R.~A.\ gratefully acknowledges Core Research Grant -
  CRG/2021/004078 from Science and Engineering Research Board,
  Department of Science \& Technology, Government of India.} 
\address{School of Mathematics, Indian Institute of Science Education
  and Research Thiruvananthapuram, Thiruvananthapuram 695551, India}  
\email{arun@iisertvm.ac.in, rahuldev19@iisertvm.ac.in,
  mainak17@iisertvm.ac.in}

\author[Ghorai]{Rahuldev Ghorai}

\author[Kar]{Mainak Kar}



\date{\today}

\subjclass[2010]{Primary 35L45, 35L60, 35L65, 35L67; Secondary 65M06, 65M08}

\keywords{Compressible Euler system, Incompressible limit, Asymptotic
  preserving, Finite volume method, MAC grid, Entropy stability}   

\begin{abstract}
    An asymptotic preserving and energy stable scheme for the
    barotropic Euler system under the low Mach number scaling is
    designed and analysed. A velocity shift proportional to the
    pressure gradient is introduced in the convective fluxes, which
    leads to the dissipation of mechanical energy and the entropy
    stability at all Mach numbers. The resolution of the semi-implicit
    in time and upwind finite volume in space fully-discrete scheme involves two
    steps: the solution of an elliptic problem for the density and an
    explicit evaluation for the velocity. The proposed scheme
    possesses several physically relevant attributes, such as the
    positivity of density, the entropy stability and the consistency
    with the weak formulation of the continuous Euler system. The AP
    property of the scheme, i.e.\ the boundedness of the mesh
    parameters with respect to the Mach number and its consistency
    with the incompressible limit system, is shown rigorously. The
    results of extensive case studies are presented to substantiate
    the robustness and efficiency of the proposed method.       
\end{abstract}

\maketitle

\section{Introduction}
\label{sec:Int}

Mechanical processes acting on different spatial and temporal scales
are quite common in many applications in physics, engineering and
industry. Compressible fluid flow problems modelled by the Euler
equations often give rise to such multiscale phenomena. At room
temperatures, pressure pulses travel at a speed of about 343 meters
per second whereas a typical fluid motion under such conditions is
usually for about 10 meters per second. In other words, there can
exist two different velocity scales within the same physical problem
and this apparent disparity in the wave-speeds is usually quantified
using the Mach number which is the ratio of convection velocity to the
sound velocity. From a mathematical point of view, the hydrodynamics
at low Mach numbers merits a careful attention. It has been well
established that solutions of the compressible Euler equations
converge to their incompressible counterparts when the Mach number
goes to zero; see e.g.\ \cite{KM82,Sch94}. In the mathematical
literature, the zero Mach number limit is often treated as a singular
limit under which the purely hyperbolic compressible Euler system
changes its nature to the mixed hyperbolic-elliptic incompressible
Euler system; see \cite{LM98} and the references therein for a
rigorous treatment.     

Since the zero Mach number limit is singular for the compressible
Euler equations, classical explicit time-stepping methods designed for
hyperbolic conservation laws are inadequate for low Mach number
computations. Indeed, explicit schemes are known to exhibit severe
pathologies when used to simulate low Mach number flows; see
\cite{Kle95} for a review. The difficulties associated with the
numerical solution of Euler equations at low Mach numbers are
numerous. Due to the large discrepancy between the fluid and sound
velocities, the CFL condition imposes an acute restriction on the
time-step to maintain stability and this in turn contributes to
stiffness and prohibitive computational costs. One of the early
attempts in this direction to circumvent the stiffness by using
semi-implicit schemes for low Mach number hydrodynamics is available 
in \cite{MRK+03,PM05}. Extension of these ideas to high orders and
their implementation on general unstructured meshes can be found in
\cite{BBD+20,BRV+21, BTB+20,TV17}. If the numerical viscosity
introduced by upwind discretisations depends on the Mach number, then
it can ultimately lead to instabilities or inconsistent numerical
solutions; see e.g.\ \cite{GV99} for more details. Furthermore, it has
been also reported that the presence of spurious acoustic waves can
deteriorate the order of accuracy of Godunov-type explicit finite
volume schemes at low Mach numbers; see \cite{ADS21,AS20,Del10}.     

An effective platform to design robust numerical approximations for
the low Mach number Euler equations, or singularly perturbed
hydrodynamic models of fluids in general, is the so-called `Asymptotic
Preserving' (AP) methodology which was introduced in the context of
kinetic models of diffusive transport \cite{Jin99}. AP schemes provide
a general framework to numerically tackle any singular perturbation
problem and their working principle can be explained as follows. Let
$\mcal{P}_{\veps}$ denote a singularly perturbed problem with $\veps$,
the perturbation parameter. Suppose that in the limit as $\veps \to
0$, the solution of $\mcal{P}_{\veps}$ converges to the solution of a
well-posed problem denoted by  $\mcal{P}_{0}$, called the singular
limit or the limit problem. A numerical scheme for $\mcal{P}_{\veps}$,
denoted by $\mcal{P}_{\veps}^{h}$ with $h$ being a discretisation
parameter, is said to be AP if     
\begin{enumerate}[label=(\roman*)]
\item as $\veps \to 0$ the numerical scheme $\mcal{P}_{\veps}^{h}$
  converges to a numerical scheme $\mcal{P}_{0}^{h}$, which is a
  consistent discretisation of the limit system $\mcal{P}_{0}$ and  
\item the stability constraints on the discretisation parameter $h$
  are independent of $\veps$. 
\end{enumerate}
Therefore, the AP methodology turns out to be a natural choice for the
numerical approximation of the zero Mach number limit in the sense
that it respects the limit at a discrete level. Furthermore, the AP
framework automatically recognises the singular (weakly compressible
or nearly incompressible) and non-singular (compressible) regions in
the flow as well as the transient regions where regime shifts take
place. Thus, using an AP discretisation for low Mach number flows is
an effective method that drastically reduces the computational
complexity while simultaneously enhancing the accuracy.  

A common practice to design AP schemes for hydrodynamic models, such
as the Euler equations, is to employ an implicit-explicit (IMEX) time
discretisation. We refer the reader to 
\cite{AKS22,AS20,BAL+14,BJR+19,BRS18,CGK16,DT11,HJL12,HLS21,NBA+14,TPK20}
for some developments on IMEX AP schemes for stiff hydrodynamic
problems. The IMEX time-stepping procedure relies on a stiff/non-stiff
splitting of the fluxes and subsequently, the stiff terms are treated
implicitly and the non-stiff terms explicitly. The semi-implicit
formalism has the advantage that it avoids the need to invert large
and dense matrices, typical of fully-implicit schemes, at the same
time it overcomes the restrictive stability conditions of
fully-explicit schemes. In the case of the barotropic Euler equations, a
semi-implicit AP scheme necessitates the implicit treatment of the
mass flux and the pressure gradient in the momentum flux
\cite{AS20,BRS18,DT11}. A reformulation of the resulting scheme can be
performed, nonetheless, to yield an elliptic equation for the pressure
which is nothing but the time discretisation of a wave equation for
the pressure in the continuous case. However, the discrete diffusion
operator in the above elliptic problem, arising by combining discrete
gradient and divergence operators, does not satisfy the so-called
`inf-sup' condition. As a result, the scheme might require an
additional stabilisation procedure at low Mach numbers
\cite{MKB+12}. A possible cure to this ailment is to extend the
standard, well-established discretisation techniques of
incompressible flows to the compressible case. A first attempt in this
direction can be found in \cite{HA68} which can be seen as an 
extension of the well-known `Marker and Cell' (MAC) scheme
\cite{HW65}. We refer the interested reader to
\cite{BW98,CP99,IGW86,KP89,KSG+09,WPM02} for related treatments. Yet
another demand while dealing with the numerical approximation of
hyperbolic models of compressible fluids is to maintain the entropy
stability. Since the solutions are known to develop discontinuities in
finite time, the entropy inequalities are vital in choosing the
physically valid weak solutions. Obtaining entropy stable schemes,
though nontrivial, is mandatory in applications, such as the low Mach
number hydrodynamics, to derive energy estimates and to perform
rigorous asymptotic convergence analysis.  

The goal of the present work is the design and analysis of an AP,
semi-implicit and entropy stable scheme for the barotropic Euler
system on a MAC grid. The key to entropy stability is the introduction
of a shifted velocity in the convective fluxes of mass and momenta,
following the ideas introduced in
\cite{CDV17,DVB17,DVB20,GVV13,PV16}. The velocity shift is
proportional to the stiff pressure gradient, it stabilises the flow
via the dissipation of mechanical energy at all Mach numbers, and the
parameters involved in it have to respect a CFL-type condition. The
present scheme is closely related to the so-called non-incremental
projection scheme due to Chorin \cite{CHO68}. A semi-implicit time
discretisation is performed to overcome stiff stability
restrictions. The upwind space discretised fully-discrete scheme
satisfies the required apriori entropy stability inequalities as in
the continuous model. The velocity stabilised mass update gives rise
to a well-posed, semilinear elliptic problem for the pressure, thanks
to the MAC grid discretisation. After the resolution of the elliptic
problem, the momentum update can be evaluated explicitly. We carry out
a mathematical and numerical analysis of the proposed scheme that is
motivated to an extent by an analogous analysis performed for the
prediction-correction semi-implicit scheme in \cite{HLS21}. However,
the key differences are that the semi-implicit scheme of \cite{HLS21}
is multi-step and it requires an almost well-prepared initial datum to
obtain the AP property. For the present velocity stabilised
semi-implicit scheme, the existence of the numerical solution and the 
positivity of the mass density are established by exploiting some
topological degree theory results. A weak consistency analysis of the
numerical scheme is performed to establish its consistency with the
weak formulation of the Euler equations. A sufficient condition to
enforce the stability restriction shows the scheme's ability to admit
large time-steps in low Mach number regimes. The AP property of the
scheme, i.e.\ the boundedness of the time-steps with respect to the
Mach number and the consistency of the scheme with the incompressible
limit, is shown rigorously and numerically by means of extensive case
studies.   

The rest of this paper is organised as follows. In
Section~\ref{sec:cont-case} we recall some apriori energy stability
estimates that are crucial to establish the incompressible limit and
the velocity stabilisation technique based on the energy
stability. The discretisation of the domain and the discrete
differential operators are introduced in
Section~\ref{sec:MAC_disc_diff}. The semi-implicit numerical scheme
and its energy stability attributes are presented in
Section~\ref{sec:upwind_scheme}. In Section~\ref{sec:weak_cons} we
present the weak consistency results and in
Section~\ref{sec:cons_incomp}, the asymptotic consistency with the
incompressible Euler system. The results of numerical case studies are
presented in Section~\ref{sec:num_res}. Finally, the paper is
concluded with some remarks in Section~\ref{sec:conclusion}.  

\section{Apriori Energy Estimates and the Incompressible Limit}
\label{sec:cont-case}

We start with the following initial boundary problem for the
compressible Euler system parametrised by the Mach number $\veps$: 
\begin{align}
  \D_t\rho^\veps+\dive (\rho^\veps\uu{u}^\veps)&=0, \label{eq:cons_mas}\\ 
  \D_t(\rho^\veps\uu{u}^\veps)+\bdive
  (\rho^\veps\uu{u}^\veps\otimes\uu{u}^\veps)+\frac{1}{\veps^2}\bgrd p^\veps &=0, \label{eq:cons_mom} \\
  \rho^\veps\vert_{t=0}=\rho^\veps_0, \quad
  \uu{u}^\veps\vert_{t=0}=\uu{u}^\veps_0, \quad \uu{u}^\veps\vert_{\partial\Omega}=0, \label{eq:eq_ic}
\end{align}
for $(t,\uu{x})\in Q:=(0,T)\times\Omega$, where $\Omega$ is an open,
bounded and connected subset of $\mbb{R}^d$, $d\geq1$. The parameter
$\veps\in(0,1]$, usually known as the reference Mach number, is an
infinitesimal and is defined as the ratio of a characteristic fluid
velocity to that of a sound velocity. The variables $\rho^\veps$ and
$\uu{u}^\veps$ are the  density and velocity of the fluid,
respectively. The pressure $p^\veps = p(\rho^\veps)$ is assumed to follow a barotropic
equation of state $p(\rho):=\rho^\gamma$ with $\gamma\geq 1$ being the
ratio of specific heats. 

We first review some apriori energy estimates satisfied by the
solutions of \eqref{eq:cons_mas}-\eqref{eq:cons_mom}, which are needed
to perform the incompressible limit $\veps\to0$. To this end, we start
with the internal energy per unit volume or the so-called Helmholtz
function: 
\begin{equation}
   \label{eq:psi_gamma}
   \psi_\gamma(\rho) :=
   \begin{cases}
     \rho\ln\rho, & \mbox{if} \ \gamma=1, \\
     \dfrac{\rho^\gamma}{\gamma-1}, & \mbox{if} \ \gamma>1.
   \end{cases}  
\end{equation}
The internal energy $\mcal{I}_\veps$ and kinetic energy $\mcal{K}_\veps$ of the system
\eqref{eq:cons_mas}-\eqref{eq:cons_mom} are defined by 
\begin{equation}
  \mcal{I}_{\veps}(t):=\int_{\Omega}\psi_\gamma(\rho^\veps)\dd\uu{x}, \quad
  \mcal{K}_{\veps}(t):= \int_{\Omega}\frac{1}{2}\rho^\veps{\abs{\uu{u}^\veps}}^2
  \dd\uu{x}. \label{eq:int_kin_engy}
\end{equation}

\subsection{Apriori Energy Estimates}
\label{sec:engy_est}
\begin{proposition}
  \label{prop:engy_balance}
  We recall from \cite{HLS21} the following identities satisfied by
  regular solutions of \eqref{eq:cons_mas}-\eqref{eq:cons_mom}. 
  \begin{enumerate}[label=(\roman*)]
  \item A renormalisation identity:
    \begin{equation}
      \label{eq:renorm}
      \Dt\psi_\gamma(\rho^\veps)
      +\dive\left(\psi_\gamma(\rho^\veps)\uu{u}^\veps\right)+p^\veps\dive\uu{u}^\veps
      =0. 
    \end{equation}
  \item A positive renormalisation identity:
    \begin{equation}
      \label{eq:porenorm}
      \Dt\Pi_\gm(\rho^\veps)
      +\dive\left(\psi_\gamma(\rho^\veps)-\psi_\gamma^\prime(1)\rho^\veps\right)\uu{u}^\veps+p^\veps\dive\uu{u}^\veps=0,
    \end{equation}
    where
    $\Pi_\gm(\rho):=\psi_\gamma(\rho)-\psi_\gamma(1)-\psi_{\gamma}^{\prime}(1)(\rho-1)$ 
    is the relative internal energy, which is an affine approximation
    of $\psi_\gamma$ with respect to the constant state $\rho=1$. 
  \item The kinetic energy identity:
    \begin{equation}
      \label{eq:kinbal}
      \Dt\Big(\frac{1}{2}\rho^\veps{\abs{\uu{u}^\veps}}^2\Big)
      +\dive\Big(\frac{1}{2}\rho^\veps{\abs{\uu{u}^\veps}}^2\uu{u}^\veps\Big)
      +\frac{1}{\veps^2}\grd p^\veps\cdot\uu{u}^\veps
      =0.
    \end{equation}
  \item The total energy identity:
  \begin{equation}
  \label{eq:eng_id}
  \Dt\Big(\frac{1}{2}\rho^\veps{\abs{\uu{u}^\veps}}^2+\frac{1}{\veps^2}\Pi_\gamma(\rho^\veps)\Big)
      +\dive\Big(\frac{1}{2}\rho^\veps{\abs{\uu{u}^\veps}}^2+\frac{1}{\veps^2}\psi_\gamma(\rho^\veps)-\frac{1}{\veps^2}\psi_\gamma^{\prime}(1)\rho^\veps+\frac{1}{\veps^2}p^\veps\Big)\uu{u}^\veps=0.      
  \end{equation}
  \end{enumerate}
\end{proposition}
\begin{proof}
The proof follows from straightforward calculations; see \cite{HLS21} and the references therein.
\end{proof}
\subsection{Incompressible Limit}
\label{sec:incomp_lim}
In this section we introduce the zero Mach number or the
incompressible limit of the Euler system
\eqref{eq:cons_mas}-\eqref{eq:cons_mom}, which is obtained by letting
$\veps\to0$. In order to perform the above limit, we first introduce
the notion of the so-called `ill-prepared' initial data; see
\cite{HLS21} for further details and also \cite{Del10} for some
related discussions on the relevance of initial data in carrying out
the incompressible limit. 
\begin{definition}
    \label{def:ill_prep_id}
    An initial datum $(\rho^\veps_0,\uu{u}^{\veps}_0)$ of
    \eqref{eq:cons_mas}-\eqref{eq:cons_mom} is called ill-prepared if
    $(\rho^{\veps}_{0},\uu{u}^{\veps}_{0})\in
    L^{\infty}(\Omega)^{1+d}$ with $\rho^{\veps}_{0}>0$, and satisfy
    the bound:  
    \begin{equation}
    \label{eq:ill_prep_id}
        \norm{\uu{u}_{0}^{\veps}}_{L^{2}(\Omega)^d}+\frac{1}{\veps}\norm{\rho_{0}^{\veps}-1}_{L^{\infty}(\Omega)}\leq C,
    \end{equation}
    where $C>0$ is a constant that does not depend on $\veps$. 
\end{definition}
\begin{remark}
    The estimate \eqref{eq:ill_prep_id} implies that $\rho_0^\veps-1=\mcal{O}(\veps)$ and that $\uu{u}^\veps_0$ is uniformly bounded in $L^{2}(\Omega)^d$ with respect to $\veps$. The so-called `well-prepared' data from the literature, e.g.\ \cite{KM82,Sch94}, requires more stringent restrictions, such as $\rho_{0}^{\veps}-1=\mcal{O}(\veps^2)$, a uniform bound on $\uu{u}_0^\veps$ in $H^1(\Omega)^d$ and $\dive\uu{u}_0^\veps$ be close to zero.
\end{remark}

In the subsequent discussions we will assume the existence of a weak solution $(\rho^\veps,\uu{u}^\veps)\in L^\infty(Q)^{1+d}$ of \eqref{eq:cons_mas}-\eqref{eq:cons_mom} that satisfies the total entropy inequality:
 \begin{equation}
 \label{eq:tot_ent_ineq_cont}
 \Dt\Big(\frac{1}{2}\rho^\veps{\abs{\uu{u}^\veps}}^2+\frac{1}{\veps^2}\Pi_\gamma(\rho^\veps)\Big)
       +\dive\Big(\frac{1}{2}\rho^\veps{\abs{\uu{u}^\veps}}^2+\frac{1}{\veps^2}\psi_\gamma(\rho^\veps)-\frac{1}{\veps^2}\psi_\gamma^{\prime}(1)\rho^\veps+\frac{1}{\veps^2}p^\veps\Big)\uu{u}^\veps\leq 0.      
 \end{equation}
 Integrating \eqref{eq:tot_ent_ineq_cont} over $Q$ and taking into account the homogeneous Dirichlet boundary condition on the velocity, we further obtain that $(\rho^\veps,\uu{u}^\veps)$ satisfies the following total entropy estimate:
 \begin{equation}
 \label{eq:tot_ent_est_cont}
     \half\int_\Omega\rho^{\veps}(t)\absq{\uu{u}^\veps(t)}\dd\uu{x}+\frac{1}{\veps^2}\int_\Omega\Pi_\gamma(\rho^\veps(t))\dd\uu{x}\leq \half\int_\Omega\rho^{\veps}_0\absq{\uu{u}^\veps_0}\dd\uu{x}+\frac{1}{\veps^2}\int_\Omega\Pi_\gamma(\rho^\veps_0)\dd\uu{x}.
 \end{equation}
 The entropy estimate \eqref{eq:tot_ent_est_cont} yields the
 convergence of $\rho^\veps$ to $1$ as $\veps$ goes to $0$. Precisely,
 it can be shown that $\rho^\veps\to 1$ in $L^\infty(0,T;L^r(\Omega))$
 for every $r\in[1,\mathrm{min}\{2,\gamma\}]$ and
 $\uu{u}^\veps\overset{\ast}{\rightharpoonup}\uu{U}$ in
 $L^\infty(0,T;L^2(\Omega)^d)$ as $\veps\to0$; see \cite{HLS21, KM82,
   LM98} and the references therein for details. Furthermore,
 $(\pi,\uu{U})\in L^\infty(0,T;L^2(\Omega))^{1+d}$ is a weak solution
 of the initial value problem  
\begin{gather}
    \dive \uu{U}=0, \\
    \Dt\uu{U}+\uu{\dive}(\uu{U}\otimes\uu{U})+\grd \pi =0, \\
    \uu{U}(0,\cdot)=\uu{U}_0, \ \dive \uu{U}_{0} =0, 
\end{gather}
for $(t,\uu{x})\in Q$, which is formally the incompressible limit of
\eqref{eq:cons_mas}-\eqref{eq:cons_mom}. 

\subsection{Velocity Stabilisation}
\label{sec:stab}

In the following, we introduce a semi-implicit scheme for the system
\eqref{eq:cons_mas}-\eqref{eq:cons_mom} and study its stability in the
sense of numerical control of total energy. In other words, we aim at
obtaining a discrete equivalent of the energy stability
\eqref{eq:eng_id}. In order to enforce the energy stability of the
numerical solution, we adopt the formalism introduced in \cite{DVB17,
  DVB20, GVV13,PV16}. The idea is to apply a stabilisation, therein a
shifted velocity is introduced in the mass and momentum fluxes to
yield the modified system: 
\begin{align}
  \D_t\rho^\veps+\dive (\rho^\veps(\uu{u}^\veps-\delta\uu{u}^\veps))&=0, \label{eq:r_cons_mas}
  \\ 
  \D_t(\rho^\veps\uu{u}^\veps)+\bdive
  (\rho^\veps\uu{u}^\veps\otimes(\uu{u}^\veps-\delta\uu{u}^\veps))+\frac{1}{\veps^2}\bgrd
  p^\veps &=0. \label{eq:r_cons_mom}
\end{align}
Analogous to Proposition~\ref{prop:engy_balance}, one can derive the
following apriori estimates for the solutions of the modified system
\eqref{eq:r_cons_mas}-\eqref{eq:r_cons_mom}. The expression for the
stabilisation term $\delta\uu{u}^\veps$ is determined accordingly from
the total energy estimate so that it ensures the entropy stability. 
\begin{proposition}
  \label{prop:r_engy_balance}
  Regular solutions of \eqref{eq:r_cons_mas}-\eqref{eq:r_cons_mom} satisfy the
  following identities. 
  \begin{enumerate}[label=(\roman*)]
  \item A renormalisation identity:
    \begin{equation}
      \label{eq:r_renorm}
      \Dt\psi_\gamma(\rho^\veps)
      +\dive\left(\psi_\gamma(\rho^\veps)(\uu{u}^\veps-\delta\uu{u}^\veps)\right)+p^\veps\dive(\uu{u}^\veps-\delta\uu{u}^\veps)=0.
    \end{equation}
  \item A positive renormalisation identity:
    \begin{equation}
      \label{eq:r_porenorm}
     \Dt\Pi_\gamma(\rho^\veps)
      +\dive\big(\psi_\gamma(\rho^\veps)-\psi_\gamma^\prime(1)\rho^\veps\big)(\uu{u}^\veps-\delta\uu{u}^\veps)+p^\veps\dive(\uu{u^\veps}-\delta\uu{u}^\veps)=0.
    \end{equation}
  \item The kinetic energy balance:
    \begin{align}
      \label{eq:r_kinbal}
      \Dt\Big(\frac{1}{2}\rho^\veps{\abs{\uu{u}^\veps}}^2\Big)
      +\dive\Big(\frac{1}{2}\rho^\veps{\abs{\uu{u}^\veps}}^2(\uu{u}^\veps-\delta\uu{u}^\veps)\Big)
      +\frac{1}{\veps^2}(\uu{u}^\veps-\delta\uu{u}^\veps)\cdot\grd p^\veps=-\frac{1}{\veps^2}\delta\uu{u}^\veps\cdot\grd p^\veps.
    \end{align}
  \item Adding \eqref{eq:r_porenorm} and \eqref{eq:r_kinbal} yields the entropy balance:
  \begin{align}
    \label{eq:r_eng_id}
    \Dt\Big(\frac{1}{\veps^2}\Pi_\gm(\rho^\veps)+\frac{1}{2}\rho^\veps{\abs{\uu{u}^\veps}}^2\Big)+\dive\Big(\frac{1}{2}\rho^\veps{\abs{\uu{u}^\veps}}^2+\frac{1}{\veps^2}\psi_\gamma(\rho^\veps)-\frac{1}{\veps^2}\psi_\gamma^\prime(1)\rho^\veps+\frac{1}{\veps^2}p^\veps\Big)(\uu{u}^\veps-\delta\uu{u}^\veps)=-\frac{1}{\veps^2}\delta\uu{u}^\veps\cdot\grd p^\veps. 
  \end{align}
  \end{enumerate}
\end{proposition}
\begin{proof}
    The proof is classical and uses straightforward manipulations as in Proposition~\ref{prop:engy_balance}.
\end{proof}
Thus, at the continuous level, we immediately see that if we formally take $\delta\uu{u}^\veps=\eta\grd p^\veps$ with $\eta>0$, then we get the entropy stability inequality:
\begin{equation}
    \label{eq:r_eng_id_stab}
    \Dt\Big(\frac{1}{\veps^2}\Pi_\gm(\rho^\veps)+\frac{1}{2}\rho^\veps{\abs{\uu{u}^\veps}}^2\Big)+\dive\Big(\frac{1}{2}\rho^\veps{\abs{\uu{u}^\veps}}^2+\frac{1}{\veps^2}\psi_\gamma(\rho^\veps)-\frac{1}{\veps^2}\psi_\gamma^\prime(1)\rho^\veps+\frac{1}{\veps^2}p^\veps\Big)(\uu{u}^\veps-\delta\uu{u}^\veps)=-\frac{1}{\veps^2}\eta\abs{\grd p^\veps}^2 \leq 0. 
\end{equation}
In other words, shifting the velocity has a stabilising effect, which ultimately yields the energy stability; see \cite{DVB17, DVB20, GVV13,PV16}.  Consequently, as $\veps$ tends to zero, the modified system \eqref{eq:r_cons_mas}-\eqref{eq:r_cons_mom} formally converges to the following velocity stabilised incompressible Euler equations:
\begin{gather}
    \dive(\uu{u}-\delta\uu{u})=0,\label{eq:stab_incomp_eul_divfree}\\
    \D_t\uu{u}+\bdive(\uu{u}\otimes(\uu{u}-\delta\uu{u}))+\bgrd\pi=0,\label{eq:stab_incomp_eul_pres}
\end{gather}
where $\pi$ is the formal limit of $\frac{p(\rho^\veps)-1}{\veps^2}$ and $\delta\uu{u}=\eta\bgrd\pi$.
Motivated by the above, we design a semi-implicit scheme in which the
velocity stabilisation introduced via numerical flux functions is the
key to achieve nonlinear energy stability. Using the energy estimate,
we establish a discrete analogue of the incompressible limit. 

\section{Domain Discretisation and Discrete Differential Operators}
\label{sec:MAC_disc_diff}
In order to approximate the velocity stabilised Euler system
\eqref{eq:r_cons_mas}-\eqref{eq:r_cons_mom} in a finite volume framework, we take a computational space-domain $\Omega\subseteq \mbb{R}^d$, such that the closure of $\Omega$ is the union of closed rectangles ($d=2$) or closed orthogonal parallelepipeds ($d=3$).
\subsection{Mesh and Unknowns}
\label{subsec:msh_unkn}
In the subsequent discussion, we introduce a discretisation of the domain $\Omega$ using a marker and cell (MAC) grid and the corresponding discrete function spaces; see \cite{Cia91, GHL+18, GHM+16, HW65} for more details. 
A MAC grid is a pair $\mcal{T}=(\mcal{M},\mcal{E})$, where $\mcal{M}$ is called the primal mesh which is a partition of $\bar{\Omega}$ consisting of possibly non-uniform closed rectangles ($d=2$) or parallelepipeds ($d=3$) and $\mcal{E}$ is the collection of all edges of the primal cells. For each $\sigma\in\mcal{E}$, we construct a dual cell $\Ds$ which is the union of half-portions of the primal cells $K$ and $L$, where $\s=\bar{K}\cap\bar{L}$. Furthermore, we decompose $\E$ as $\mcal{E}=\cup_{i=1}^d\mcal{E}^{(i)}$, where $\mcal{E}^{(i)}=\mcal{E}_\intr^{(i)}\cup\mcal{E}_\extr^{(i)}$. Here, $\mcal{E}_\intr^{(i)}$ and $\mcal{E}_\extr^{(i)}$ are, respectively, the collection of $d-1$ dimensional internal and external edges that are orthogonal to the $i$-th unit vector $e^{(i)}$ of the canonical basis of $\mbb{R}^d$. We denote by $\mcal{E}(K)$, the collection of all edges of $K\in\mcal{M}$ and $\tilde{\mcal{E}}(D_\sigma)$, the collection of all edges of the dual cell $D_\sigma$. We define the mesh parameter 
$h = \max\{\diam(K): K\in\M \}$.

Now, we define a discrete function space $L_{\mcal{M}}(\Omega) \subset L^{\infty}(\Omega)$, consisting of scalar valued functions which are piecewise constant on each primal cell $K\in\mcal{M}$. Analogously, we denote by $\uu{H}_{\mcal{E}}(\Omega)=\prod_{i=1}^{d} H^{(i)}_{\mcal{E}}(\Omega)$, the set of vector valued (in $\mbb{R}^d$) functions which are constant on each dual cell $D_\sigma$ and for each $i=1,2,\dots,d$. The space of vector valued functions vanishing on the external edges is denoted as  $\uu{H}_{\mcal{E},0}(\Omega)=\prod_{i=1}^d H^{(i)}_{\mcal{E},0}(\Omega)$, where  $H^{(i)}_{\mcal{E},0}(\Omega)$ contains those elements of $H^{(i)}_{\mcal{E}}(\Omega)$ which vanish on the external edges. For a primal grid function $q\in L_{\mcal{M}}(\Omega)$, such that $q =\sum_{K\in\mcal{M}}q_K\chark$, and for each $\sigma = K|L \in\cup_{i=1}^d\mcal{E}^{(i)}_\intr$, the dual average $q_{D_\sigma}$ of $q$ over $D_\sigma$ is defined via the relation
  \begin{equation}
   \label{eq:mass_dual}
     \abs{D_\sigma}q_{D_\sigma}=\abs{D_{\sigma,K}}q_K+\abs{D_{\sigma,L}}q_L.
  \end{equation}
\begin{remark}
The assumption that the dual variables vanish at the boundary is taken only for the sake of simplicity of the following exposition. In a similar setup, we can consider the space of piecewise constant functions on the dual grid, which can be extended to external dual cells in order to implement periodic boundary conditions. In this case, an external dual cell is obtained by adjoining half portions of a primal cell and a fictitious primal cell outside the domain. Furthermore, the function values on opposite external dual cells are identified. We have implemented periodic boundary conditions in all of the numerical test problems in Section \ref{sec:num_res}.
\end{remark}

\subsection{Discrete Convection Fluxes and Differential Operators}
\label{sec:dic_convect}
In this section, we introduce the discrete convection fluxes and
discrete differential operators on the functional spaces described
above.  
\begin{definition}\label{def:disc_conv_flux}
  Assume a discretisation of $\Omega$ with a MAC grid
  $\mcal{T}=(\mcal{M},\mcal{E})$ and the discrete function spaces as
  defined above.  
  \begin{enumerate}[label=(\roman*)]
  \item For each $K \in \mcal{M}$ and $\sigma\in \mcal{E}(K)$, $\sigma
    = K|L$, the mass flux $F_{\sigma,K} \colon  L_{\mcal{M}}(\Omega)
    \times \uu{H}_{\mcal{E},0}(\Omega) \to \mbb{R}$ is defined by the
    following splitting of the positive and negative components: 
    \begin{equation}
      \label{eq:mass_flux}
      F_{\sigma,K}(\rho,\uu{v}):=\abs{\sigma}\big\{\rho_{K}(v_{\sigma,K})^{+}+\rho_{L}(v_{\sigma,K})^{-}\big\},
      \ (\rho,\uu{v})\in L_{\mcal{M}}(\Omega) \times
      \uu{H}_{\mcal{E},0}(\Omega).
    \end{equation}
    Here, $v_{\sigma,K}=v_{\sigma}\uu{e}^{(i)}\cdot\uu{\nu}_{\sigma, K}$, where
    $\uu{\nu}_{\sigma,K}$ is the unit vector normal to the edge
    $\sigma\in\mcal{E}^{(i)}_\intr$ in the direction outward to the
    cell $K$. Note that in Subsection~\ref{sec:stab}, it is the
    velocity stabilisation which yields the energy stability of the
    modified Euler system
    \eqref{eq:r_cons_mas}-\eqref{eq:r_cons_mom}. Motivated by the same
    considerations, we introduce the following stabilisation in the
    discrete setup:   
\begin{equation}
\label{eq:reg_vel}
v_\s^{n}=u_\s^{n}-\delta u_\s(\rho^{n+1}), \ \text{with} \ \delta u_\s(\rho^{n+1}):=\frac{\eta\dt}{\veps^2}(\partial_{\E}^{(i)}p^{n+1})_\s, \ \forall \ 1 \leq i \leq d, \  \forall \s\in\mcal{E}_\intr^{(i)},
\end{equation}
where $\eta>0$ is a parameter which will be determined later. Clearly,
the stabilised velocity $\uu{v}_\s\in \uu{H}_{\E,0}(\Omega)$ since
both the explicit velocity component and the pressure gradient are
defined on the dual cells. It will be shown that the implicit
treatment of the stabilising pressure gradient term is crucial to get
the entropy stability of the numerical scheme. In order to maintain
the sign of the split velocities $(v_{\s,K})^\pm$ and to bring in an
upwind-bias in the mass flux $F_{\s,K}$, we define 
\begin{equation}
    \label{eq:v_sigKpm}
    (v_{\s,K}^n)^+:=u_{\s,K}^{n,+}-\delta u_{\s,K}(\rho^{n+1})^{-}\geq 0, \ 
    (v_{\s,K}^n)^-:=u_{\s,K}^{n,-}-\delta u_{\s,K}(\rho^{n+1})^{+} \leq 0.
\end{equation}
We have denoted by $a^\pm=\half(a \pm \abs{a})$, that splits a real number $a$ in positive and negative valued parts respectively . 
  \item For a fixed $i=1,2,\dots,d$, for each $\sigma\in\mcal{E}^{(i)}, \epsilon\in\tilde{\E}({D_\sigma})$
    and $(\rho,\uu{v},u)\in \Lm\times \uu{H}_{\mcal{E},0}\times H^{(i)}_{\mcal{E},0}$, the net upwind momentum convection flux through the edges of the dual cell $D_\sigma$ is given by the expression
    \label{mom_flux_up} 
    \begin{align}
      \sum_{\epsilon\in\Eds}\fesig(\rho,\uu{v})u_{\epsilon,\mathrm{up}}, 
    \end{align}
    where $\fesig(\rho,\uu{v})$ is the mass flux across the edge
    $\epsilon$ of the dual cell $\Ds$. We follow the convention that
    $\fesig=0$ if $\epsilon\in\mcal{E}_\extr$, otherwise it is defined
    as a suitable linear combination of the primal mass convection
    fluxes through the neighbouring primal edges with constant
    coefficients; see \cite{GHL+18a} for details.  
  \item The following upwind choice is used for obtaining
    $u_{\epsilon,\mathrm{up}}$: 
    \label{eq:mom_up}
    \begin{equation}
    u_{\epsilon,\mathrm{up}}=
    \begin{cases}
        u_{\sigma}, & \mbox{if} \ \fesig(\rho,\uu{v})\geq 0,\\
        u_{\sigma^{\prime}}, & \mbox{otherwise}, 
    \end{cases}
\end{equation}
where $\epsilon\in\Eds$, $\epsilon=\Ds|D_{\sigma^{\prime}}$.
\end{enumerate}
\end{definition}
\begin{definition}[Discrete gradient and discrete divergence]
\label{def:disc_grad_div}
    The discrete gradient operator  $\bgrd_{\mcal{E}}:L_{\mcal{M}}(\Omega)\rightarrow\uu{H}_{\mcal{E},0}(\Omega)$ is defined by the map $q \mapsto \bgrd_{\mcal{E}}q=\Big(\D^{(1)}_{\mcal{E}}q,\D^{(2)}_{\mcal{E}}q,\dots,\D^{(d)}_{\mcal{E}}q\Big)$, where for each $i=1,2,\dots,d$, $\partial^{(i)}_{\mcal {E}}q$ denotes
\begin{equation}
\partial^{(i)}_{\mcal {E}}q=\sum_{\sigma\in \mcal{E}^{(i)}_\intr}(\partial^{(i)}_{\mcal{E}}q)_{\sigma}\mcal{X}_{D_{\sigma}}, \ \mbox{with} \  (\partial^{(i)}_{\mcal{E}}q)_{\sigma}= \frac{\abs{\sigma}}{\abs{D_\sigma}}(q_L-q_K)\uu{e}^{(i)}\cdot \uu{\nu}_{\sigma,K}, \; \sigma=K|L\in \mcal{E}^{(i)}_\intr.
\end{equation}
We set the gradient to zero on the external faces. The discrete divergence operator $\divM:\uu{H}_{\E,0}(\Omega)\rightarrow L_{\mcal{M}}(\Omega)$ is defined as $\uu{v} \mapsto \divM \uu{v}=\sum_{K\in\M}(\dive_{\mcal{M}} \uu{v})_K \mcal{X}_{K}$, where for each $K\in\M$, $(\dive_{\mcal{M}} \uu{v})_K $ denotes
\label{def:disc_div}
\begin{equation}
(\divM \uu{v})_K =\frac{1}{\abs{K}}\sum_{\sigma\in\mcal{E}(K)}\abs{\sigma} v_{\sigma,K}.
\end{equation}
The above discrete operators satisfy the following
`gradient-divergence duality'; see \cite{GHL+18a} for further details.  
\end{definition}
\begin{proposition}
  For any $(q,\uu{v})\in\Lm\times\Hez$, the gradient-divergence duality is given by
  \begin{equation}
    \label{eq:disc_dual}
    \int_{\Omega}q(\divM \uu{v})\dd\uu{x}+\int_{\Omega}\bgrd_{\mcal{E}}q\cdot\uu{v}\dd\uu{x}=0.
  \end{equation}
\end{proposition}

In the following, we state the so-called `inf-sup stability' satisfied by the MAC grid discretisation introduced above. It is well known from the literature, e.g.\ \cite{EG04}, that this condition is crucial in establishing the well-posedness of finite element formulations. As mentioned in the introduction, the inf-sup stability is the key to maintaining the stability of the elliptic problem satisfied by the pressure in $L^2$-norm. Along the lines of \cite{HLS21}, we further deduce that this bound subsequently allows the passage to the incompressible limit of the discrete pressure gradient term. The inf-sup stability can be stated as follows. 
\begin{lemma}[Inf-sup Stability]
\label{lem:inf_sup}
There exists a constant $\beta>0$, depending only on $\Omega$ and the discretisation, such that for all $p =\{p_{K}\colon K\in\mcal{M}\}$, there exists $\uu{u}=\{u_{\s,i} \colon \s\in\mcal{E}^{(i)}, 1\leq i\leq d\}$ in $\uu{H}_{\mcal{E},0}$ satisfying  
\begin{equation}
    \label{eq:inf-sup}
    \norm{\uu{u}}_{1,\mcal{T}}=1 \ \text{and} \  \sum_{K\in\mcal{M}}|K|p_{K}(\dive\uu{u})_{K}\geq\beta\norm{p-m(p)}_{L^{2}(\Omega)}.
\end{equation}
Here, $\norm{.}_{1,\mcal{T}}$ is the discrete $H_{0}^{1}$ norm and $m(p)=\displaystyle\frac{1}{\abs{\Omega}}\int_{\Omega}p\dd\uu{x}$; see, e.g.\ \cite{GHL12,GHL+18,SS97}, for more details. 
\end{lemma}

\section{An Energy Stable Semi-implicit Upwind Scheme}
\label{sec:upwind_scheme}

In the following, we introduce our semi-implicit in time, upwind in space, fully-discrete scheme for the Euler system \eqref{eq:r_cons_mas}-\eqref{eq:r_cons_mom}.
\subsection{The scheme}
\label{sec:scheme}
Let us consider a discretisation $0=t^0<t^1<\cdots<t^N=T$ of the time-interval $(0,T)$ and let $\dt=t^{n+1}-t^n$, for $n=0,1,\dots,N-1$, be the constant time-step. We consider the following fully-discrete scheme for $0\leq n\leq{N-1}$:
\begin{gather}
    \frac{1}{\dt}\big(\rho_{K}^{n+1}-\rho_{K}^{n}\big)+\frac{1}{\left|K\right|}\sum_{\s\in\E(K)}F_{\sigma,K}(\rho^{n+1},\uu{v}^{n})=0, \ \forall K\in \M,\label{eq:dis_cons_mas}\\ 
    \frac{1}{\dt}\big(\rho_{\Ds}^{n+1}u_\s^{n+1}-\rho_{\Ds}^{n}u_\s^{n}\big)+\frac{1}{\left|\Ds\right|}\sum_{\epsilon\in\tilde{\E}(\Ds)}F_{\epsilon,\sigma}(\rho^{n+1},\uu{v}^{n})u_{\eps,\mathrm{up}}^{n}+\frac{1}{\veps^2}(\partial^{(i)}_{\E}p^{n+1})_{\s}=0, \ \forall \ 1\leq i\leq d, \ \forall \s\in\E_\intr^{(i)}. \label{eq:dis_cons_mom}
\end{gather}
The mass balance equation \eqref{eq:dis_cons_mas} is defined on primal
mesh cells, whereas the momentum balance \eqref{eq:dis_cons_mom} is on
the dual cells. To derive the energy stability, we recall the
following mass balance on the dual cell $\Ds$ \cite{GHL+18a}: 
\begin{equation}
    \label{eq:dis_cons_mass_dual}
    \frac{1}{\dt}(\rho_{\Ds}^{n+1}-\rho_{\Ds}^{n})+\frac{1}{\left|\Ds\right|}\sum_{\epsilon\in\tilde\E(\Ds)}F_{\epsilon,\sigma}(\rho^{n+1},\uu{v}^{n})=0. 
\end{equation}
The discrete momentum update \eqref{eq:dis_cons_mom} and the dual mass
balance \eqref{eq:dis_cons_mass_dual} together yield the following
update of the velocity components: 
\begin{equation}
    \label{eq:dis_vel_dual}
    \frac{u_\s^{n+1}-u_\s^n}{\dt}+\frac{1}{|\Ds|}\sum_{\epsilon\in\tilde\E(\Ds)}F_{\epsilon,\s}(\rho^{n+1},\uu{v}^{n})^-\frac{u_{\s^{\prime}}^{n}-u_\s^n}{\rho_{\Ds}^{n+1}}+\frac{1}{\veps^2}\frac{(\partial_{\E}^{(i)}p^{n+1})_{\s}}{\rho_{\Ds}^{n+1}}=0.
\end{equation}

\subsubsection{Initialisation of the Scheme}
\label{sec:initialise}
We take initial approximations for $\rho$ and $\uu{u}$ as the averages of the initial conditions $\rho_{0}$ and $\uu{u}_{0}$ on primal and dual cells, respectively, to get the initial values:
\begin{equation}
\label{eq:dis_ic}
\begin{aligned}
    \rho_{K}^{0}&=\frac{1}{|K|}\int_{K}\rho_{0}(\uu{x})\dd\uu{x}, \  \forall K\in\M,\\
    u_{\s,i}^{0}&=\frac{1}{|\Ds|}\int_{\Ds}(\uu{u}_{0}(\uu{x}))_{i}\dd\uu{x}, \ \text{for} \ 1\leq i\leq d, \ \forall \s\in\E_\intr^{(i)}.
\end{aligned}
\end{equation}

In the following lemma, we observe an immediate implication of the ill-prepared initial data from Definition~\ref{def:ill_prep_id} in the light of the above discretisation.
\begin{lemma}
\label{lem:dis_ill_prep_est}
Let the initial condition $(\rho^\veps_0, \uu{u}^\veps_0)$ be ill-prepared in the sense of Definition \ref{def:ill_prep_id}. Then there exists a constant $C>0$, independent of $\veps$, such that
\begin{equation}
\label{eq:dis_ill_prep_est}
    \frac{1}{\veps}\max_{K\in\mcal{M}}\abs{(\rho^\veps)^0_K - 1} \leq C.
\end{equation}
\end{lemma}
\begin{proof}
For each $K\in\mcal{M}$ and $\veps>0$, we have from \eqref{eq:dis_ic} and Definition~\ref{def:ill_prep_id} that there exists a constant $C>0$, independent of $\veps$, such that
\begin{equation}
    \frac{1}{\veps}\abs{(\rho^\veps)^0_K - 1}\leq \frac{1}{\veps}\frac{1}{\abs{K}}\int_{K}\abs{\rho^\veps_0 - 1}\dd\uu{x}\leq \frac{1}{\veps}\norm{\rho^\veps_0 - 1}_{L^\infty(K)}\leq C,
\end{equation}
which proves the required estimate \eqref{eq:dis_ill_prep_est}.
\end{proof}

\subsection{Existence of a Numerical Solution}
\label{sec:exist_soln}
The mass update \eqref{eq:dis_cons_mas} is nonlinear in $\rho^{n+1}$
due to the presence of the stabilisation terms. However, once
$\rho^{n+1}$ is calculated, the momentum update
\eqref{eq:dis_cons_mom} can be explicitly evaluated to get the
velocity. In what follows, we establish the existence of a discrete
solution to the numerical scheme
\eqref{eq:dis_cons_mas}-\eqref{eq:dis_cons_mom}. Our treatment is
analogous to the one in \cite{EGG+98,GMN19, NIR01, DYY06} which uses
the classical tools from topological degree theory in finite
dimensions to hyperbolic problems; see also \cite{Dei85}.  
\begin{theorem}
    \label{thm:existence}
    Let $(\rho^n,\uu{u}^n)\in
    L_{\mcal{M}}(\Omega)\times\uu{H}_{\mcal{E},0}(\Omega)$ be such
    that $\rho^{n}>0$ on $\Omega$. Then, there exists a solution
    $(\rho^{n+1},\uu{u}^{n+1})\in
    L_{\mcal{M}}(\Omega)\times\uu{H}_{\mcal{E},0}(\Omega)$ of
    \eqref{eq:dis_cons_mas}-\eqref{eq:dis_cons_mom}, satisfying
    $\rho^{n+1}>0$ on $\Omega$. 
\end{theorem}
\begin{proof}
Let $C>0$ be a constant such that 
\begin{equation}
\label{eq:C_thm_ex}
    C>\frac{\abs{\Omega}\max_{K\in\mcal{M}}\{\rho_{K}^n\}}{\min_{K\in\mcal{M}}\{\abs{K}\}}.
\end{equation}
Consider the bounded, open subset $V$ of $L_{\mcal{M}}(\Omega)$ defined by
\begin{equation*}
    V=\bigg\{\rho=\sum_{K\in\mcal{M}}\rho_{K}\mcal{X}_{K}\in L_{\mcal{M}}(\Omega)\colon 0<\rho_{K}<C, \ \forall K\in\mcal{M}\bigg\}.
\end{equation*}
We introduce a continuous function $H\colon[0,1]\times L_{\mcal{M}}(\Omega)\rightarrow L_{\mcal{M}}(\Omega)$, via 
$H(\lambda,\rho)=\sum_{K\in\mcal{M}}H(\lambda,\rho)_{K}\mcal{X}_{K}$, where
\begin{equation*}
    H(\lambda,\rho)_{K}=\frac{1}{\dt}(\rho_{K}-\rho^n_{K})+\frac{\lambda}{\abs{K}}\sum_{\s\in\mcal{E}(K)}F_{\s,K}(\rho,\uu{v}), \ \forall K \in \mcal{M},
\end{equation*}
with $v_\s = u_\s^n-\frac{\eta\dt}{\veps^2}(\D_\E^{(i)}p(\rho))_\s$. It is easy to note that $H(\lambda,\rho)$ is a homotopy connecting $H(0,\rho)$ and $H(1,\rho)$. In order to establish the existence of a discrete solution $\rho^{n+1}>0$ of \eqref{eq:dis_cons_mas}, we need to show that for any $\lambda\in [0,1]$, $H(\lambda,\cdot)$ has a non-zero topological degree with respect to $V$. In other words, we need to show that $H(\lambda,\cdot)\neq 0$ on $\D V$, for all $\lambda\in [0,1]$. On the contrary, let us assume that $H(\lambda,\rho)=0$, for some $\lambda\in[0,1]$ and $\rho\in\D V$, which implies
\begin{equation}
\label{eq:H_zero}
    \frac{1}{\dt}(\rho_{K}-\rho^n_{K})+\frac{\lambda}{|K|}\sum_{\s\in\mcal{E}(K)}F_{\s,K}(\rho,\uu{v})=0, \ \forall K \in \mcal{M}.
\end{equation}
Summing \eqref{eq:H_zero} over all $K\in\mcal{M}$, and using the definition of $C$ from \eqref{eq:C_thm_ex}, we obtain
\begin{equation}
\label{eq:rhoKub}
    \rho_K\leq \frac{\abs{\Omega}\max_{K\in\mcal{M}}\{\rho_{K}^n\}}{\min_{K\in\mcal{M}}\{\abs{K}\}}< C, \ \forall K\in\mcal{M}.
\end{equation}
Using \eqref{eq:H_zero}, we have for each $K\in\mcal{M}$ and any $\lambda\in[0,1]$,
\begin{equation}
    \label{eq:ex_pos}
    \rho_{K}\bigg[|\Omega|+\lambda\dt\sum_{\substack{\s\in\mcal{E}(K)\\\s=K|L}}|\s|\uu{v}_{\s,K}^+\Bigg]>0.  
\end{equation}
Combining the above inequality with \eqref{eq:rhoKub} gives $0<\rho_K<C$, for each $K\in\mcal{M}$, a contradiction as $\rho\in\D V$. Hence, $0\notin H(\{\lambda\}\times\partial V)$ for any $\lambda\in[0,1]$, and by topological degree arguments for finite dimensional spaces \cite{Dei85}, we get that $\deg(H(1,\cdot),V,0)=\deg(H(0,\cdot),V,0)$. Since $\deg(H(0,\cdot),V,0)\neq 0$ we further conclude that $H(1,\cdot)$ has a zero in $V$, i.e.\ there exists a solution $\rho^{n+1}\in L_{\mcal{M}}(\Omega)$ of \eqref{eq:dis_cons_mas}, such that $\rho^{n+1}>0$.
\end{proof}

\subsection{Discrete Energy Estimates}
\label{sec:id}
This section is devoted to proving apriori energy estimates satisfied
by the scheme \eqref{eq:dis_cons_mas}-\eqref{eq:dis_cons_mom} which
are discrete counterparts of the stability estimates stated in
Proposition~\ref{prop:r_engy_balance}.   
\begin{lemma}[Discrete positive renormalisation identity]
  Any solution to the system
  \eqref{eq:dis_cons_mas}-\eqref{eq:dis_cons_mom} satisfies the
  following equality: 
  \begin{equation}
    \label{eq:dis_porenorm}
    \begin{aligned}
      &\frac{|K|}{\dt}\big[\Pi_\gamma(\rho_{K}^{n+1})-\Pi_\gamma(\rho_{K}^{n})\big]+|K|p_{K}^{n+1}(\divM\uu{v}^n)_{K}\\
      &+\sum_{\substack{\s\in\mcal{E}(K)\\\s=K|L}}|\s|\Big[\big(\psi_\gamma(\rho_{K}^{n+1})-\rho_{K}^{n+1}\psi_\gamma^{\prime}(1)\big)(v_{\s,K}^n)^{+}+\big(\psi_\gamma(\rho_{L}^{n+1})-\rho_{L}^{n+1}\psi_\gamma^{\prime}(1)\big)(v_{\s,K}^{n})^{-}\Big]+R_{K,\dt}^{n+1}=0,
    \end{aligned}
  \end{equation}
  where the non-negative remainder term is defined by
  \begin{equation}
    \label{eq:RKn+1}
    R_{K,\dt}^{n+1}=\frac{|K|}{2\dt}\big(\rho_{K}^{n+1}-\rho_{K}^{n}\big)^2\psi_\gamma^{\prime\prime}\Big(\bar\rho_{K}^{n+\frac{1}{2}}\Big)+\sum_{\s\in\mcal{E}(K)}|\s|(-(v_{\s,K}^{n})^{-})(\rho_{L}^{n+1}-\rho_{K}^{n+1})^2\psi_\gamma^{\prime\prime}\big(\bar\rho_{\s}^{n+1}\big). 
  \end{equation}
\end{lemma}
\begin{proof}
  The proof follows from the definition of $\Pi_\gamma$ and
  straightforward calculations; see \cite[Lemma A.2]{HKL14} for
  details. 
\end{proof}
\begin{lemma}[Discrete kinetic energy identity]Any solution to the system \eqref{eq:dis_cons_mas}-\eqref{eq:dis_cons_mom} satisfies the following equality for $1\leq i\leq d,\; \s\in\E_\intr^{(i)}$ and $0\leq n\leq{N-1}$:
\begin{equation}
\label{eq:dis_kinbal}
\begin{aligned}
 \frac{1}{2}\frac{|\Ds|}{\dt}\big(\rho_{\Ds}^{n+1}(u_\s^{n+1})^2-\rho_{\Ds}^{n}(u_\s^{n})^2\big)+\sum_{\epsilon\in\tilde\E(\Ds)}F_{\eps,\s}(\rho^{n+1},\uu{v}^{n})\frac{|u_\eps^{n}|^2}{2}+\frac{1}{\veps^2}|\Ds|v_\s^{n}\big(\D^{(i)}_\E p^{n+1}\big)_{\sigma}+R_{\s,\dt}^{n+1}\\
 =-\frac{\eta\dt}{\veps^4}|\Ds|\big(\D^{(i)}_\E p^{n+1}\big)_\s^2,
\end{aligned}
\end{equation}
\end{lemma}
where the remainder term $R_{\s,\dt}^{n+1}$ is defined by
\begin{equation}
    \label{eq:Rsn+1}
    R_{\s,\dt}^{n+1}=
    -\frac{|\Ds|}{2\dt}\rho_{\Ds}^{n+1}(u_\s^{n+1}-u_\s^n)^2-\sum_{\substack{\eps\in\tilde\E(\Ds)\\\eps=\Ds|D_{\s^{\prime}}}}F_{\eps,\s}(\rho^{n+1},\uu{v}^{n})^{-}\frac{(u_\s^n-u_{\s^{\prime}}^n)^2}{2}.
\end{equation}
\begin{proof}
Multiplying the momentum balance equation \eqref{eq:dis_cons_mom} by $|\Ds|u_\s^n$ and using the dual mass balance \eqref{eq:dis_cons_mass_dual} completes the proof.
\end{proof}
\begin{theorem}[Total energy balance]
\label{lem:dis_totbal}
Any solution to the system \eqref{eq:dis_cons_mas}-\eqref{eq:dis_cons_mom} satisfies the following entropy inequality: 
\begin{equation}
\label{eq:dis_totbal}
  \frac{1}{\veps^2}\sum_{K\in\M}|K|\Pi_\gamma(\rho_{K}^{n+1})+\sum_{\s\in\E_\intr}|\Ds|\frac{1}{2}\rho_{\Ds}^{n+1}(u_{\s}^{n+1})^2 \leq 
  \frac{1}{\veps^2}\sum_{K\in\M}|K|\Pi_\gamma(\rho_{K}^{n})+\sum_{\s\in\E_\intr}|\Ds|\frac{1}{2}\rho_{\Ds}^{n}(u_{\s}^{n})^2,
\end{equation}
under the conditions $\forall\s \in\E_\intr^{(i)}, \ 1 \leq i \leq d$:

\begin{enumerate}[label=(\roman*)]
\item $\displaystyle \eta \geq \frac{1}{\rho_{\Ds}^{n+1}}$, 
\item $\displaystyle \frac{\dt}{|\Ds|}\sum_{\eps\in\tilde\E(\Ds)}\frac{-F_{\eps,\s}(\rho^{n+1},\uu{v}^n)^-}{\rho_{\Ds}^{n+1}}\leq\frac{1}{2}$.
\end{enumerate}
\end{theorem}
\begin{proof}
We take the sum over $K\in\M$ in \eqref{eq:dis_porenorm} and over $\s\in\E_\intr$ in \eqref{eq:dis_kinbal}, and add the resulting equations to get
\begin{equation}
    \label{eq:thm_TE_eq3}
    \begin{aligned}
    \frac{1}{\veps^2\dt}\sum_{K\in\M}|K|\Big(\Pi_\gamma(\rho_{K}^{n+1})-\Pi_\gamma(\rho_{K}^{n})\Big)+\sum_{\s\in\E_\intr}\frac{1}{2}\frac{|\Ds|}{\dt}\Big(\rho_{\Ds}^{n+1}(u_\s^{n+1})^2-\rho_{\Ds}^{n}(u_\s^n)^2\Big)\\
    \quad+\mcal{R}_{\M,\dt}^{n+1}+\mcal{R}_{\E,\dt}^{n+1}=-\frac{\eta\dt}{\veps^4}\sum_{i=1}^d\sum_{\s\in\E_\intr^{(i)}}|\Ds|\big(\D^{(i)}_\E p^{n+1}\big)_\s^2,
    \end{aligned}
\end{equation}
where the global remainder terms $\mcal{R}_\M$ and $\mcal{R}_\E$ are
obtained by summing the local remainders $R_K$ and $R_\s$,
respectively. Clearly, $\mcal{R}_\M$ is non-negative
unconditionally. In order to derive a condition to make $\mcal{R}_\E$
non-negative, we apply the inequality $(a+b)^2\leq 2(a^2 + b^2), \;a,b\in\mbb{R}$ in the velocity update \eqref{eq:dis_vel_dual} and the
Cauchy-Schwarz inequality on the second term of $\mcal{R}_{\E}$
analogously as in \cite[Lemma 3.1]{DVB17} and finally obtain   
\begin{equation}
\label{eq:dis_totenbal_press}
\begin{aligned}
   &\frac{1}{\veps^2\dt}\sum_{K\in\M}|K|\Big(\Pi_\gamma(\rho_{K}^{n+1})-\Pi_\gamma(\rho_{K}^{n})\Big)+\sum_{\s\in\E_\intr}\frac{1}{2}\frac{|\Ds|}{\dt}\Big(\rho_{\Ds}^{n+1}(u_\s^{n+1})^2-\rho_{\Ds}^{n}(u_\s^n)^2\Big)\\
    &+\frac{\dt}{\veps^4}\sum_{i=1}^d\sum_{\s\in\tilde\E_\intr^{(i)}}\bigg(\eta- \frac{1}{\rho_{\Ds}^{n+1}}\bigg)\big(\D_\E^{(i)}p^{n+1}\big)_\s^2\\
   &\leq\sum_{\s\in\E_\intr}\left(\frac{1}{2}-\frac{\dt}{|\Ds|}\sum_{\epsilon\in\E(\Ds)}\frac{-(F_{\epsilon,\s}(\rho^{n+1},\uu{v}^{n}))^{-}}{\rho_{\Ds}^{n+1}}\right)\left(\sum_{\substack{\epsilon\in\bar\E(\Ds)\\\epsilon=\Ds|D_{\s^{\prime}}}}F_{\epsilon,\sigma}(\rho^{n+1},\uu{v}^{n})^-(u_\s^{n}-u_{\s^\prime}^{n})^2\right).
\end{aligned}
\end{equation}
Note that the right hand side is non-positive under the time step restriction
\begin{equation*}
\displaystyle
\frac{\dt}{|\Ds|}\sum_{\eps\in\tilde\E(\Ds)}\frac{-F_{\eps,\s}(\rho^{n+1},\uu{v}^n)^-}{\rho_{\Ds}^{n+1}}\leq\frac{1}{2}, 
\end{equation*}
and the third term on the left hand side is non-negative under the
condition $\eta \geq \frac{1}{\rho^{n+1}_{\Ds}}$. 
\end{proof}
\begin{remark}
  The conditions (i) and (ii) in the above theorem are implicit in
  nature. However, we observe that the following implicit time-step
  restriction 
  \begin{equation}
    \dfrac{\dt}{|\Ds|}\sum_{\eps\in\tilde\E(\Ds)}
    \frac{\abs{F_{\eps,\s}(\rho^{n+1},\uu{v}^n)}}{\rho_{\Ds}^{n+1}}\leq\frac{1}{2}
    \label{eq:suftime}  
  \end{equation}
  gives a sufficient condition from which we can deduce the condition
  (ii). Now, from \eqref{eq:suftime} and the dual mass balance
  \eqref{eq:dis_cons_mass_dual}, we have  
  \begin{equation}
    \frac{3}{2}\rho_{\Ds}^{n+1} - \rho_{\Ds}^n \geq  \rho_{\Ds}^{n+1}
    - \rho_{\Ds}^n +
    \frac{\dt}{|\Ds|}\sum_{\eps\in\tilde\E(\Ds)}\abs{F_{\eps,\s}(\rho^{n+1},\uu{v}^n)}\geq
    0. 
  \end{equation}
  Hence, we get
  \begin{equation}
    \label{eq:rho_3/2}
    \frac{\rho_{\Ds}^n }{\rho_{\Ds}^{n+1}}\leq \frac{3}{2}.
  \end{equation}
  Therefore, at each interface $\sigma$, choosing $\eta$ such that
  $\eta=\eta_1/\rho_{\Ds}^n$ with $\eta_1>\frac{3}{2}$ will guarantee
  the condition (i) and hence the stability of the scheme. In other
  words, the value of $\eta$ can be obtained explicitly. Analogous
  considerations can also be found in \cite{DVB20}.   
\end{remark}

\section{Weak Consistency of the Scheme}
\label{sec:weak_cons}

Goal of this section is to show a Lax-Wendroff-type weak consistency
of the scheme \eqref{eq:dis_cons_mas}-\eqref{eq:dis_cons_mom} which is
essentially proving the consistency of the numerical solution with a
weak solution of the Euler system when the mesh parameters tend to
zero. To this end, we have borrowed the tools developed in
\cite{GHL19, GHL22} to prove weak consistency of finite volume
discretisations of general convective operators. 
\begin{definition}
\label{def:weak_soln}
$(\rho^\veps,\uu{u}^\veps)\in L^\infty(Q)^{1+d}$ is a weak solution to
the Euler system \eqref{eq:cons_mas}-\eqref{eq:cons_mom} with
initial-boundary conditions \eqref{eq:eq_ic} if $\rho^\veps>0$ a.e.\
in $Q$ and the following identities hold:   
\begin{gather}
    \int_0^T\int_\Omega\left(\rho^\veps\D_t\psi+\rho^\veps \uu{u}^\veps\cdot\grd_x \psi\right)\dd \uu{x}\dd t
     =-\int_\Omega\rho^\veps_0\psi(0,\cdot)\dd \uu{x} \label{eq:weak_soln_mas} ,\ \text{for all } \psi \in C_c^\infty([0,T)\times\bar{\Omega}),\\
     \int_0^T\int_\Omega\left(\rho^\veps \uu{u}^\veps\cdot\D_t\uu{\psi}+(\rho^\veps \uu{u}^\veps\otimes\uu{u}^\veps):\grd_x \uu{\psi}\right)\dd \uu{x}\dd t+\int_0^T\int_{\Omega}p^\veps\grd_x\uu{\psi} \dd \uu{x}\dd t= -\int_{\Omega}\rho^\veps_0 \uu{u}^\veps_0\cdot\uu{\psi}(0,\cdot)\dd \uu{x},\label{eq:weak_soln_mom}\\ 
     \text{for all } \uu{\psi} \in C_c^\infty([0,T)\times\bar{\Omega})^d.\nonumber
    \end{gather}
\end{definition}
 
The following theorem is a Lax-Wendroff-type consistency formulation
for the semi-implicit scheme
\eqref{eq:dis_cons_mas}-\eqref{eq:dis_cons_mom}. 
\begin{theorem}
\label{thm:weak_cons}
Let $\Omega$ be an open bounded subset of $\mbb{R}^d$. Assume that $\big(\mcal{T}^{(m)},\delta t^{(m)}\big)_{m\in\mbb{N}}$ is a sequence of discretisations such that both $\lim_{m\rightarrow +\infty}\delta t^{(m)}$ and $\lim_{m\rightarrow +\infty}h^{(m)}$ are $0$. Let $\big(\rho^{(m)},\uu{u}^{(m)})_{m\in\mbb{N}}$ be the corresponding sequence of discrete solutions with respect to an initial data $(\rho^\veps_0,\uu{u}^\veps_0)\in L^\infty(\Omega)^{1+d}$. We assume that $(\rho^{(m)},\uu{u}^{(m)})_{m\in\mbb{N}}$ satisfies the following:

$\big(\rho^{(m)},u^{(m)}\big)_{m\in\mbb{N}}$ is uniformly bounded in
$L^\infty(Q)^{1+d}$, i.e.\  
\begin{align}
\underbar{C}\leq(\rho^{(m)})^n_K &\leq \overline{C}, \ \forall K\in\mcal{M}^{(m)}, \ 0\leq n\leq N^{(m)}, \ \forall m\in\mbb{N}\label{eq:dens_abs_bound}, \\
|(u^{(m)})^n_\sigma| &\leq C, \ \forall \sigma\in\mcal{E}^{(m)}, \ 0\leq n\leq N^{(m)}, \ \forall m\in\mbb{N}\label{eq:u_abs_bound}. 
\end{align}
where $\underbar{C}, \overline{C}$ and $C$ are positive constants independent of the discretisations. We also assume that there exists $\theta>0$ such that the sequence of discretisations $\big(\mcal{T}^{(m)},\delta t^{(m)}\big)_{m\in\mbb{N}}$ satisfies the condition:
\begin{equation}
\frac{\delta t^{(m)}}{\min_{K\in\mcal{M}^{(m)}}\abs{K}}\leq\theta,\ \max_{K \in \mcal{M}^{(m)}} \frac{\text{diam}(K)^2}{|K|}\leq\theta, \; \forall m\in\mbb{N}.
\label{eq:CFL_restr}
\end{equation}
We suppose that the sequence $(\rho^{(m)}, \uu{u}^{(m)})_{m\in\mbb{N}}$ converges to $(\rho_\veps, \uu{u}_\veps)\in L^\infty(Q)^{1+d}$ as $m\rightarrow \infty$ in $L^r(Q)^{1+d}$ for $1\leq r<\infty$. Then $(\rho_\veps, \uu{u}_\veps)$ satisfies the weak formulation \eqref{eq:weak_soln_mas}-\eqref{eq:weak_soln_mom}.
\end{theorem}

\begin{proof}
%
Our approach follows analogous lines as in \cite[Lemma 4.1]{GHL22} and hence we skip 
most of the calculations except the ones related to the velocity stabilisation term.
Proceeding as in \cite[Lemma 4.1]{GHL22}, the additional term $R^{(m)}$ arising from the mass update \eqref{eq:dis_cons_mas} can be obtained as
\begin{equation}
R^{(m)} =
C_\psi\sum_{n=0}^{N^{(m)}-1}\dt^n\sum_{K\in\M}\diam(K)\sum_{\s\in\E(K)}\abs{\s}\abs{\delta
  u_{\s}(\rho^{n+1})}, 
\label{eq:trunc_error_stab}
\end{equation}
where $C_\psi$ is a constant depending only on the test function $\psi \in C_c^\infty([0,T)\times\bar{\Omega})$. Using \eqref{eq:dens_abs_bound}, $R^{(m)}$ can be further estimated as
\begin{equation}
R^{(m)} \leq \gamma \overline{C}^{\gamma -
  1}C_\psi\sum_{n=0}^{N^{(m)}-1}\dt^n\sum_{K\in\M}\diam(K)\sum_{\s\in\E(K)}|\s|\eta\dt^n\frac{|\s|}{|\Ds|}\abs{\rho^{n+1}_L
  - \rho^{n+1}_K}. 
\label{eq:trunc_error_stab_est}
\end{equation}
Using the tools introduced in \cite[Section 4]{GHL19} to study the
convergence of discrete space translates, it can be shown that the
right hand side of \eqref{eq:trunc_error_stab_est} tends to $0$ as
$m\rightarrow \infty$ under the assumptions \eqref{eq:dens_abs_bound}
and \eqref{eq:CFL_restr}; see also \cite[Lemma A.1]{GHL22} for a
similar treatment on discrete space translates. 
The consistency of the discrete convection operator and the pressure
gradient in the momentum update \eqref{eq:dis_cons_mom} can be
obtained under the given assumptions using a similar method as in the
proof of \cite[Theorem 4.1]{HLN+23}. The residual term appearing due
to the velocity stabilisation in the momentum update again converges
to $0$ by a similar argument used in the case of the mass update. 
\end{proof}
Adopting similar techniques as in the proof of \cite[Theorem 5.1]{HLN+23} yields the following LW weak-entropy consistency of the velocity stabilised scheme:
\begin{theorem}
\label{thm:entr_cons}
Suppose that all the assumptions \eqref{eq:dens_abs_bound}-\eqref{eq:CFL_restr} hold true for the sequence of discrete solutions $\big(\rho^{(m)},\uu{u}^{(m)})_{m\in\mbb{N}}$ which converge to $(\rho^\veps, \uu{u}^\veps)\in L^\infty(Q)^{1+d}$ in $L^{r}(Q)^{1+d}$ for $1\leq r <\infty$. In addition, assume the following discrete BV estimates:
\begin{align}
	\sum_{n=0}^{N^{(m)} - 1}\sum_{K\in\M^{(m)}}\abs{K}\abs{(\rho^{(m)})^{n+1}_K - (\rho^{(m)})^n_K} &\leq C,\; \forall m\in \mbb{N}, \\
	\sum_{n=0}^{N^{(m)} - 1}\sum_{\s\in\E^{(m)}}\abs{\Ds}\abs{(u^{(m)})^{n+1}_\sigma - (u^{(m)})^n_\sigma} &\leq C,\; \forall m\in\mbb{N} 
\end{align}
where $C$ is a positive constant independent of the mesh parameters. Also assume that the space and time discretisations satisfy the following:
\begin{equation}
	\lim_{m\rightarrow \infty}\frac{\dt^{(m)}}{\min_{\s\in\E^{(m)}}\abs{\sigma}} = 0.
\end{equation} 
Then, the limiting terms $(\rho^\veps, \uu{u}^\veps)\in L^\infty(Q)^{1+d}$ satisfy the entropy condition:
\begin{equation}
    -\int_0^T\int_\Omega\bigg[\zeta^\veps \D_t\psi + \Big(\zeta^\veps + \frac{1}{\veps^2} p^\veps\Big)\uu{u}^\veps\cdot\bgrd \psi\bigg]\dd \uu{x} \dd t\leq \int_\Omega\zeta^\veps_0 \psi(0,x)\dd \uu{x},\; \forall \psi\in C_c^\infty([0,T)\times\bar{\Omega}),\; \psi(T)=0,\;\psi\geq 0,
\end{equation}
where $\zeta^\veps=\frac{1}{\veps^2}\Pi_\gamma(\rho^\veps) + \half\rho^\veps\absq{\uu{u}^\veps}$ and $\zeta^\veps_0 = \frac{1}{\veps^2}\Pi_\gamma(\rho^\veps_0) + \half\rho^\veps_0\absq{\uu{u}^\veps_0}$.
\end{theorem}
\section{Consistency of the Scheme with the Zero Mach Number Limit} 
\label{sec:cons_incomp}

In the present section, we derive the zero Mach number limit of the
semi-implicit scheme \eqref{eq:dis_cons_mas}-\eqref{eq:dis_cons_mom}
as done in the continuous case in Section~\ref{sec:cont-case}. We show
that the $\veps\to0$ limit of the scheme
\eqref{eq:dis_cons_mas}-\eqref{eq:dis_cons_mom} is a semi-implicit
scheme for the velocity stabilised incompressible Euler system
\eqref{eq:stab_incomp_eul_divfree}-\eqref{eq:stab_incomp_eul_pres}. In
order to establish this asymptotic consistency result, we use the
discrete energy estimates obtained from \eqref{eq:dis_totbal} and
prove the convergence of the numerical solution in the discrete
function spaces. We start by proving the following Lemma which gives
the entropy inequality for the scheme
\eqref{eq:dis_cons_mas}-\eqref{eq:dis_cons_mom}.   
\begin{lemma}
\label{lem:dis_ent}
Suppose that the initial data $(\rho_{0}^{\veps},\uu{u}_{0}^{\veps})$ satisfy the ill-prepared condition \eqref{def:ill_prep_id}. Then, there exists a constant $C>0$, independent of $\veps$, such that the numerical solution $(\rho^n,\uu{u}^n)_{0\leq n\leq N}$ satisfies the entropy inequality:
\begin{equation}
\label{eq:bdd_terms}
    \frac{1}{2}\sum_{\s\in\mcal{E}_\intr}|\Ds|\rho_{\Ds}^{n}|\uu{u}_{\s}^{n}|^2+\frac{1}{\veps^2}\sum_{K\in\mcal{M}}|K|\Pi_{\gm}(\rho_{K}^{n})+\frac{(\dt)^2}{\veps^4}\sum_{k=0}^{n-1}\sum_{\s\in\mcal{E}_\intr}\left(\eta-\frac{1}{\rho_{\Ds}^{k+1}}\right)|(\grd_{\mcal{E}}p^{k+1})_{\s}|^2\leq C.
\end{equation}
\end{lemma}
\begin{proof}
First, we multiply the total energy balance
\eqref{eq:dis_totenbal_press} by $\dt$ and sum over times $k$ from $0$
to $n-1$. Using the estimate \eqref{eq:dis_ill_prep_est} coming 
from the ill-prepared data \eqref{def:ill_prep_id} and the bounds
on $\Pi_\gamma$ \cite[Lemma 2.3]{HLS21} we can obtain the required
result.   
\end{proof}
In the following, we establish claims similar to the ones in
\cite[Proposition 2.6]{HLS21} for the discrete solution.   
\begin{lemma}
\label{lem:disc_conv_bdd}
Let $\mcal{T}=(\mcal{M},\mcal{E})$ be a fixed primal-dual mesh pair that gives a MAC discretisation of $\bar{\Omega}$ and let $\delta t>0$ be such that $\{0=t^0<t^1<\cdots<t^N=T\}$ is a discretisation of the time domain $[0,T]$ with  $t^n = n\dt$, for $1\leq n\leq N = \lfloor\frac{T}{\delta t}\rfloor$. For each $\veps>0$, let us denote by $(\rho^\veps,\uu{u}^\veps)\in L^\infty(0,T;L_{\mcal{M}}(\Omega)\times\uu{H}_{\mcal{E},0}(\Omega))$, the discrete solution to the semi-implicit scheme \eqref{eq:dis_cons_mas}-\eqref{eq:dis_cons_mom} with respect to the  ill-prepared initial data discretised via \eqref{eq:dis_ic}. Then, the following holds.
\begin{enumerate}[label=(\roman*)]
\item For $\gamma\geq 2$, we have for all $\veps>0$ sufficiently small
\begin{equation}
\label{eq:disc_dens_est_gamma_big}
\frac{1}{\veps}\norm{\rho^{\veps}-1}_{L^\infty(0,T;L^2(\Omega))}\leq C(\gamma).
\end{equation}
Also for $\gamma\in(1,2)$, we have for all $\veps>0$ sufficiently small, and for any $R\in(2,+\infty)$
\begin{align}
\frac{1}{\veps}\norm{(\rho^{\veps}-1)\mcal{X}_{\{\rho^{\veps}<R\}}}_{L^\infty(0,T;L^2(\Omega))}&\leq C(\gamma,R),\label{eq:disc_dens_est_gamma_ess} \\
{\veps}^{-\frac{2}{\gamma}}\norm{(\rho^{\veps}-1)\mcal{X}_{\{\rho^\veps\geq R\}}}_{L^\infty(0,T;L^\gamma(\Omega))}&\leq C(\gamma,R).\label{eq:disc_dens_est_gamma_res}
\end{align}
Here $C(\gamma)$ and $C(\gamma,R)$ are positive constants independent of $\veps$.
\item $\rho^\veps\rightarrow 1$ as $\veps\rightarrow 0$ in $L^\infty(0,T;L^r(\Omega))$ for any $r\in(1,\min\{2,\gamma\}]$. Moreover, for any $\gamma>1$ $\rho^\veps\rightarrow 1$ as $\veps\rightarrow 0$ in $L^\infty(0,T;L^\gamma(\Omega))$.
\item The sequence of approximate solutions for the velocity component $(\uu{u}^\veps)_{\veps>0}$ is uniformly bounded with respect to $\veps$, i.e.\ for all $\veps>0$ sufficiently small we have
\begin{equation}
\label{eq:disc_vel_bound}
\norm{\uu{u}^\veps}_{L^\infty(0,T;L^2(\Omega)^d)}\leq C,
\end{equation}
where $C>0$ is a constant independent of $\veps$.
\end{enumerate}
\end{lemma}
\begin{proof}
The estimates for the point-values of the weak solutions derived in \cite[Lemma 2.3]{HLS21} make use of the bounds on the relative entropy functional $\Pi_\gamma$ which leads to the convergence results \cite[Proposition 2.6]{HLS21}. In an analogous manner, we can adopt the same results for the discrete solutions of the present scheme and use the estimate \eqref{eq:bdd_terms} to obtain the above convergence results for $\rho^\veps$.
Next, we proceed to derive the estimate \eqref{eq:disc_vel_bound} on
the velocity. To this end, we first note that for each
$n\in\{1,\dots,N\}$, and for each $K\in\mcal{M}$, we have
$\rho^n_K\rightarrow 1$ as $\veps\rightarrow 0$. Since there exists
$0<\rho_{\min}<1$, we can further conclude that for all $\veps>0$
sufficiently small, 
\begin{equation}
\label{eq:disc_dens_lb}
\rho^n_K>\rho_{\mathrm{min}}, \ \forall n\in\{0,\dots,N\}, \; \forall K\in \M.
\end{equation}
From the discrete entropy estimate \eqref{eq:bdd_terms},  the bounds
on $\Pi_\gamma$ \cite[Lemma 2.3]{HLS21} and the estimate
\eqref{eq:dis_ill_prep_est} on ill-prepared initial data we get
\begin{equation}
\label{eq:disc_vel_bound_a}
\sum_{i=1}^d\sum_{\sigma\in\mcal{E}^{(i)}_\intr}\rho^{n+1}_{D_\sigma}\absq{u^{n+1}_\sigma}\leq C, \ \forall n\in\{0,\dots,N-1\}.
\end{equation}
From \eqref{eq:disc_dens_lb} and \eqref{eq:disc_vel_bound_a} we
subsequently obtain the required estimate \eqref{eq:disc_vel_bound}.  
\end{proof}

It is well known that an asymptotic analysis of the Euler equations
reveals a pressure decomposition $p=p_{(0)}+\veps^2\pi$ in the low
Mach number regime, where the second order pressure $\pi$ survives as
the incompressible pressure in the limit $\veps\to0$. In the discrete
counterpart, in order that a scheme to be consistent with the
incompressible limit, it is therefore essential for the numerical
pressure to have the same decomposition and its second order to
survive in the limit $\veps\to0$. In the following lemma, we give an
estimate on the second order pressure, using the uniform boundedness
of the pressure gradient $\bgrd_\E p^n$ appearing in the discrete
entropy estimate \eqref{eq:bdd_terms}. The following estimate is a
consequence of the inf-sup stability of the MAC discretisation, cf.\
Lemma~\ref{lem:inf_sup}. 
\begin{lemma}[Bound on the second order pressure]
\label{lem:bdd_press}
Let $\veps>0$ and $(\rho^{n},\uu{u}^{n})_{0\leq n\leq N}$ be the solution to the scheme \eqref{eq:dis_cons_mas}-\eqref{eq:dis_cons_mom}. Define $\pi^n=\sum_{K\in\mcal{M}}\pi^n_K\mcal{X}_{K}$, where $\pi^n_K=\frac{p^{n}_K-m(p^n)}{\veps^2}$ with $m(p^n)$ being the mean value of $p^n$ over $\Omega$. Then, there exists a constant $C_{\mcal{T},\dt}$, independent of $\veps$, such that for $0\leq n\leq N$
\begin{equation}
\label{eq:bdd_press}
    \norm{\pi^n}\leq C_{\mcal{T},\dt},
\end{equation}
where $\norm{\cdot}$ is any norm on the discrete function space.
\end{lemma}
\begin{proof}
For a proof we refer to \cite[Lemma 7.9]{HLS21}, wherein the estimate
is obtained for a staggered discretisation. Using the uniform
boundedness of the discrete pressure gradient obtained from the
estimate \eqref{eq:bdd_terms}, a similar procedure can be adapted for
the MAC discretisation to yield \eqref{eq:bdd_press}. 
\end{proof}
As a consequence of the convergence $\rho^\veps\to1$, the uniform
bounds obtained in Lemma~\ref{lem:disc_conv_bdd}, and the bound on the
second order pressure obtained in Lemma~\ref{lem:bdd_press}, we deduce
the following theorem which gives a semi-implicit scheme for the
incompressible Euler system with velocity stabilisation as the limit
of the semi-implicit scheme
\eqref{eq:dis_cons_mas}-\eqref{eq:dis_cons_mom} when $\veps\to0$. 
\begin{theorem}
Let $(\veps^{(k)})_{k\in \mathbb{N}}$ be a sequence of positive
numbers converging to zero,
$(\rho^{(k)},\uu{u}^{(k)})_{k\in\mathbb{N}}$ be the corresponding
sequence of numerical solutions obtained from the scheme
\eqref{eq:dis_cons_mas}-\eqref{eq:dis_cons_mom}, and let the initial
data $(\rho^{(k)}_{0},\uu{u}^{(k)}_{0})$ satisfy
\eqref{def:ill_prep_id}. Then $(\rho^{(k)})_{k\in\mathbb{N}}$
converges to $1$ in $L^{\infty}(0,T;L^{\gm}(\Omega))$ and
$(\uu{u}^{(k)},\pi^{(k)})_{k\in\mathbb{N}}$ converges to
$(\uu{u},\pi)\in L^\infty(0,T; \uu{H}_{\E,0}(\Omega)\times
L_\M(\Omega))$ in any discrete norm when $k$ tends to $\infty$, where
the sequence $(\uu{u}^n,\pi^n)$ is defined as follows. Given
$(\uu{u}^n, \pi^n)\in \uu{H}_{\E,0}(\Omega)\times L_\M(\Omega)$ at
time $t^n$, $(\uu{u}^{n+1}, \pi^{n+1})\in\uu{H}_{\E,0}(\Omega)\times
L_\M(\Omega)$ is obtained as the solution of the following
semi-implicit scheme: 
\begin{gather}
    (\dive_\M(\uu{u}^n-\delta\uu{u}(\pi^{n+1})))_{K}=0,\; \forall K \in \mcal{M}, \label{eq:dis_incomp_mas}\\
    \frac{1}{\dt}\big(u_\s^{n+1}-u_\s^{n}\big)+\frac{1}{\left|\Ds\right|}\sum_{\epsilon\in\tilde{\E}(\Ds)}F_{\epsilon,\sigma}(1,\uu{u}^{n}-\delta\uu{u}(\pi^{n+1}))u_{\eps,\mathrm{up}}^{n}+(\partial^{(i)}_{\E}\pi^{n+1})_{\s}=0, \ 1\leq i\leq d, \ \forall \s\in\E_\intr^{(i)}, \label{eq:dis_incomp_mom}
\end{gather}
where the correction $(\delta u(\pi^{n+1}))_\s=\eta\dt(\D_\E^{(i)}\pi^{n+1})_\s$ for $\s\in\Eint^{(i)}$, $i=1,2,\dots d$.
\end{theorem}
\begin{proof}
Using the second term on the left hand side of \eqref{eq:bdd_terms} we
can easily show that $\rho^{(k)}\rightarrow 1$ as $k\rightarrow\infty$
using estimates on $\Pi_{\gm}$ from Lemma~\ref{lem:disc_conv_bdd}; see
also \cite{HLS21}. From the first term in \eqref{eq:bdd_terms} we get
the estimate
$\norm{\uu{u}^{(k)}}_{L^{2}(\Omega)^{d}}\leq\sqrt{\frac{2C}{\rho_{\min}}}$
for all $k\in\mathbb{N}$. Noting that $\pi^{(k)}$ bounded, cf.\
\eqref{eq:bdd_press}, we conclude that
$\norm{\frac{1}{(\veps^{(k)})^2}\bgrd_\E p^{(k)}}$ is bounded and
therefore admits a weak limit. Hence, there exists a subsequence of
$(\uu{u}^{(k)},\pi^{(k)})_{k\in\mathbb{N}}$ which tends, in any
discrete norm, to a limit $(\uu{u},\pi)$ which satisfies
\eqref{eq:dis_incomp_mas}-\eqref{eq:dis_incomp_mom}. 
\end{proof}

\section{Numerical Results}
\label{sec:num_res}

In this section, we report the results of extensive numerical tests
performed with the semi-implicit scheme
\eqref{eq:dis_cons_mas}-\eqref{eq:dis_cons_mom}. Note that the
stability analysis performed in Section~\ref{sec:upwind_scheme}, cf.\
Theorem~\ref{lem:dis_totbal}, requires that the timesteps $\dt$ be
chosen according to the condition: 
\begin{equation}
\label{eq:timestep_res}
     \frac{\dt}{|\Ds|}\sum_{\eps\in\tilde\E(\Ds)}\frac{-(F_{\eps,\s}(\rho^{n+1},\uu{v}^n))^-}{\rho_{\Ds}^{n+1}}\leq\frac{1}{2},\;\forall
     \s\in\E^{(i)},\; i=1,\dots,d. 
\end{equation}
Since the above stability condition is implicit and difficult to carry
out, along the lines of \cite{CDV17, DVB17,DVB20}, we derive a
sufficient condition which is easy to implement in practice.  
\begin{proposition}
Suppose $\dt>0$ be such that for each $\s\in\E^{(i)},\;
i\in\{1,\dots,d\},\; \s = K|L$, the following holds: 
\begin{equation}
\label{eq:timestep_suff}
    \dt\max\Bigg\{\frac{\abs{\D K}}{\abs{K}},\frac{\abs{\D L}}{\abs{L}}\Bigg\}\Bigg(\abs{u^n_\s} + \sqrt{\frac{\eta}{\veps^2}\left|p^{n+1}_L - p^{n+1}_K\right|}\Bigg)\leq \min\Big\{1,\; \frac{1}{3}\mu^{n,n+1}_{K,L}\Big\},
\end{equation}
where $\abs{\D K}=\sum_{\s\in\E(K)}\abs{\s}$ and $\mu^{n,n+1}_{K,L} = \displaystyle\frac{\min\{\rho^n_K,\rho^n_L\}}{\max\{\rho^{n+1}_K,\rho^{n+1}_L\}}$. Then $\dt$ satisfies the inequality \eqref{eq:timestep_res}.
\end{proposition}
\begin{proof}
The proof follows the same lines of \cite[Proposition 3.2]{CDV17},
where a similar result has been obtained for an explicit scheme; see
also  \cite{DVB17,DVB20} for analogous treatments. 
\end{proof}
In all the numerical case studies performed below, the above condition 
\eqref{eq:timestep_suff} is implemented using the solution at $t^n$ in 
an explicit manner as in \cite{PV16}. The time-step restriction
\eqref{eq:timestep_suff} gives a much less restrictive $\CFL$
condition compared to the classical explicit one, e.g.\ in 1D:  
\begin{equation}
\label{eq:CFL}
\CFL=\max\Big(\abs{u^n}+\frac{c^n}{\veps}\Big)\frac{\dt}{\abs{K}},
\end{equation}
where $c^n=\sqrt{\gamma p^n/\rho^n}$. In view of the boundedness of
the second order pressure $\pi$, cf.\ Lemma~\ref{lem:bdd_press}, we
note that $\dt$ permitted by \eqref{eq:timestep_suff} is
$\mcal{O}(1)$, whereas the classical explicit time-step is
$\mcal{O}(\veps)$ as $\veps\to0$. Hence, the semi-implicit scheme
\eqref{eq:dis_cons_mas}-\eqref{eq:dis_cons_mom} admits large
time-steps at low Mach numbers. In addition, the time-steps remain
bounded as $\veps\to0$ which is crucial for the AP property of the
scheme. The gain in $\CFL$ condition in low Mach number computations
is recorded in Subsection~\ref{sec:1d_riemann}.    

The numerical implementation of the scheme
\eqref{eq:dis_cons_mas}-\eqref{eq:dis_cons_mom} is done as
follows. First, the mass conservation equation \eqref{eq:dis_cons_mas}
is solved to get the updated density $\rho^{n+1}$. A Newton iteration
has to be performed due to the presence of nonlinear stabilisation
terms. In our experiments, we note that iterations converge in 2-3
steps. Once $\rho^{n+1}$ is calculated, the momentum update
\eqref{eq:dis_cons_mom} is evaluated explicitly to get the velocity
$\uu{u}^{n+1}$.   

\subsection{1D Riemann Problems}
\label{sec:1d_riemann}
We consider the following initial data from \cite{DT11} which consists
of several Riemann problems: 
\begin{align*}
    \rho(0,x) &= 1, & q(0,x) &= 1-\frac{\veps^2}{2}, & x &\in [0, 0.2]\cup[0.8,1],
    \\
    \rho(0,x) &= 1 + \veps^2, & q(0,x) &= 1,  & x &\in (0.2, 0.3],
    \\
    \rho(0,x) &= 1, & q(0,x) &= 1+\frac{\veps^2}{2}, & x &\in (0.3, 0.7],
    \\
    \rho(0,x) &= 1-\veps^2, & q(0,x) &= 1, & x &\in (0.7, 0.8],
\end{align*}
where $q=\rho u$ denotes the momentum. The pressure law is specified
as $p(\rho) = \rho^2$. The computational domain $[0,1]$ is divided
into $200$ mesh points and we apply periodic boundary
conditions. Computations are carried out up to a time $T= 0.05$ and
the time-steps are calculated using \eqref{eq:timestep_suff}. We set
$\veps= 0.8, 0.3, 0.05$ and $0.001$ which correspond to the transition
from compressible to weakly compressible to almost incompressible
regime. The minimum value of the stabilisation parameter $\eta$ is
observed to be $1.5$ whereas the maximum value is $4.36$ throughout
the case studies corresponding to different $\veps$. In
Figure~\ref{fig:riemann}, we plot the density and momentum profiles
against a reference solution obtained using an explicit Rusanov scheme
on a mesh resolution of $\Delta x = 1/500$ and $\Delta t = 1/20000$
when $\veps = 0.8, 0.3, 0.05$ and using the classical explicit
time-step with $\CFL = 0.9$ as described in \eqref{eq:CFL} when $\veps
= 0.001$. The figure shows that there are shocks and expansions in the
compressible case ($\veps=0.8$) and as $\veps\to0$, the
discontinuities get weaker and the flow tends to be smooth. We further
observe that both the density and momentum converge to constant values
in the almost incompressible regime $\veps = 10^{-3}$. Hence, the
proposed scheme can perform across the three different regimes and can
resolve the respective flow features accurately. Furthermore, we note
that the maximum value $\max_{i,n}\abs{u_i^n}\Delta t^n/\Delta x$ of
the advective Courant number for the present scheme corresponding to
$\veps=0.001$ is $224.64$ whereas for the explicit Rusanov scheme it
is $0.71$. The minimum and maximum values of $\eta$ computed for this test
problem are $1.5$ and $4.35$ respectively. 

\begin{figure}[htbp]
  \centering
  \begin{tabular}{cccc}
    \includegraphics[height=0.13\textheight]{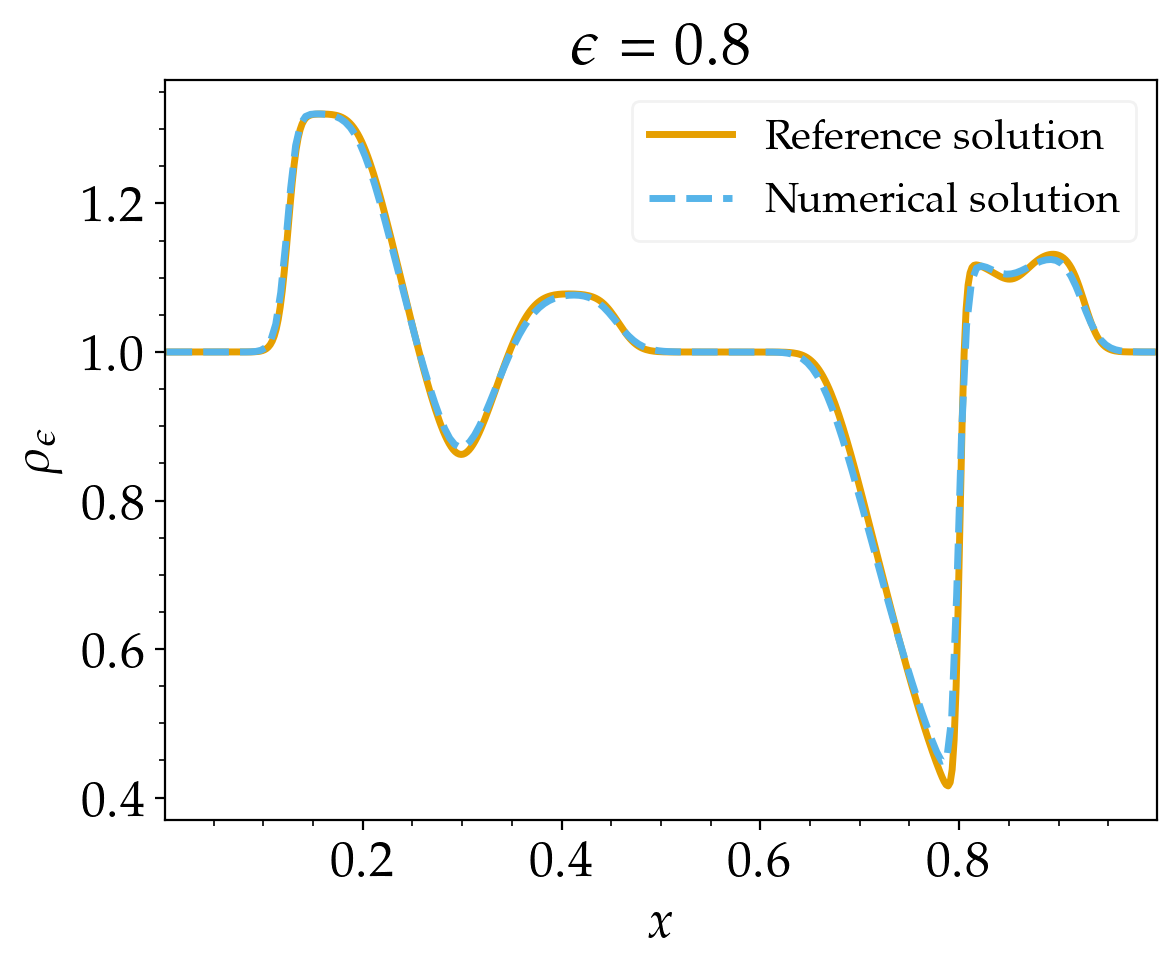} &
    \includegraphics[height=0.13\textheight]{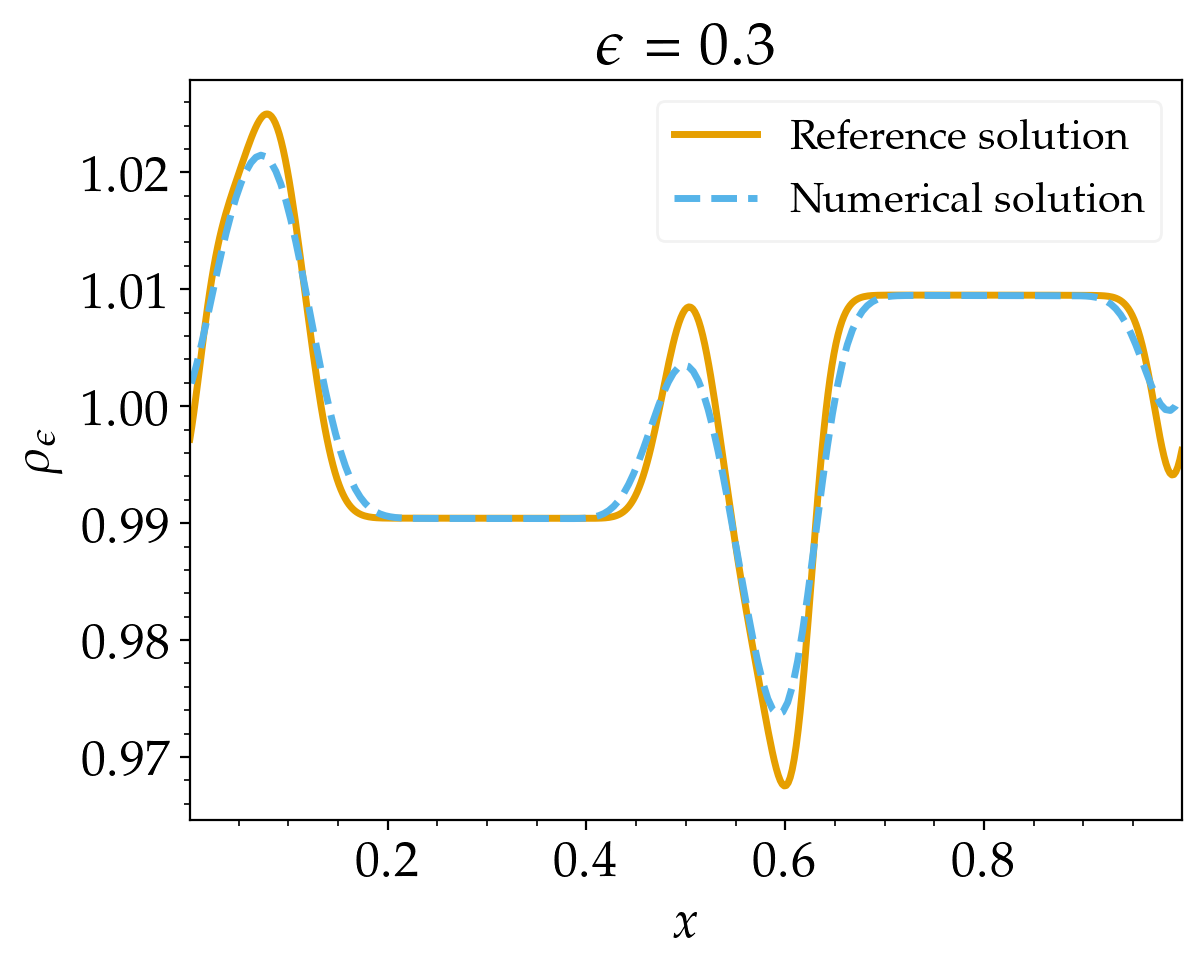} &
    \includegraphics[height=0.13\textheight]{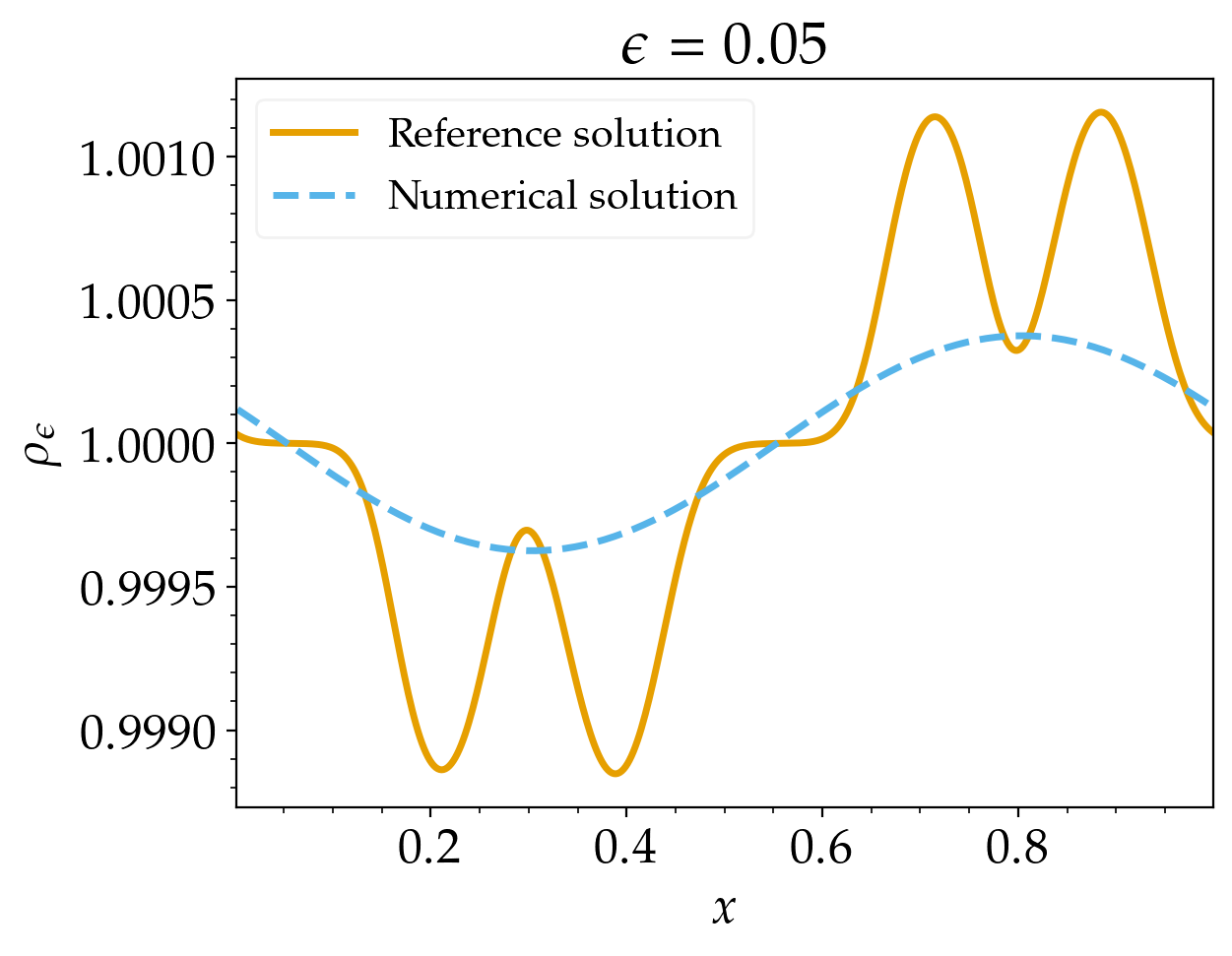} &
    \includegraphics[height=0.13\textheight]{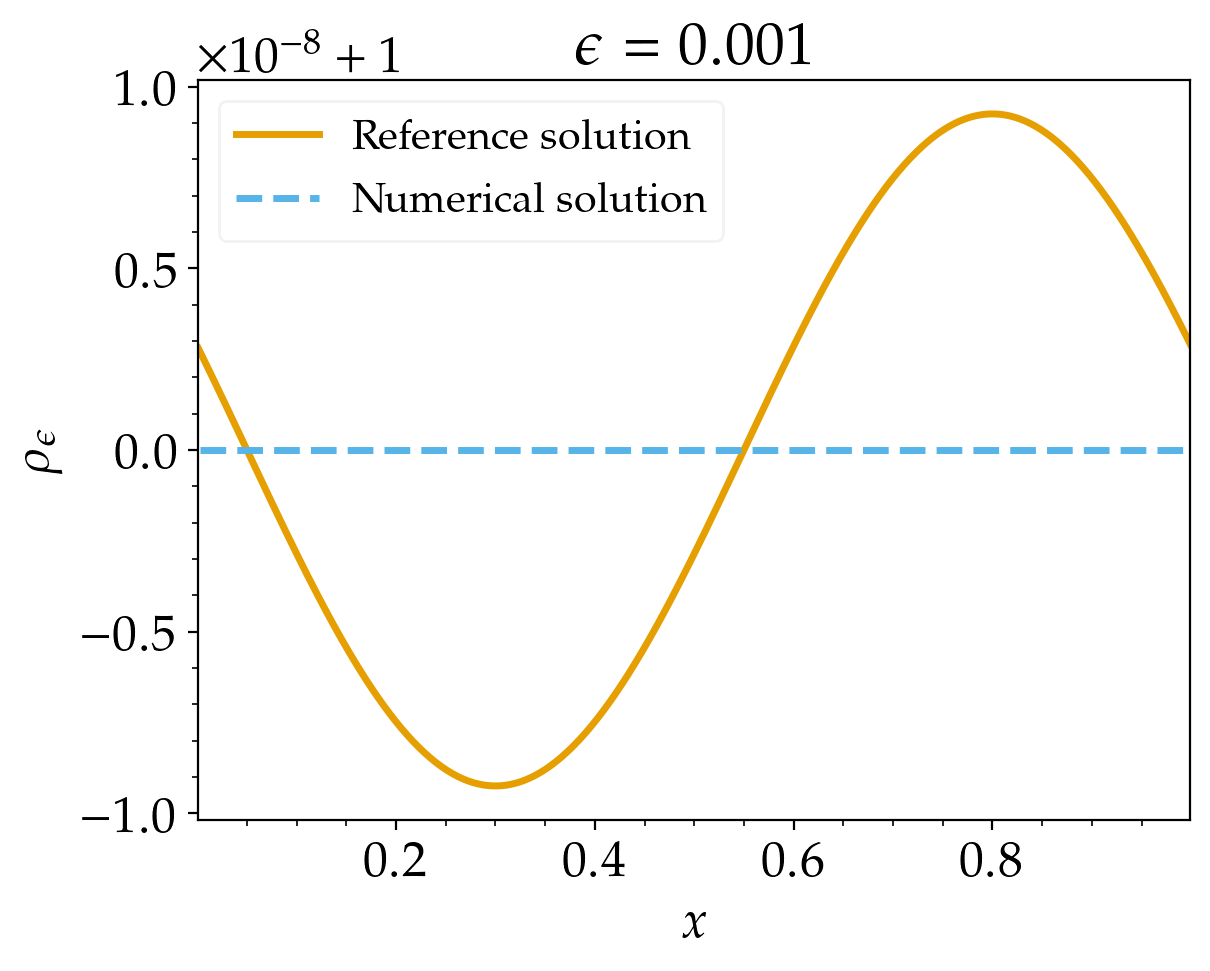} \\
    \includegraphics[height=0.13\textheight]{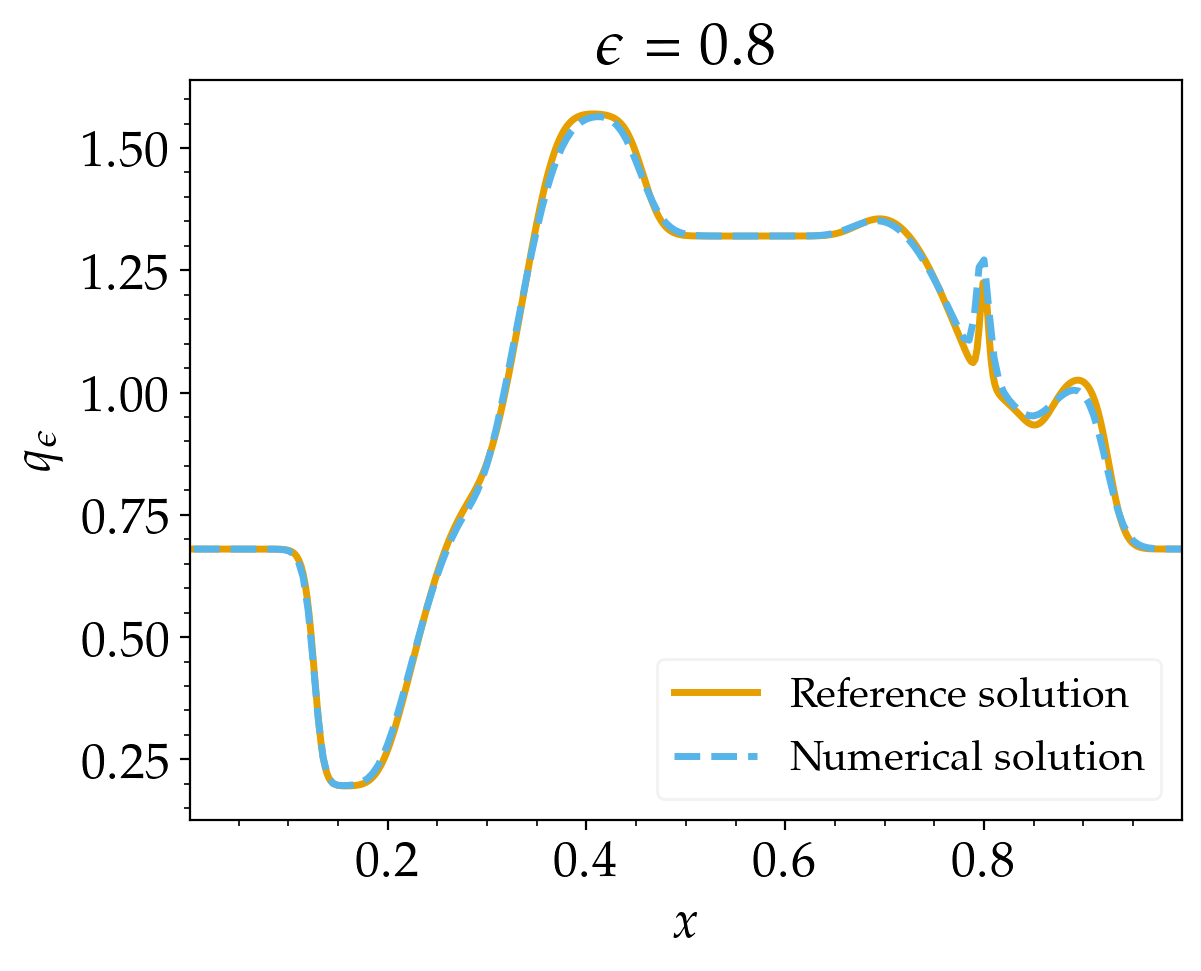} &
    \includegraphics[height=0.13\textheight]{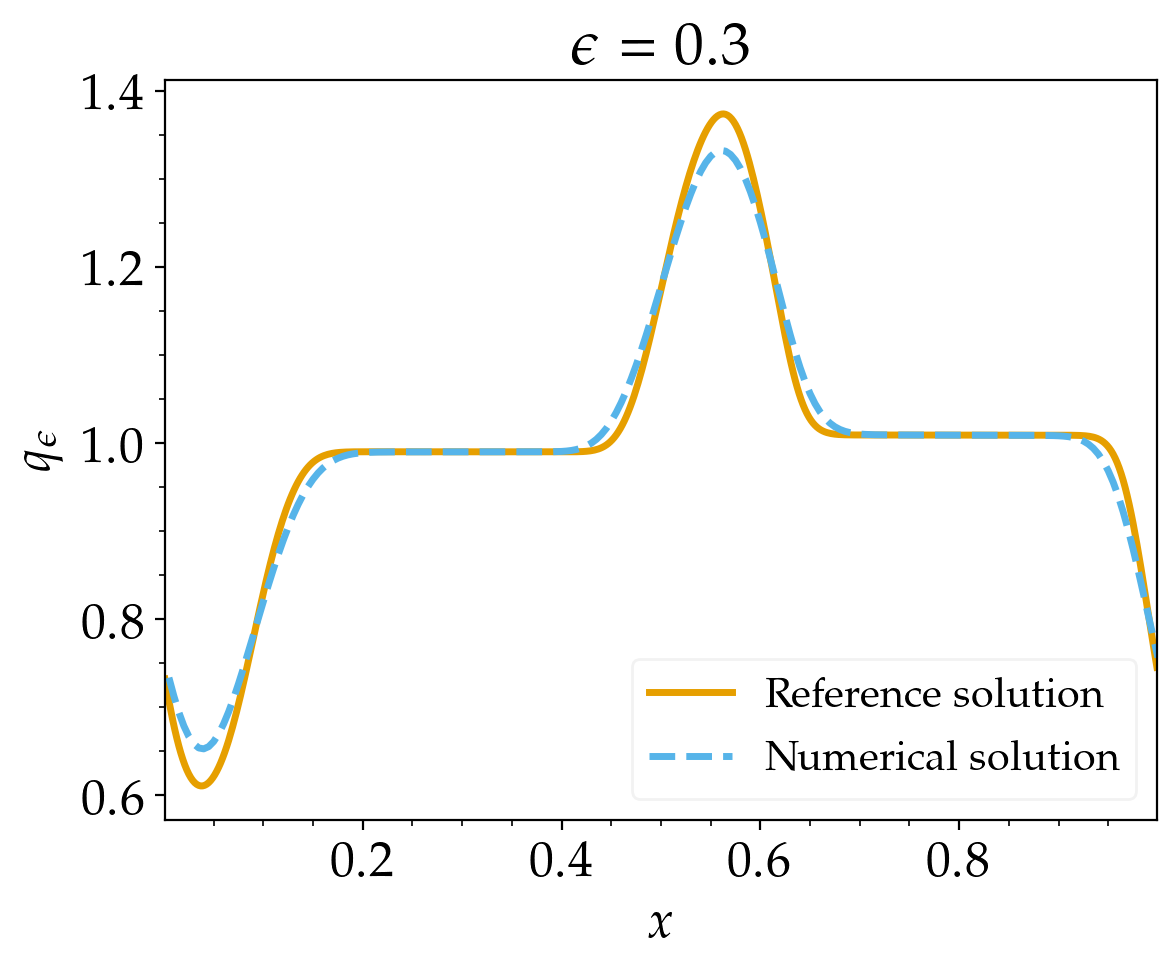} &
    \includegraphics[height=0.13\textheight]{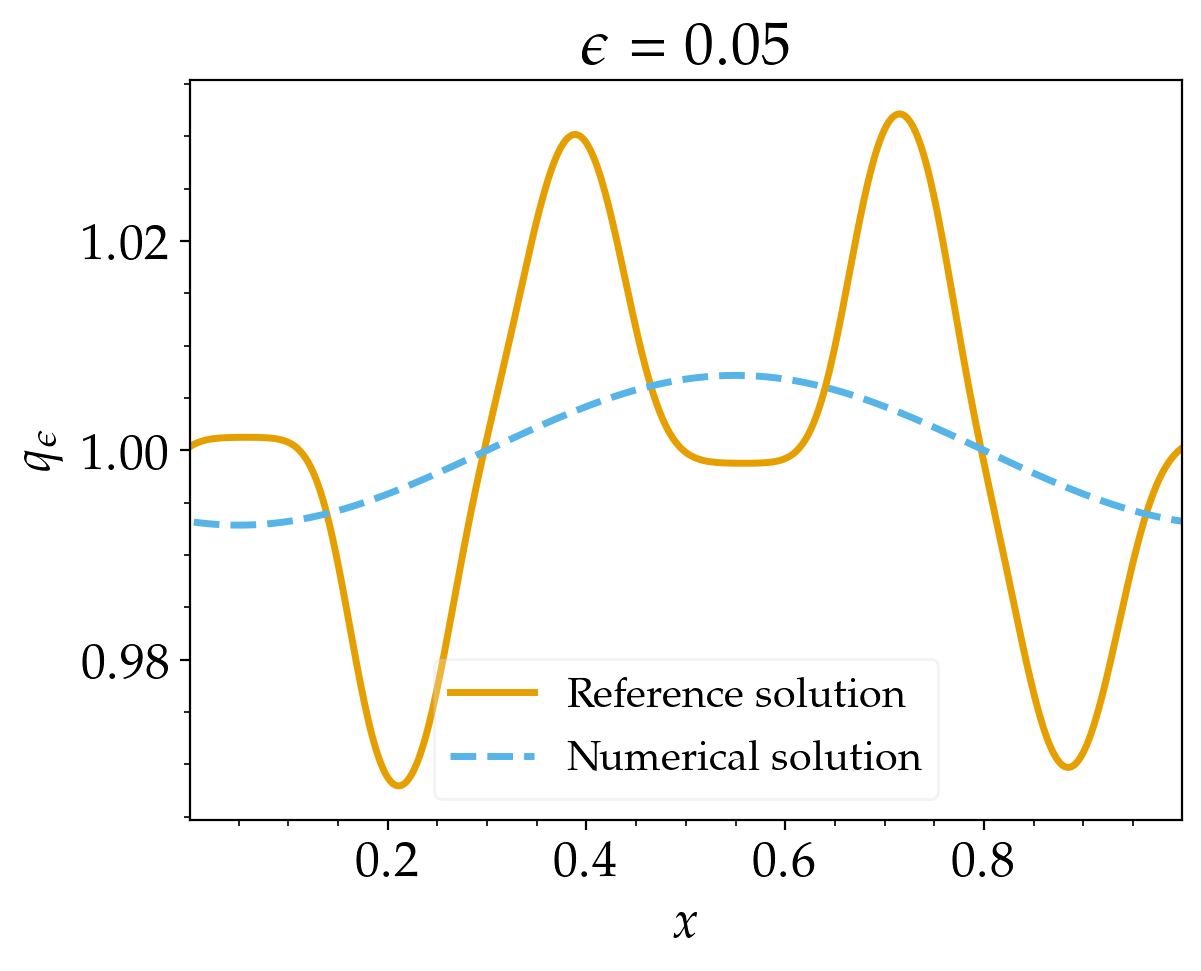} &
    \includegraphics[height=0.13\textheight]{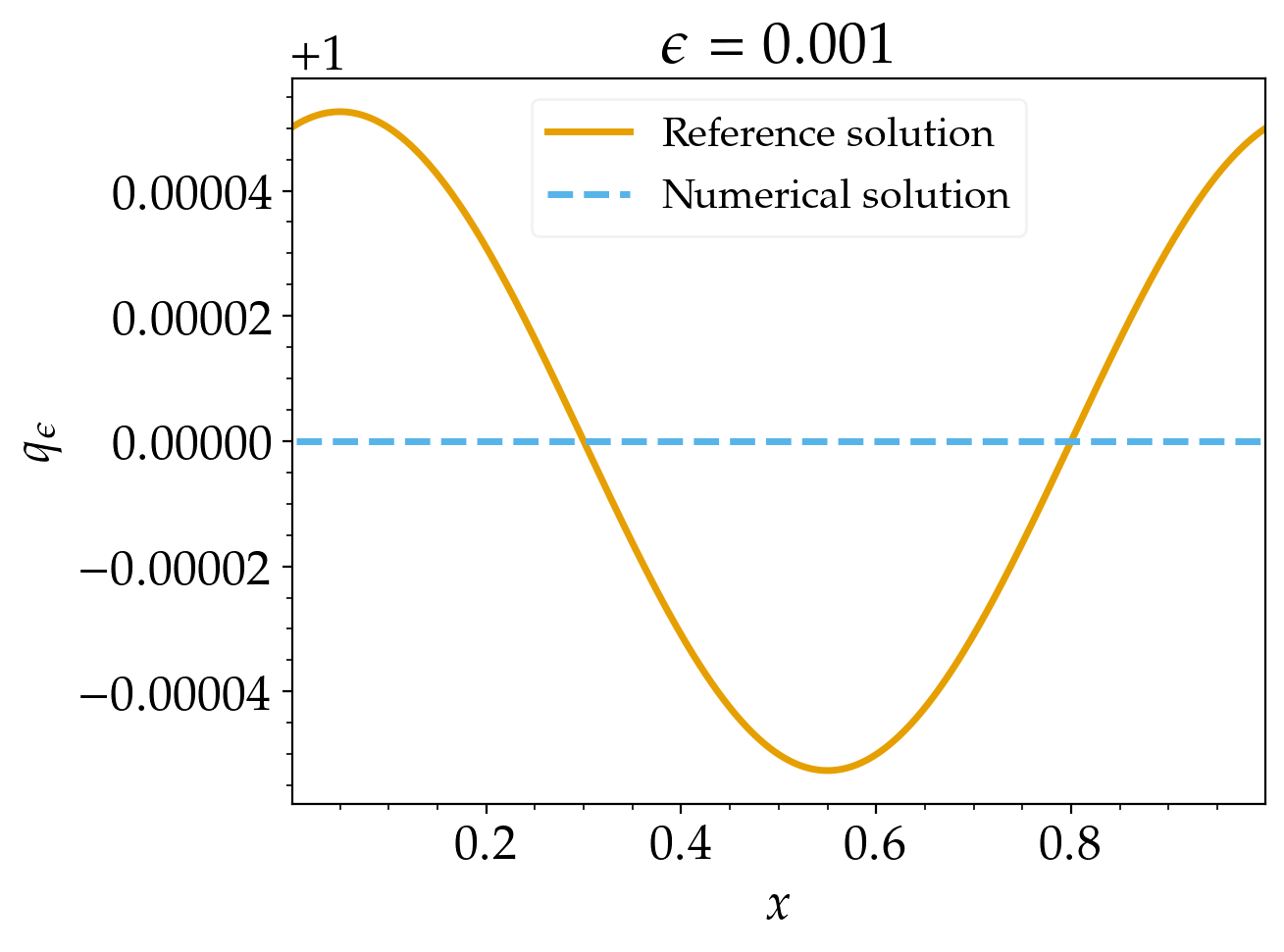}
  \end{tabular}
    \caption{Density and momentum plots on top and bottom panels
      respectively for the 1D Riemann problem at time $T=0.05$ for
      $\veps=0.8, 0.3, 0.05, 0.001$.}  
    \label{fig:riemann}
  \end{figure}
  
  In order to demonstrate the positivity preserving property of the
  scheme, we perform the simulation of the following extreme Riemann
  problem motivated by the initial data given in \cite{Tor09}: 
    \begin{equation*}
      (\rho,u)=
      \begin{cases}
        (1,-3) & \mbox{if} \ x<0,\\
        (1,3) & \mbox{if} \ x>0.
      \end{cases}
    \end{equation*}
    The computational domain $[-1,1]$ is divided into $100$ mesh
    points and the final time is $T=0.15$. The gas constant is
    $\gamma=2$, the boundaries are periodic and we set $\veps=1$ in
    order to simulate the compressible regime. The density and
    momentum obtained are given in Figure~\ref{fig:ex_riemann} which
    clearly indicates the positivity preserving property. 
    \begin{figure}[htbp]
      \centering
      \includegraphics[height=0.18\textheight]{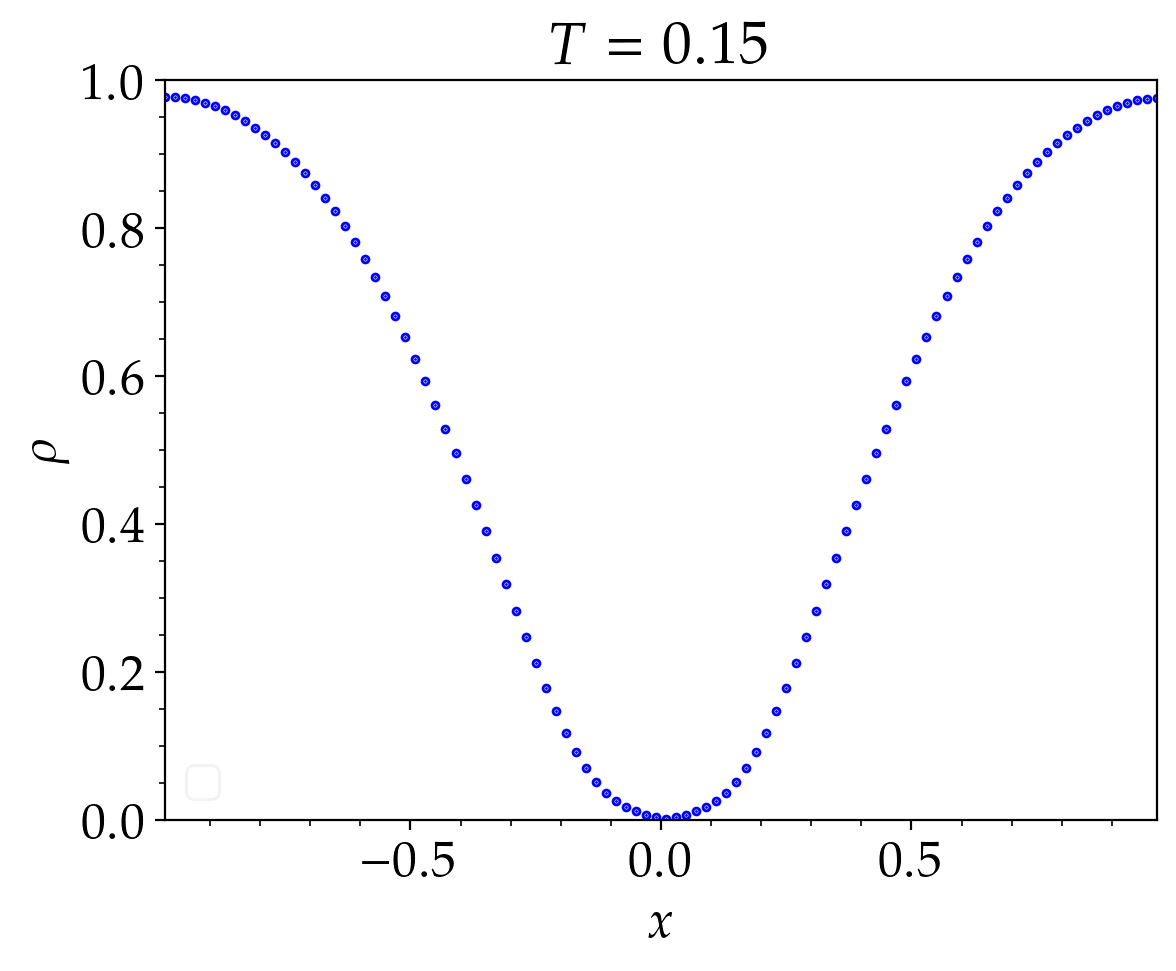}      
      \includegraphics[height=0.18\textheight]{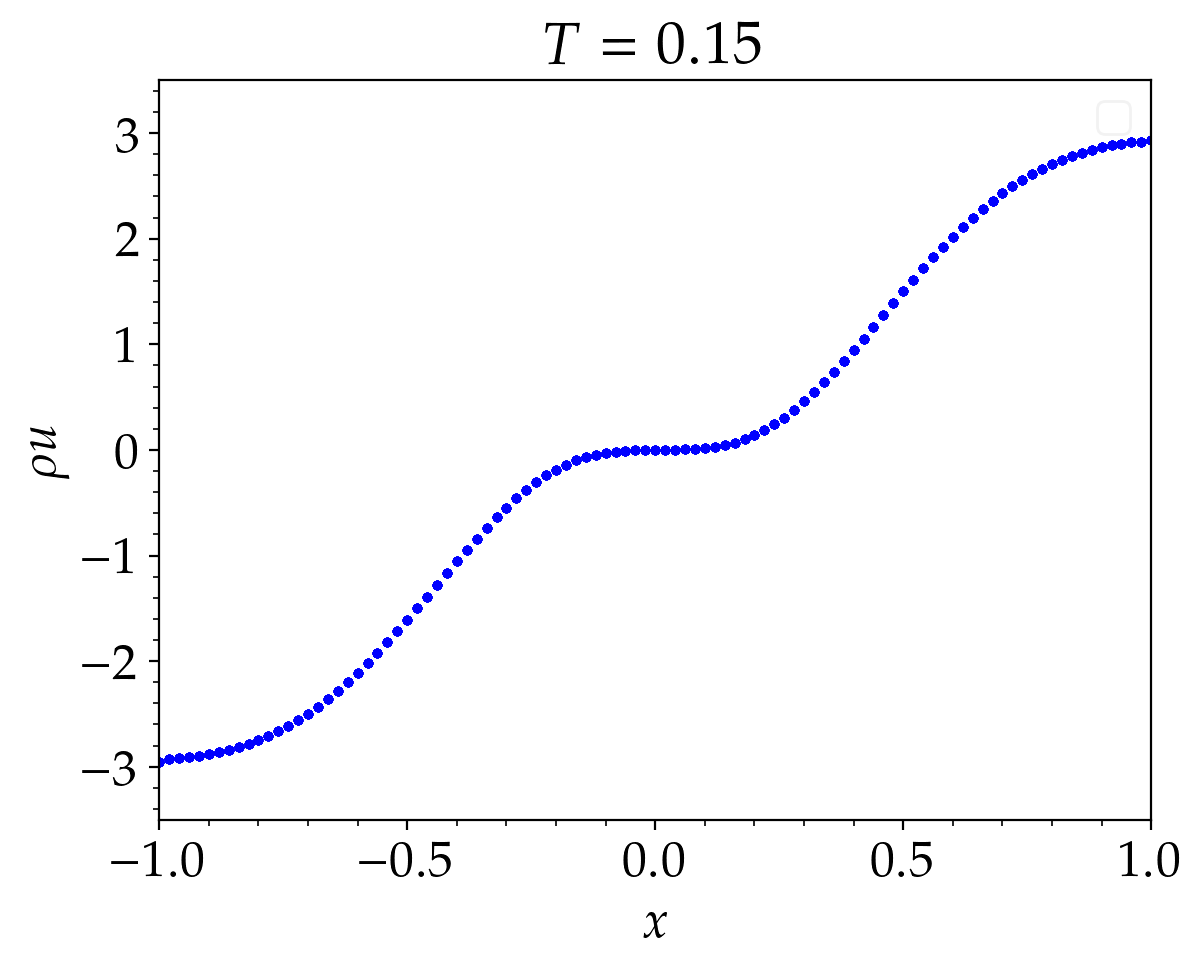}
      \caption[ex_riemann]{Extreme Riemann problem. Plots of the
        density and momentum at time $T=0.15$.}
      \label{fig:ex_riemann}
    \end{figure}
    
\subsection{Two Colliding Acoustic Pulses}
\label{sec:two_acoustic}
Here we consider the two colliding acoustic pulses problem from
\cite{Kle95}, adapted to isentropic flows as done in \cite{DT11}. The
problem describes the advection of large amplitude, short wavelength
density perturbations in a weakly compressible flow. The pressure law
is taken to be $p(\rho)=\rho^{1.4}$ and the initial density and
momentum read 
\begin{align*}
    \rho(0,x) &= 0.955 +\frac{\veps}{2}(1-\cos (2\pi x)),
    \\
    q(0,x) &= -\sign(x)\sqrt{1.4}(1-\cos (2\pi x)).
\end{align*}
We take the computational domain $[-1,1]$ which is divided into $100$
mesh points. The boundaries are periodic at both the ends. We first
plot the initial density profile corresponding to $\veps=0.1$ using
$1000$ mesh points in Figure~\ref{fig:col_pul_den}. Subsequently, we
plot the densities computed at times $T=0.01, 0.02, 0.04, 0.06, 0.08$
against a reference solution obtained using an explicit Rusanov scheme
with a mesh resolution $\Delta x=1/1000$ and $\Delta t$ computed using
the acoustic time step \eqref{eq:CFL} with CFL number
$0.5$. Analogously, in Figure~\ref{fig:col_pul_mom} we plot the
momentum profiles at times $T= 0, 0.01, 0.02, 0.04, 0.06, 0.08$. We
note the pulses superimpose, separate and this process repeats due to
the periodic boundary conditions applied. It can be seen from the
densities at $T=0$ and $T=0.08$ that because of weakly nonlinear
effects, the pulses start to steepen, resulting in the formation of two weak
shocks. Furthermore, we observe that the scheme is able to maintain
the amplitude of the pulses despite the flow being at low Mach numbers
and the initial data being ill-prepared. The minimum and maximum values of 
$\eta$ computed for this particular test problem are $1.57$ and $1.72$,
respectively.     
\begin{figure}[htbp]
  \centering
  \begin{tabular}{c c c}
  \includegraphics[height=0.16\textheight]{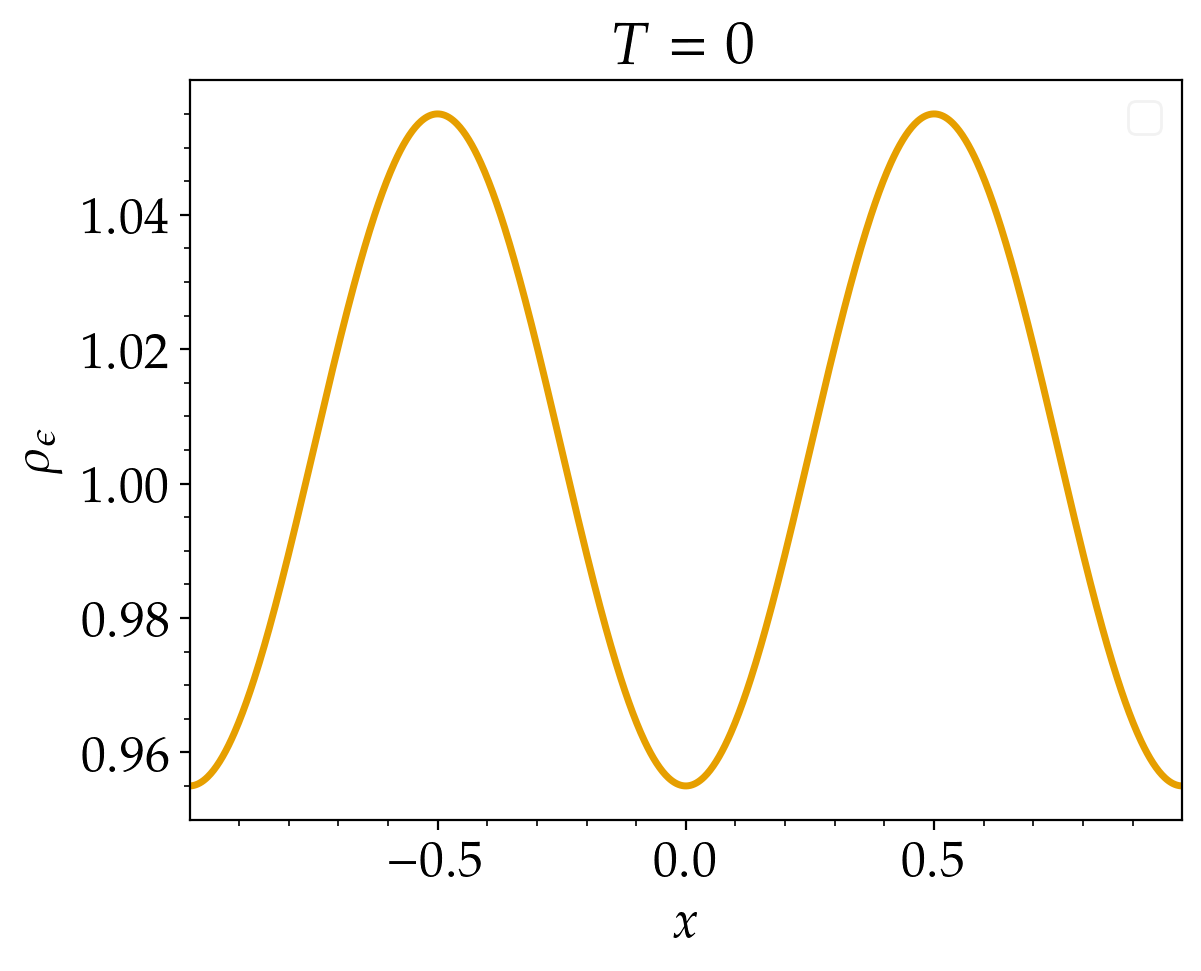} &
  \includegraphics[height=0.16\textheight]{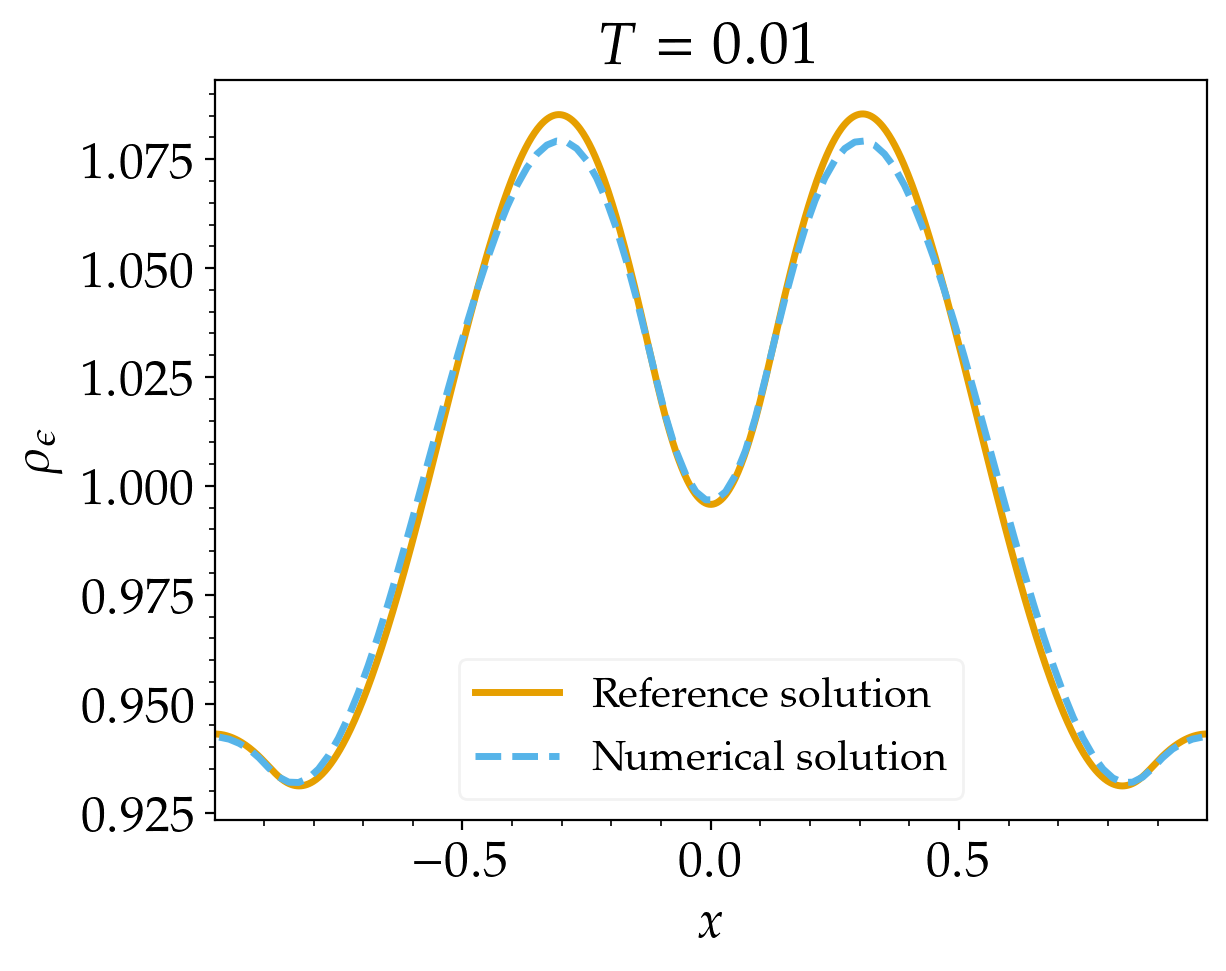} &
  \includegraphics[height=0.16\textheight]{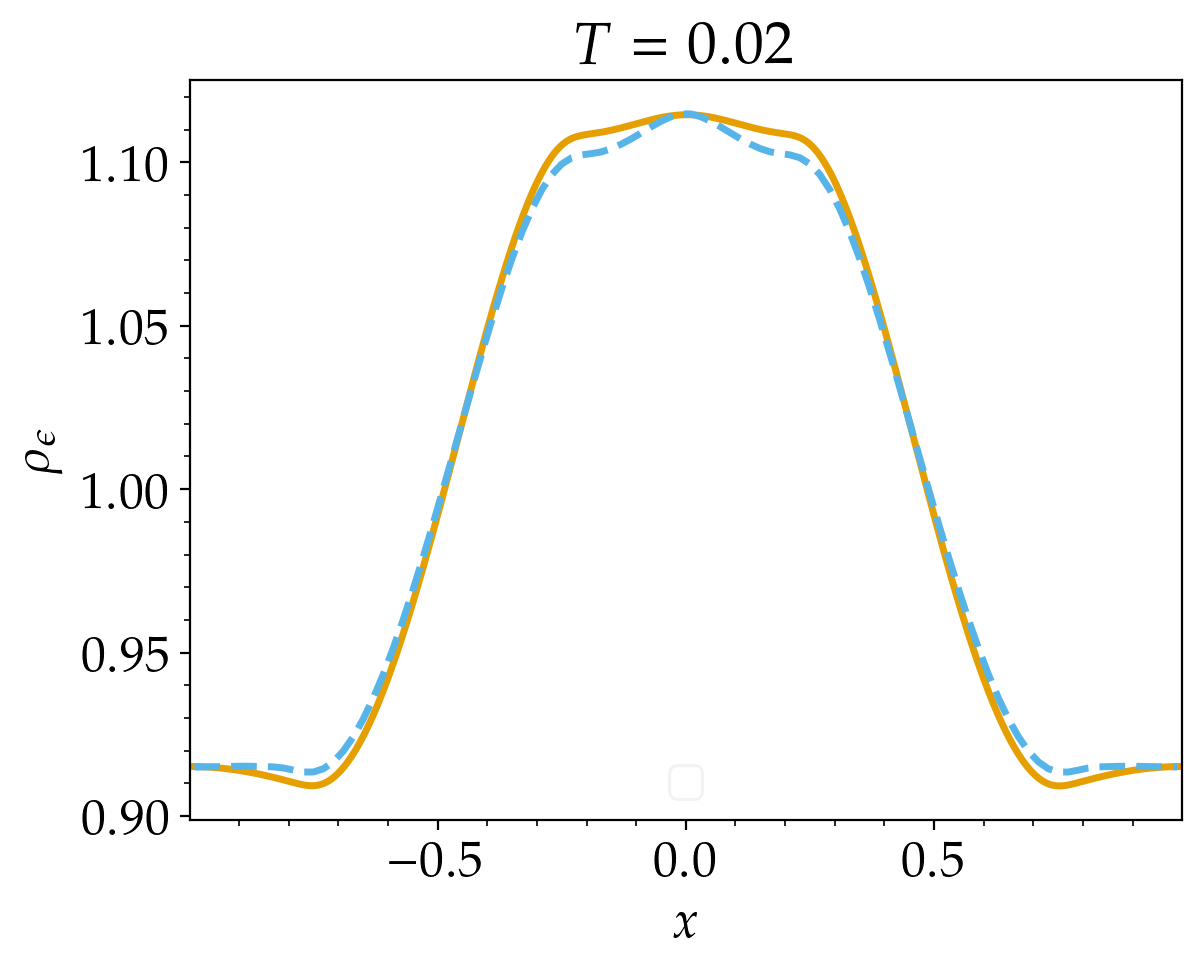} \\
  \includegraphics[height=0.16\textheight]{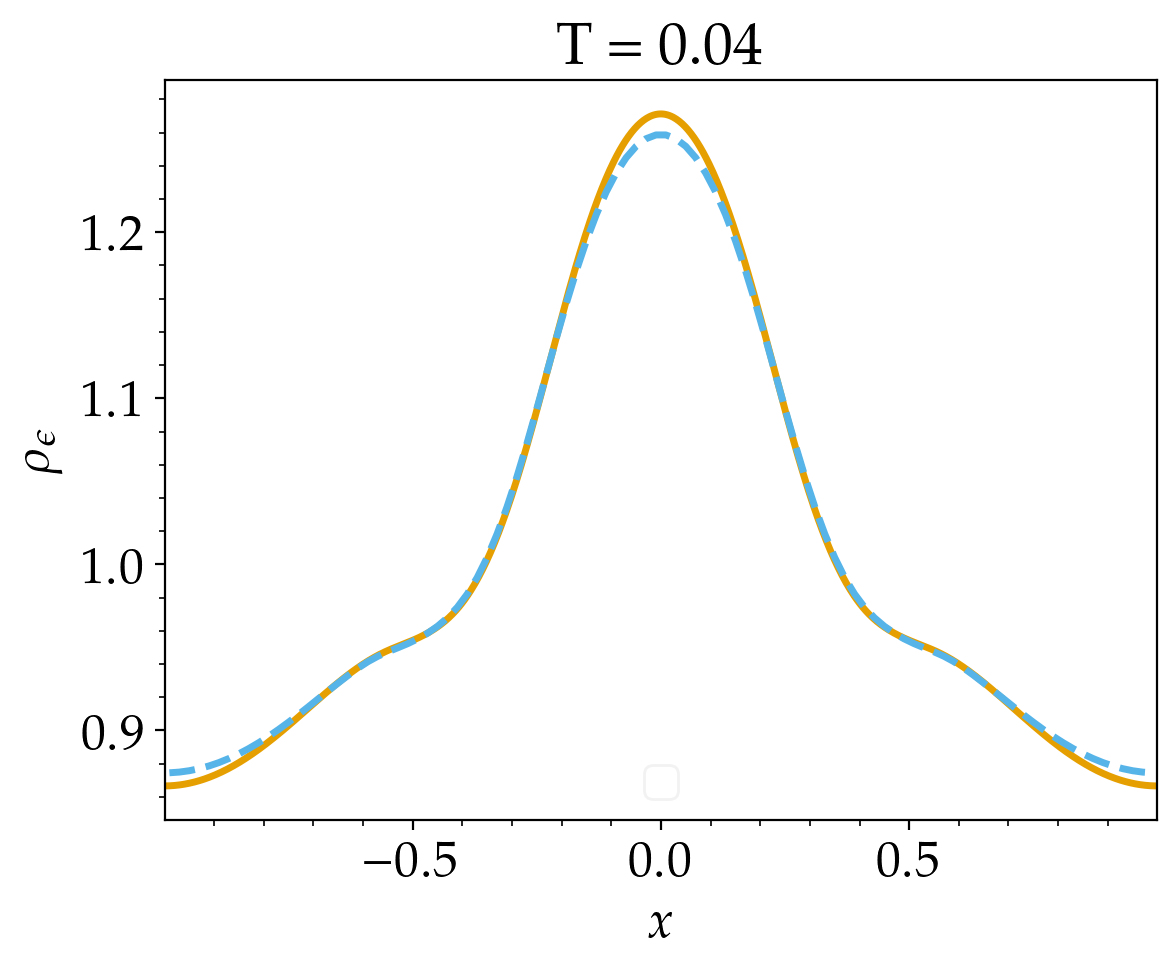} &
  \includegraphics[height=0.16\textheight]{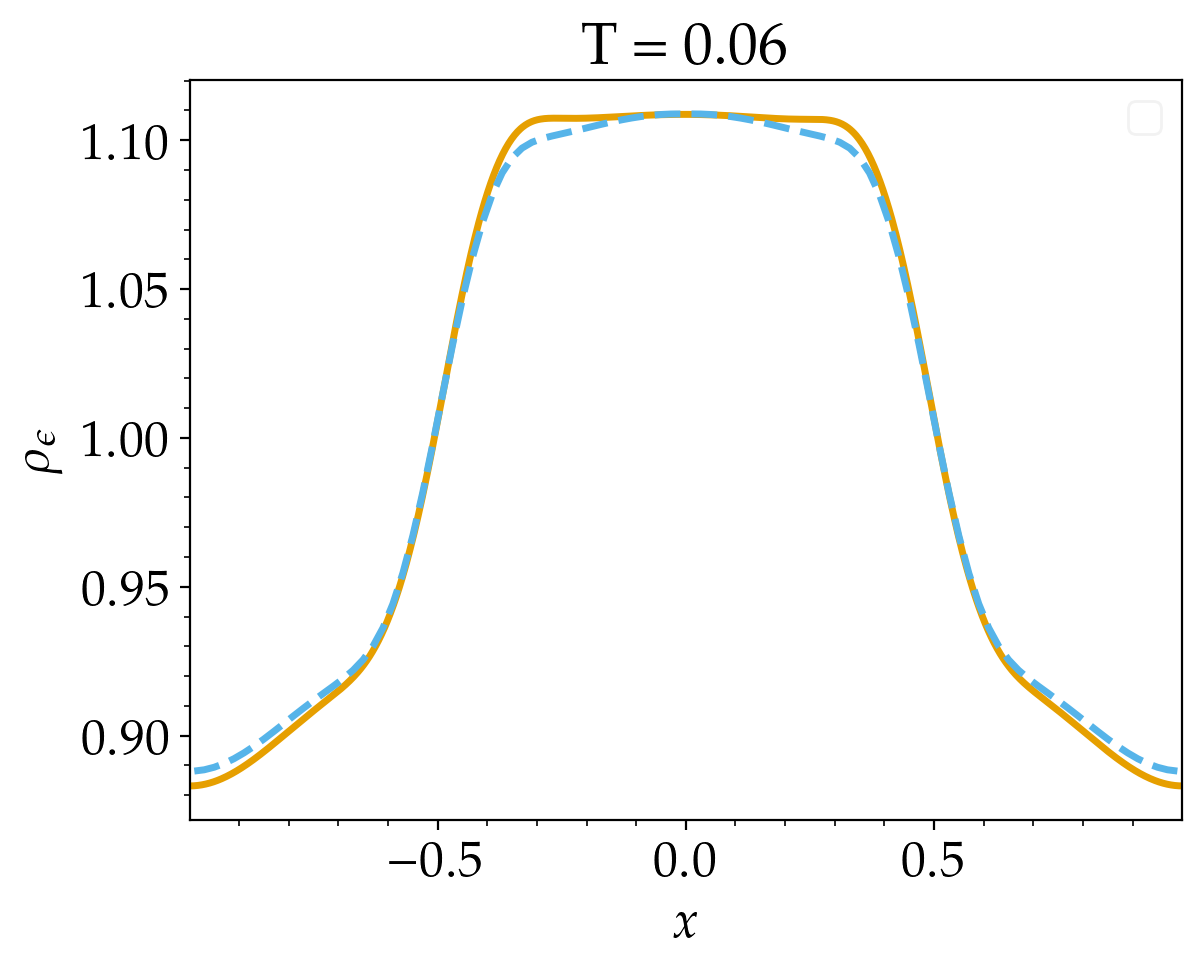} &
  \includegraphics[height=0.16\textheight]{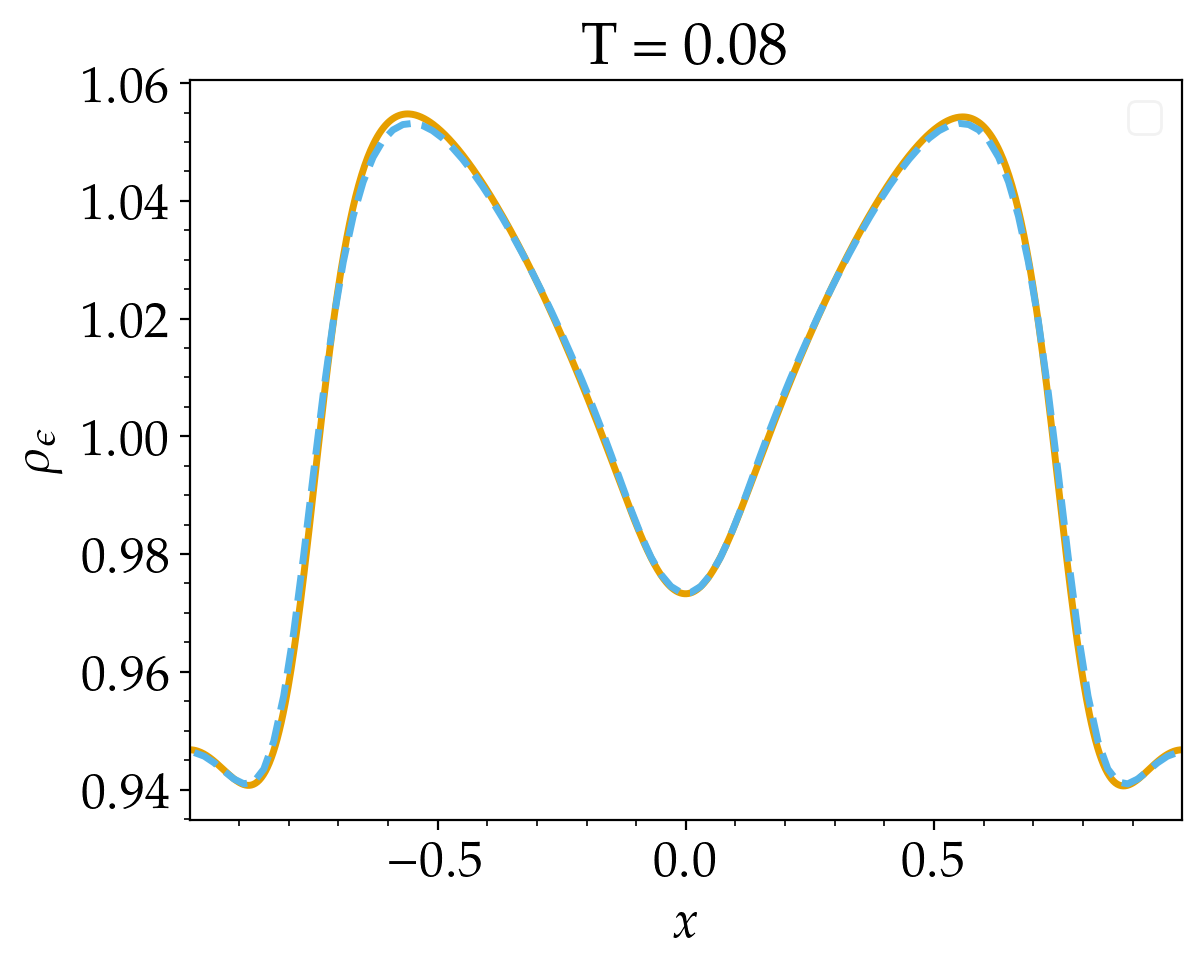}
  \end{tabular}
  \caption{Density profiles for the two colliding acoustic pulses
    problem for $\veps=0.1$ at times $T=0, 0.01, 0.02, 0.04, 0.06,
    0.08$.}      
  \label{fig:col_pul_den}
\end{figure}

\begin{figure}[htbp]
  \centering
  \begin{tabular}{c c c}
  \includegraphics[height=0.16\textheight]{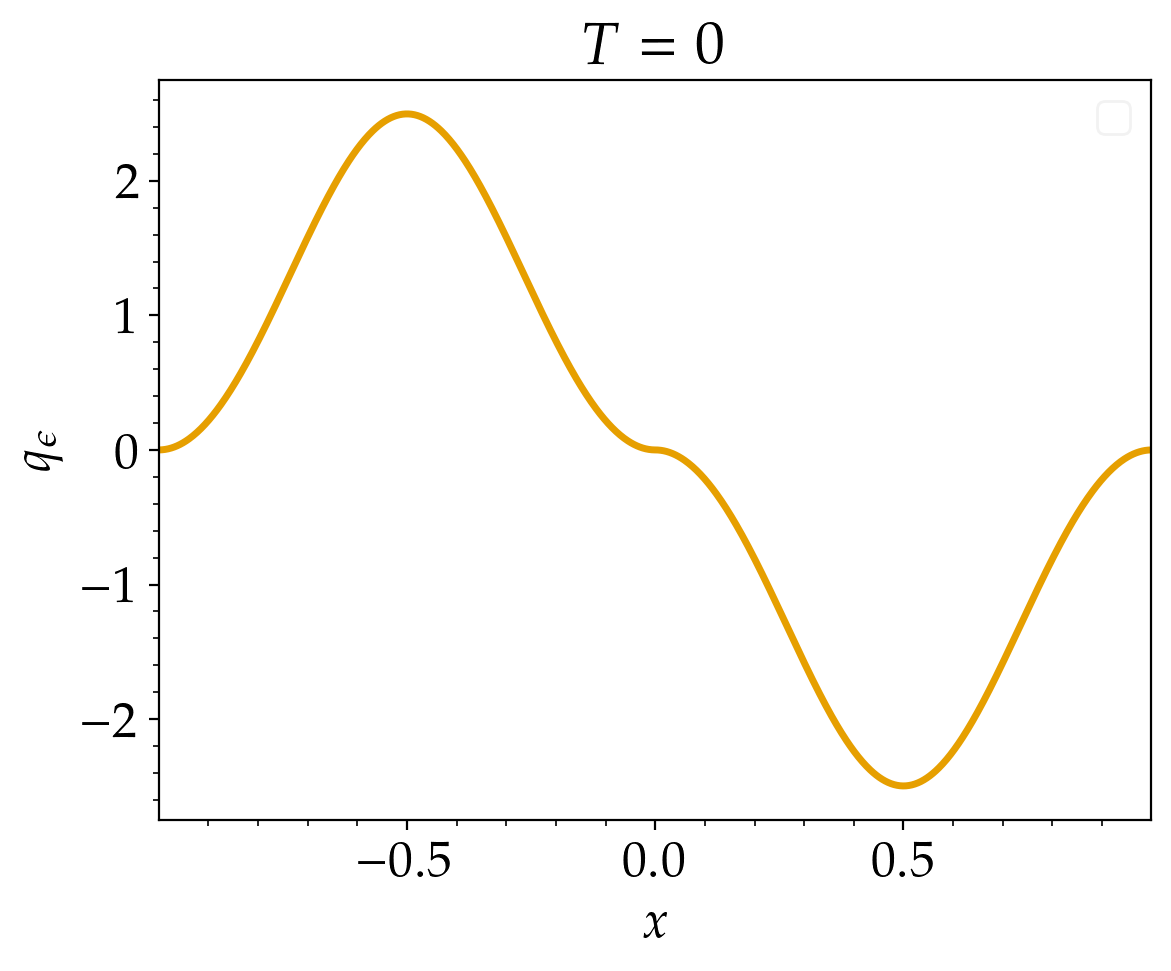}  &   
  \includegraphics[height=0.16\textheight]{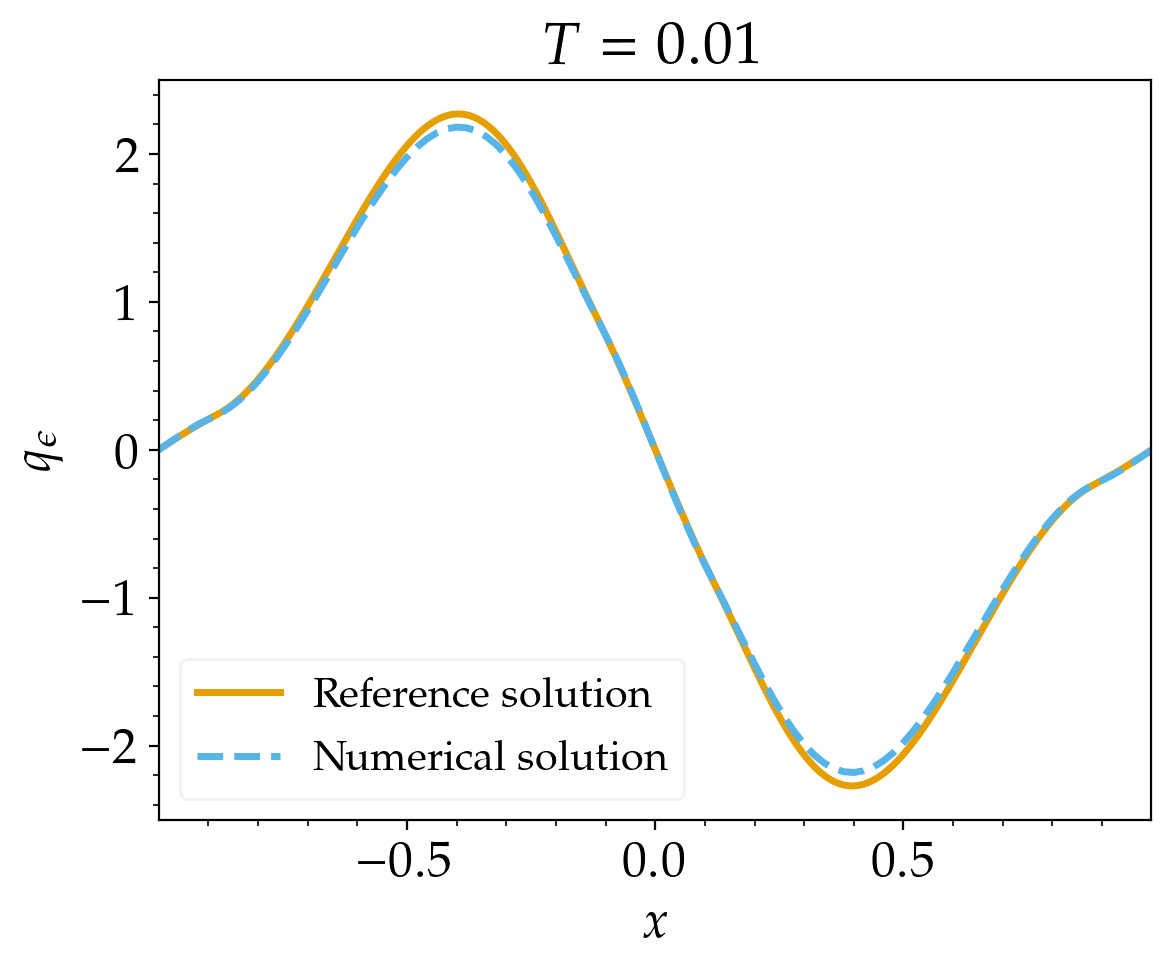} &
  \includegraphics[height=0.16\textheight]{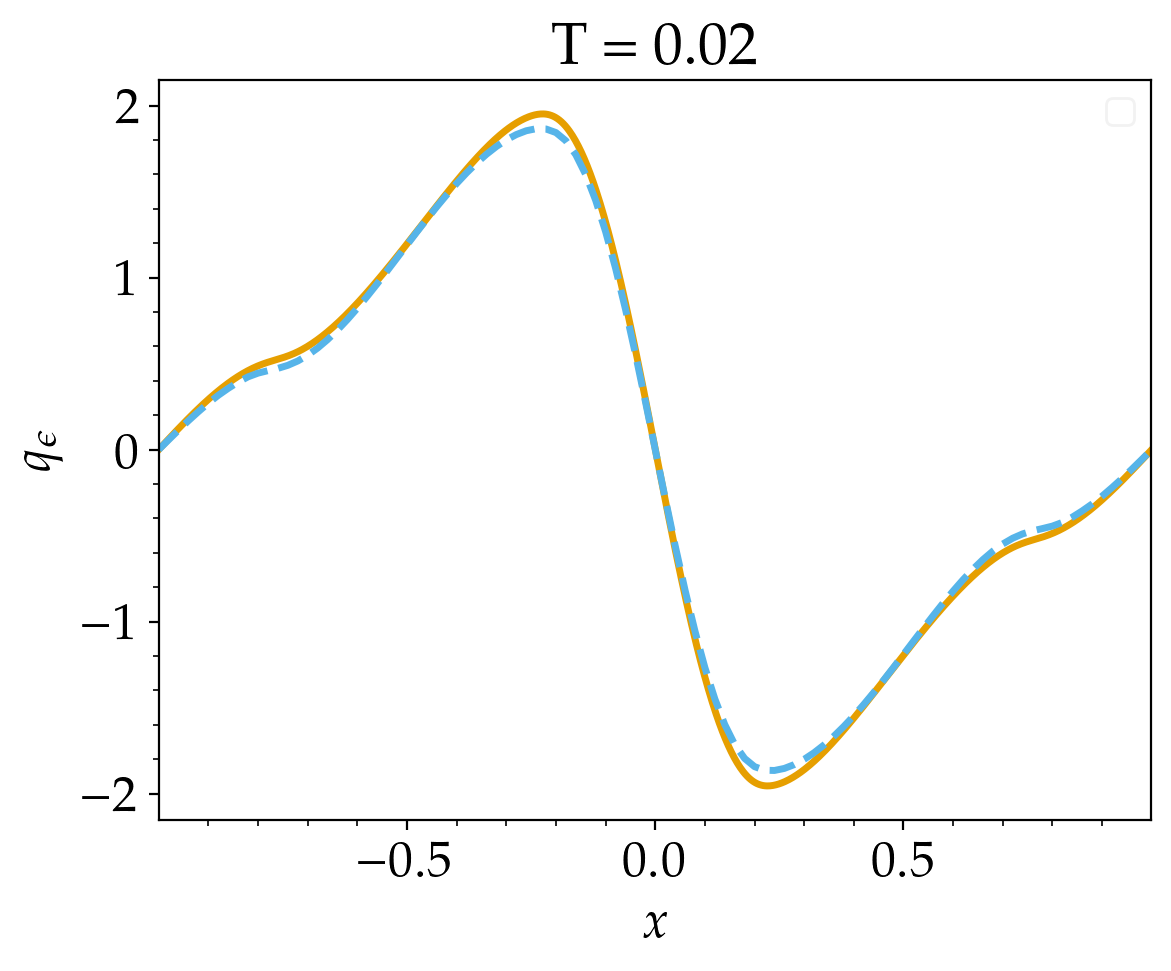} \\
  \includegraphics[height=0.16\textheight]{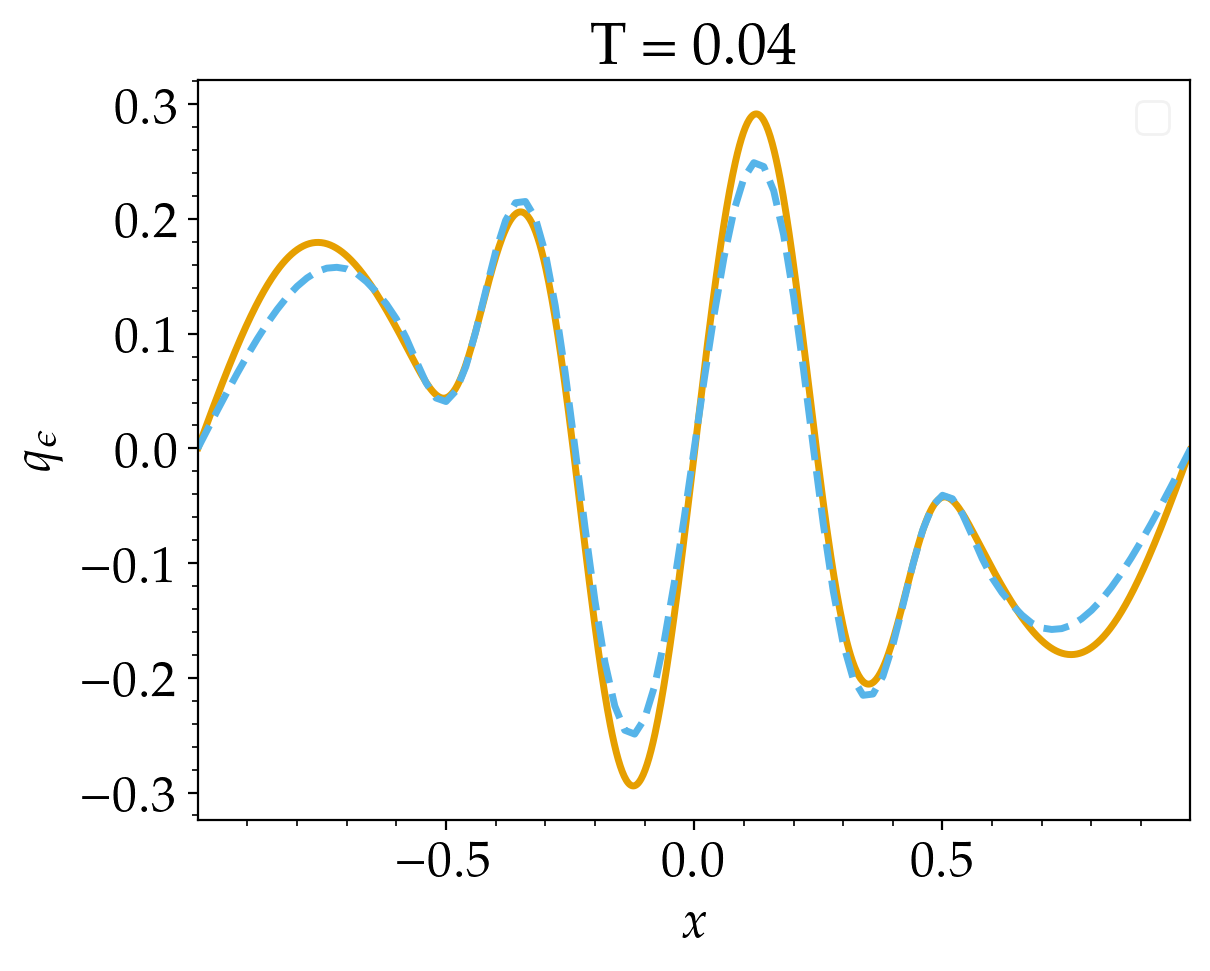} &
  \includegraphics[height=0.16\textheight]{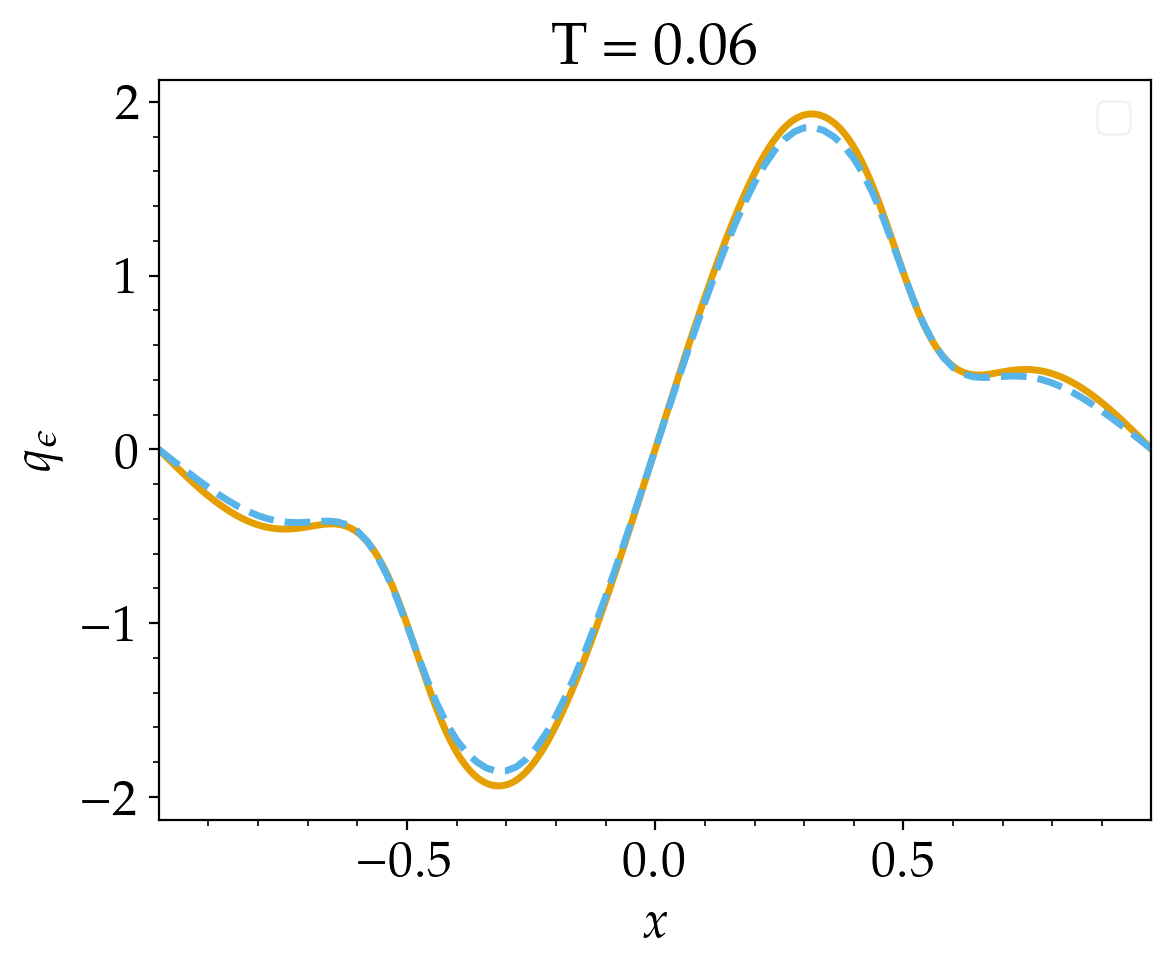} &
  \includegraphics[height=0.16\textheight]{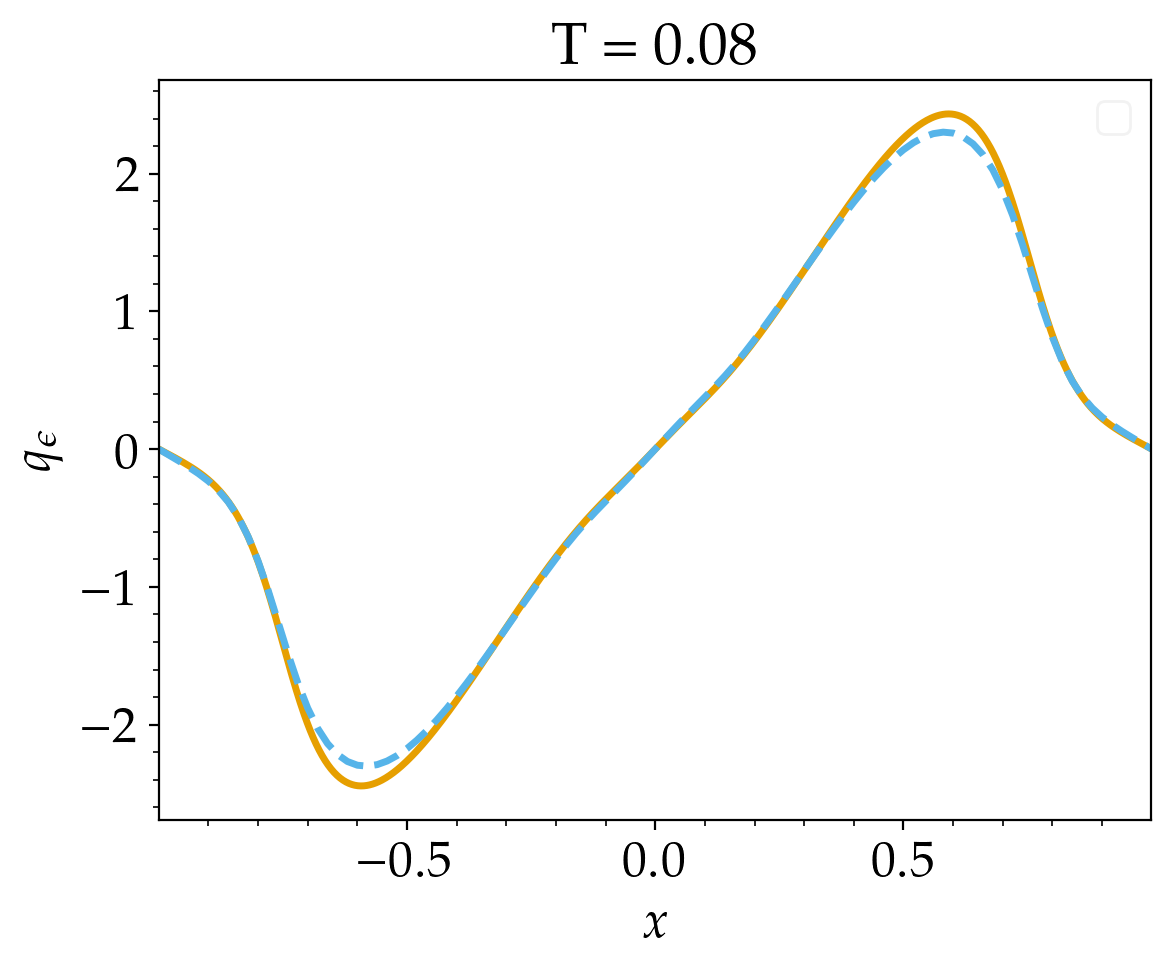}
  \end{tabular}
  \caption{Momentum profiles for the two colliding acoustic pulses
    problem for $\veps=0.1$ at times $T=0, 0.01, 0.02, 0.04, 0.06,
    0.08$.}      
  \label{fig:col_pul_mom}
\end{figure}

To further demonstrate the scheme's capability to resolve low Mach
number flows with ill-prepared initial data, we compute the density and
momenta at times $T=0.001, 0.005, 0.008$ for an even smaller value
$\veps = 0.01$. The results are displayed in
Figure~\ref{fig:col_pul_eps0p01} where the pulses superimpose, gain
maximum amplitude and separate. 

\begin{figure}[htbp]
  \centering
  \begin{tabular}{c c c}
    \includegraphics[height=0.15\textheight]{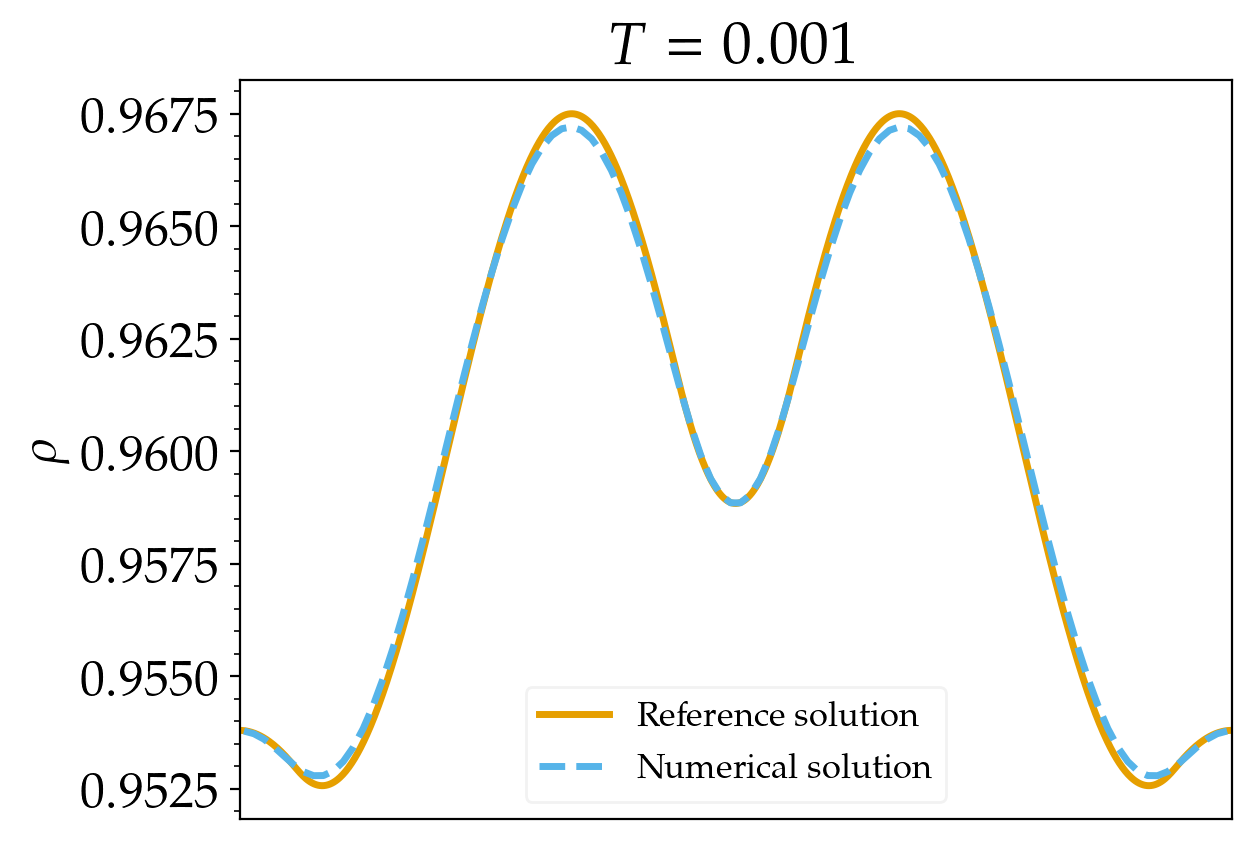} &
    \includegraphics[height=0.15\textheight]{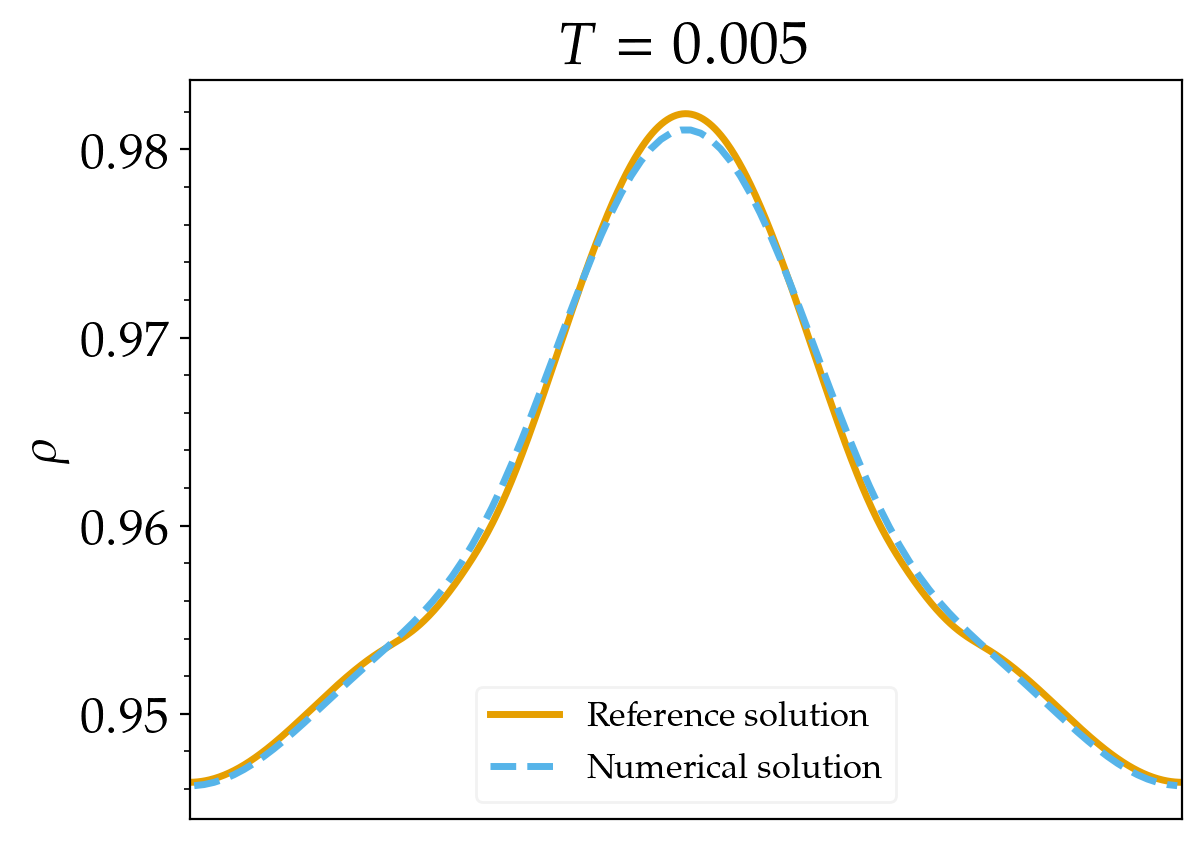} &
    \includegraphics[height=0.15\textheight]{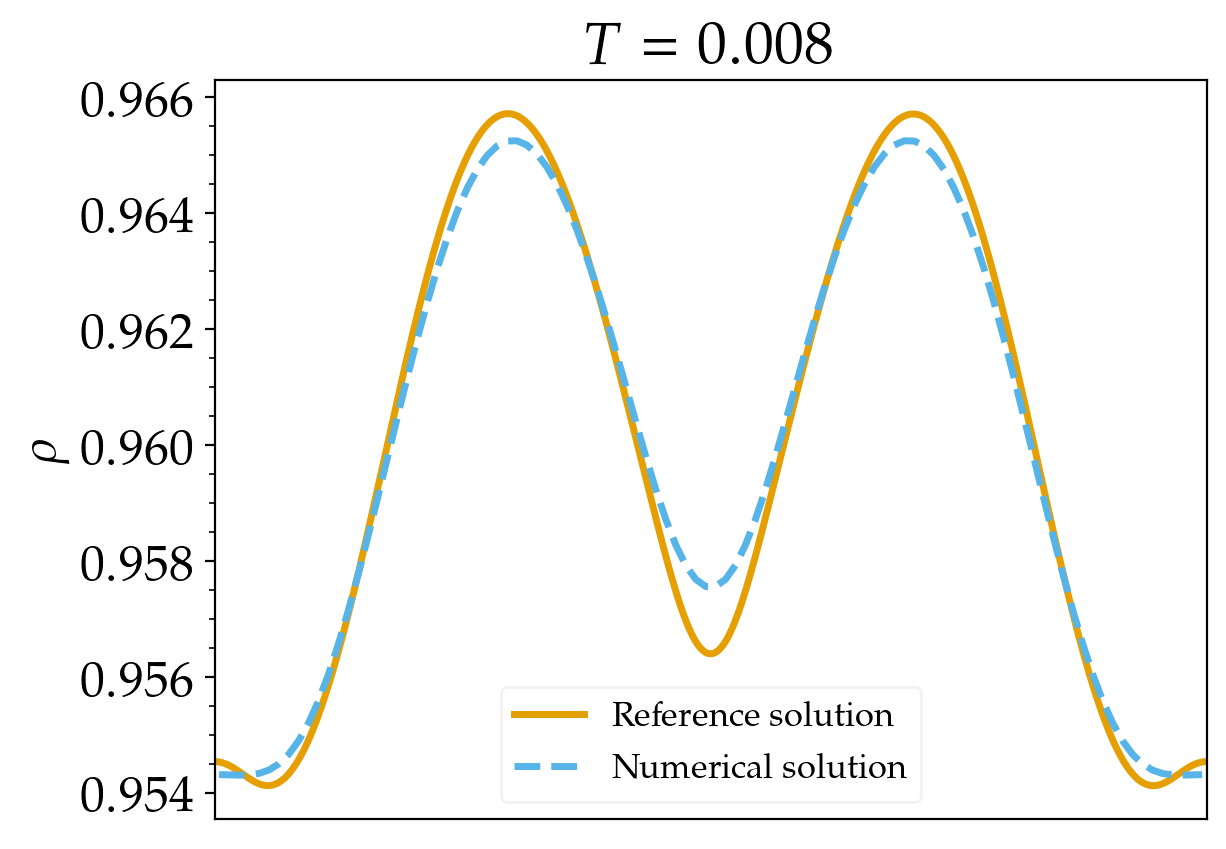}\\
    \includegraphics[height=0.16\textheight]{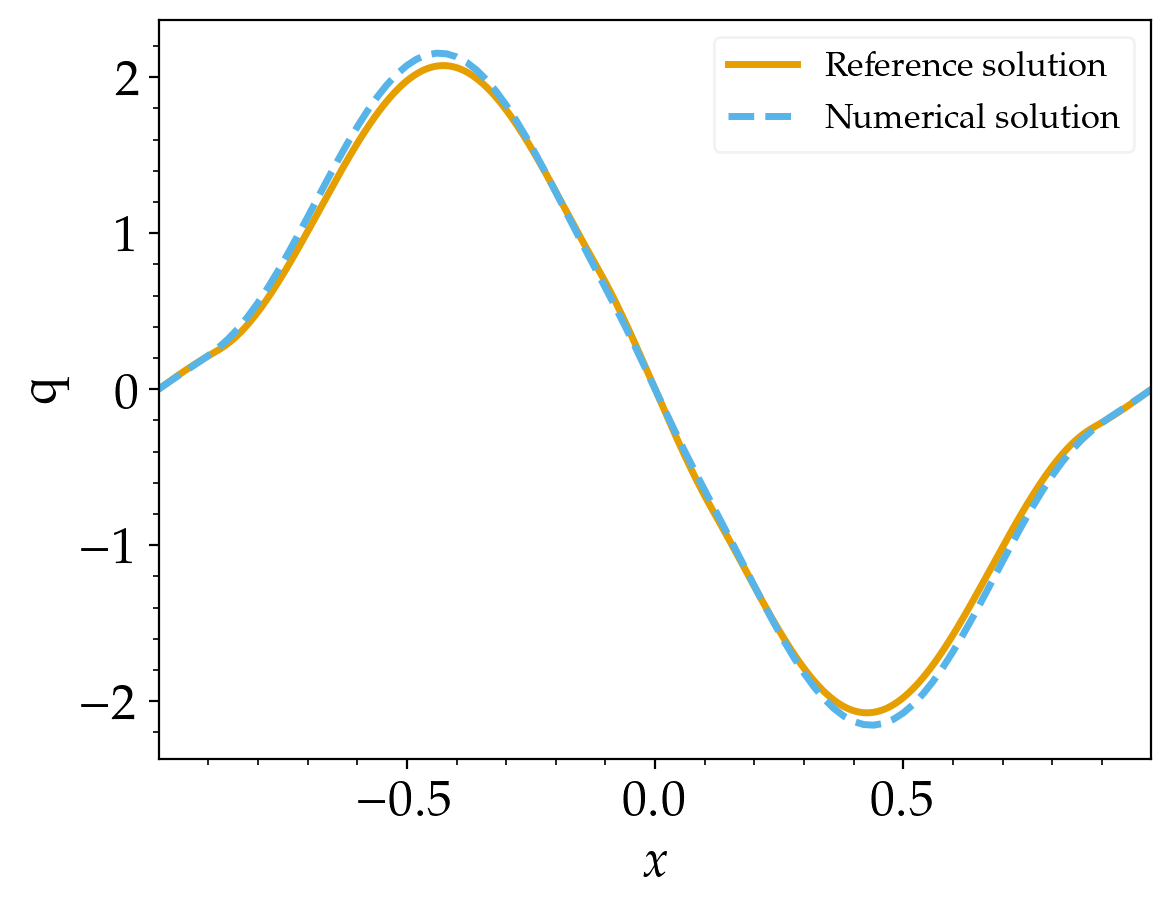} & 
    \includegraphics[height=0.16\textheight]{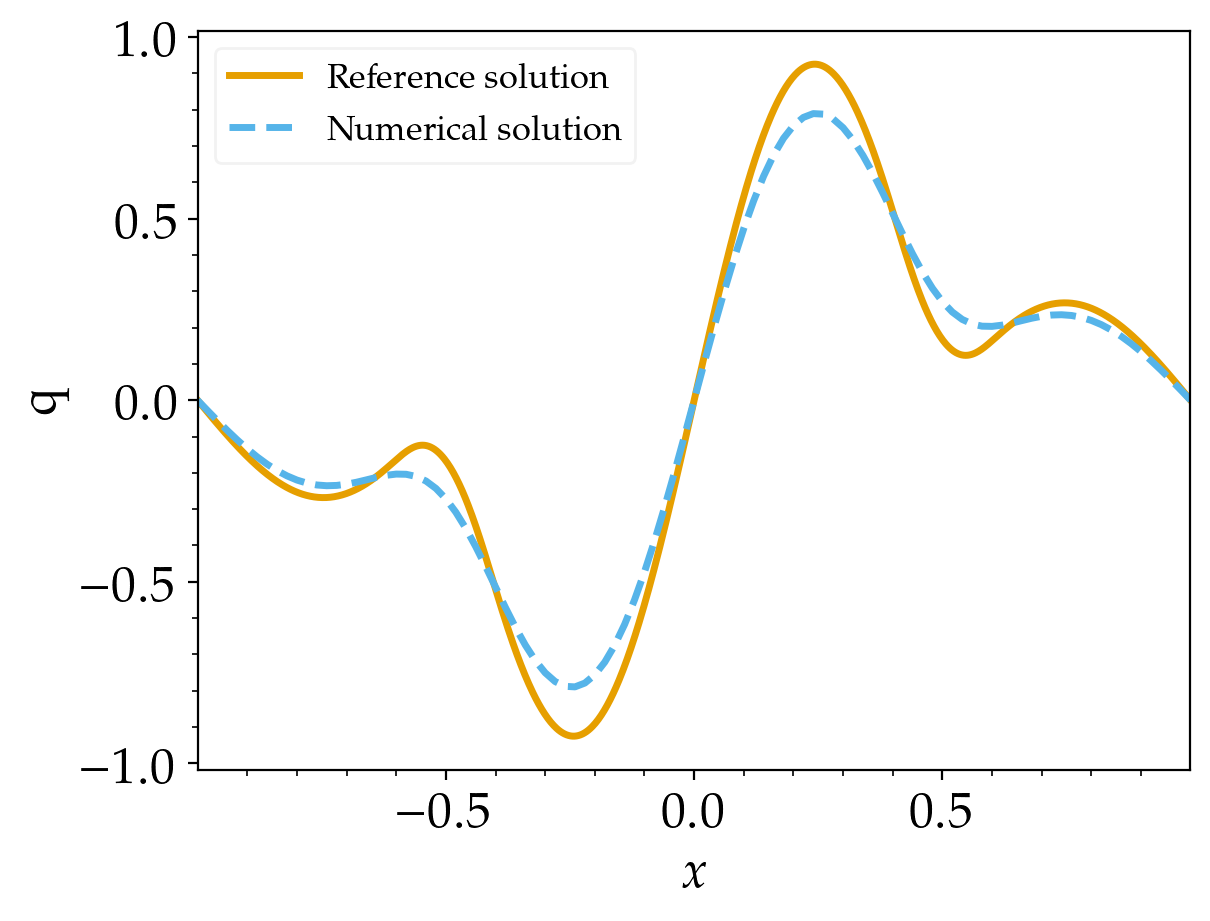} &
    \includegraphics[height=0.16\textheight]{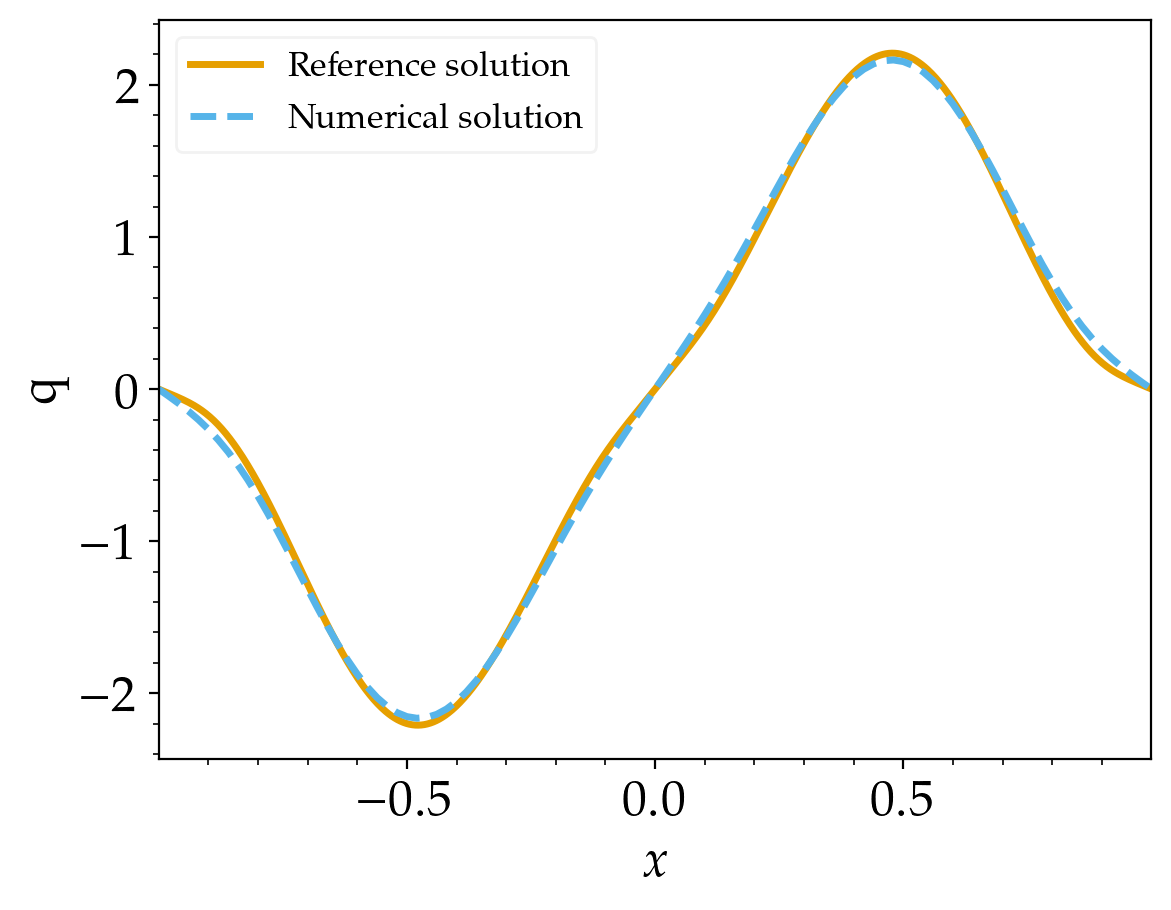}
    \end{tabular}
    \caption{Density and momentum profiles for the two colliding
      acoustic pulses problem for $\veps=0.01$ at times $T=0.001,
      0.005, 0.008$.}     
    \label{fig:col_pul_eps0p01}
  \end{figure}
  
\subsection{Advecting Vortex}
\label{sec:vortex}
Drawing inspiration from \cite{AS20}, we consider the following
advecting vortex problem which is a case study used to analyse the
numerical dissipation of explicit time-stepping schemes. It has been
reported in the literature that the dissipation increases with
decreasing Mach numbers. The moving vortex is a low Mach number flow
which does not contain acoustic waves. A radially symmetric vortex is
placed at the point $(x_0,y_0)=(0.5,0.5)$ in the computational domain
$\Omega = [0,1]\times[0,1]$. The initial data read 
\begin{align*}
   \rho(0, x, y) &= 110 +
                   \Big(\frac{\zeta\Gamma}{\omega}\Big)^2(k(\omega
                   r)-k(\pi))\mcal{X}_{\omega r\leq \pi}, 
   \\
   u(0, x, y) &= 0.1 + \Gamma(1+\cos(\omega r))(0.5-y)\mcal{X}_{\omega r\leq \pi},
   \\
   v(0, x, y) &= \Gamma(1+\cos(\omega r))(0.5-x)\mcal{X}_{\omega r\leq \pi},
\end{align*}
where $r=\sqrt{(x-x_0)^2+(y-y_0)^2}$. We set $\Gamma = 1.5$, $\omega =
4\pi$ and $k(r) =
2\cos(r)+2r\sin(r)+\frac{1}{8}\cos(2r)+\frac{r}{4}\sin(2r)+0.75r^2$.
Here, $\Gamma$ is a parameter known as the vortex intensity, $r$
denotes the distance from the core of the vortex, and $\omega$ is an
angular wave frequency specifying the width of the vortex. The Mach
number $\veps$ is controlled by adjusting the value of $\zeta$ via the
relation $\veps = 0.1\zeta/\sqrt{110}$. We further assume the
gas law $p(\rho)=\rho^2$. We take a mesh grid of $100\times100$
points and impose periodic boundary conditions on all sides. In
order to assess the dissipation of the scheme and its dependence on
$\veps$, we compute the flow Mach number $M=\sqrt{(u-0.1)^2+v^2}/c$ at
one period of rotation $T=5/3$ for different values $\veps =
10^{-1}, 10^{-2}, 10^{-3}, 10^{-4}, 10^{-5}, 10^{-6}$. We present the
pseudocolour plots of the Mach number in
Figure~\ref{fig:ML_vort_machnumpcol} which clearly shows the
independence of $M$ on $\veps$ compared to the initial Mach number
distribution given in Figure~\ref{fig:ML_vort_machnum_init}. In order
to further elucidate the idea, in Figure~\ref{fig:ML_vort_vorticity}
we plot the cross-section of the vorticity along the $x$-axis for the
chosen values of $\veps$ after one period of time and the plot of the
relative kinetic energies. It is evident that the vorticity distributions remain
independent of $\veps$ and the dissipation in relative kinetic energy
is only $0.07\%$. For this test problem, we note that $\eta\sim
0.013$. 

\begin{figure}[htbp]
  \centering
  \includegraphics[height=0.21\textheight]{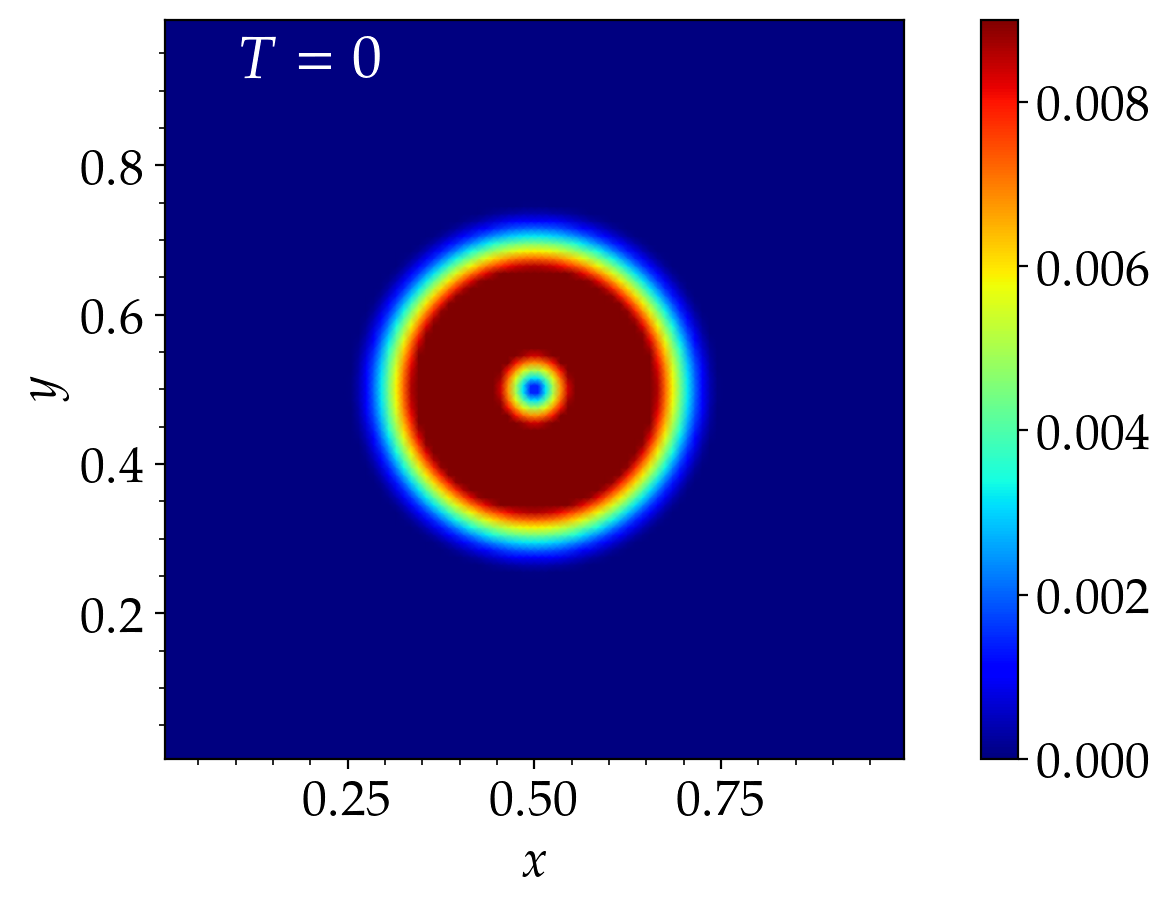}
  \caption{Pseudocolour plot of the Mach number at time $T=0$.} 
  \label{fig:ML_vort_machnum_init}
\end{figure}

\begin{figure}[htbp]
  \centering
  \begin{tabular}{c c c}
    \includegraphics[height=0.21\textheight]{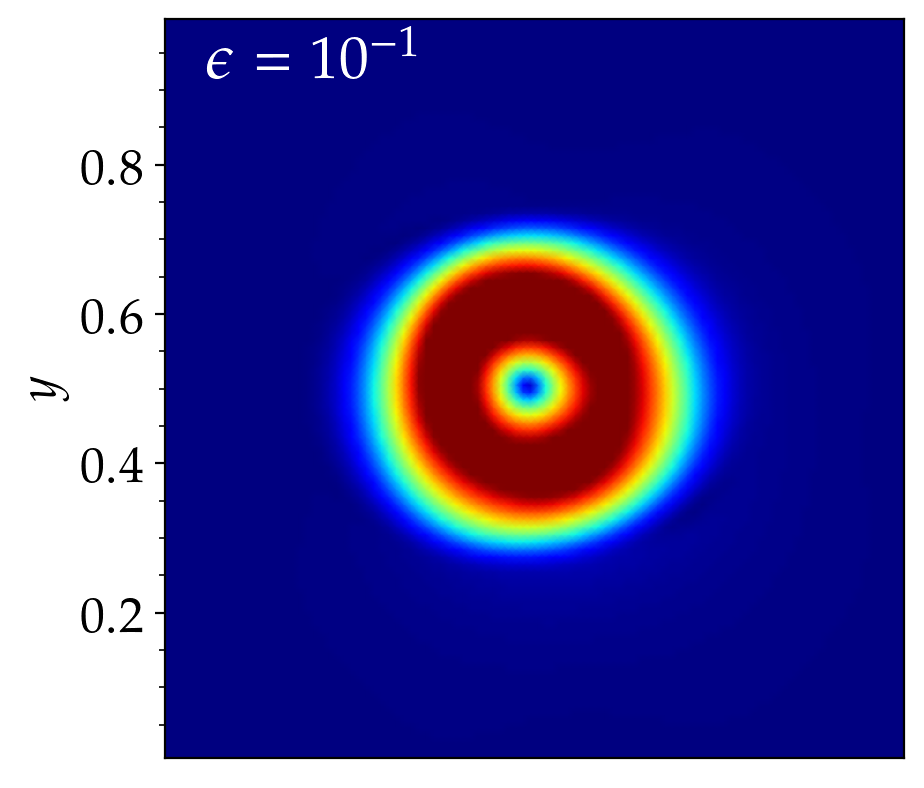} &
    \includegraphics[height=0.21\textheight]{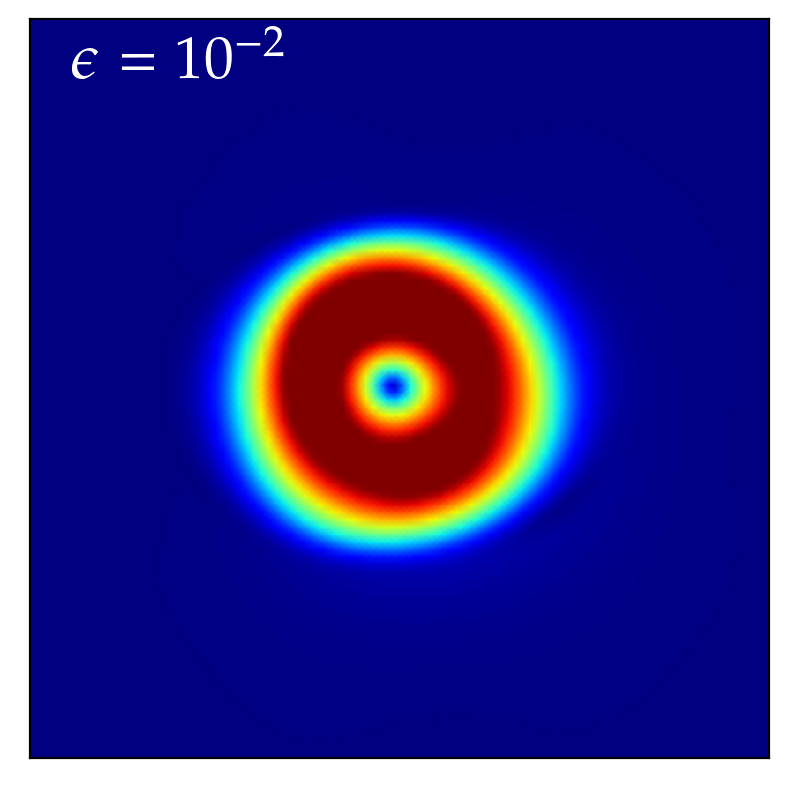} &
    \includegraphics[height=0.21\textheight]{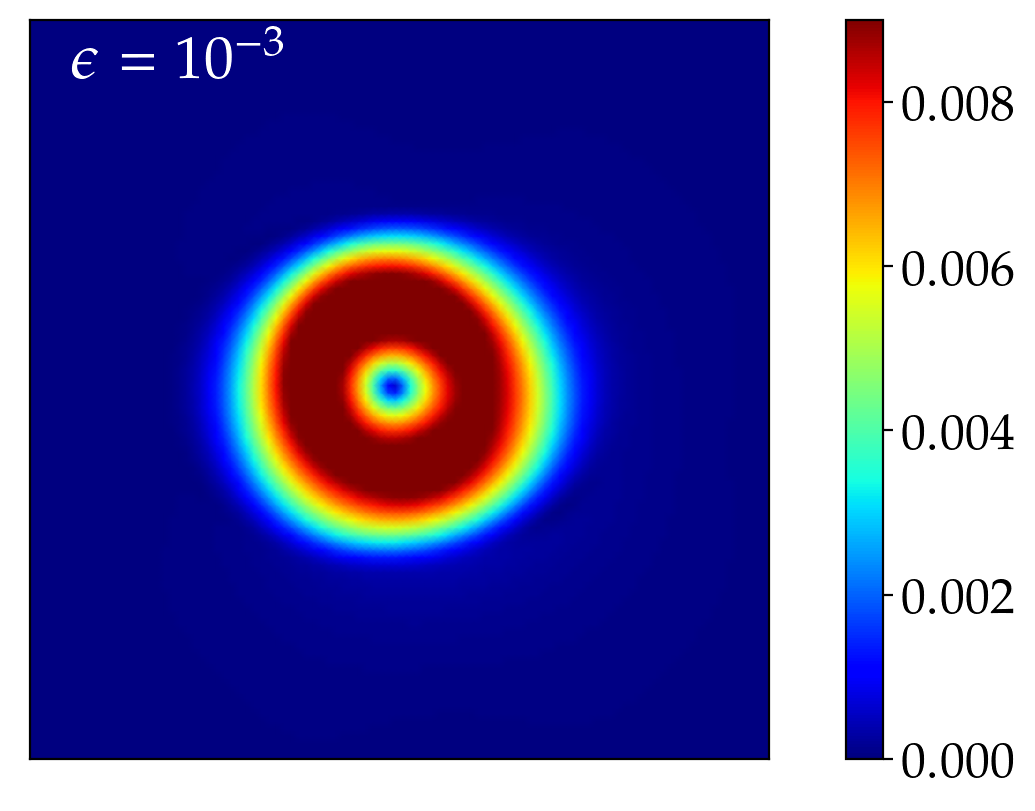} \\
    \includegraphics[height=0.24\textheight]{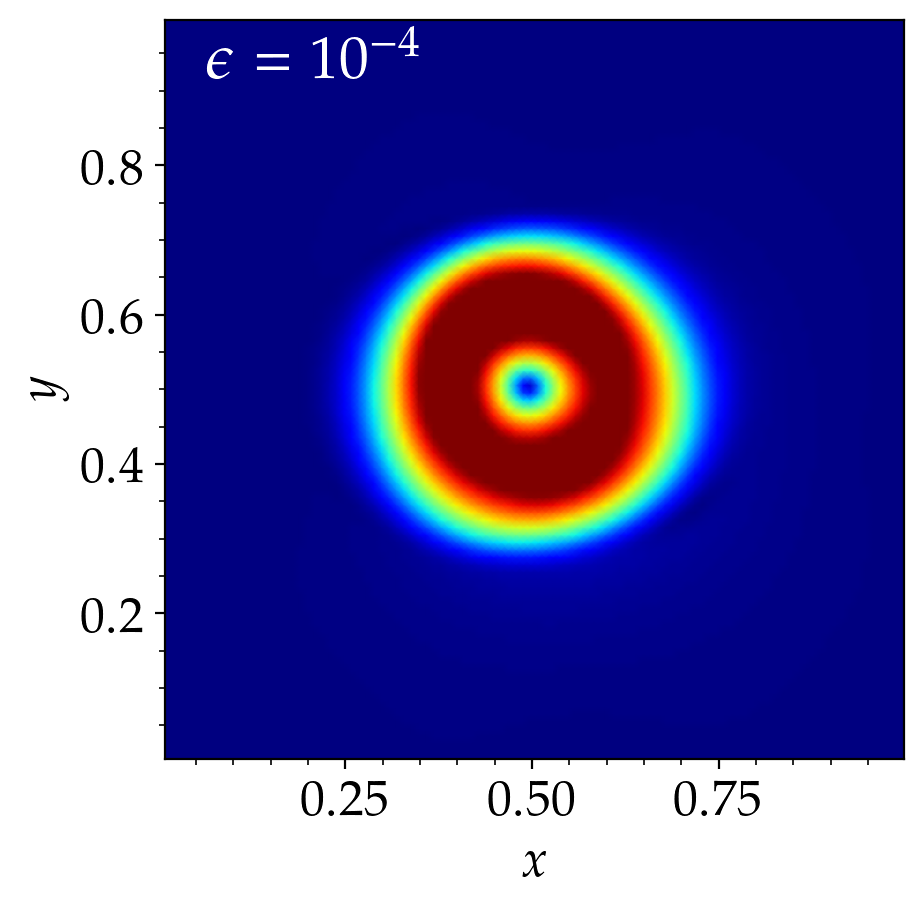} &
    \includegraphics[height=0.24\textheight]{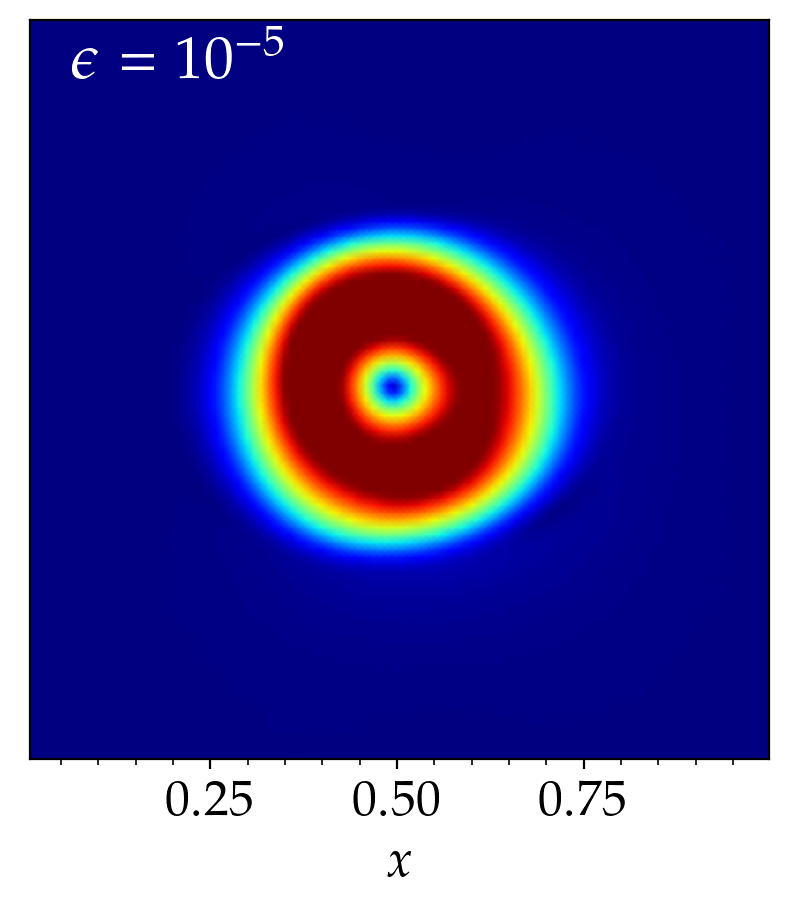} &
    \includegraphics[height=0.24\textheight]{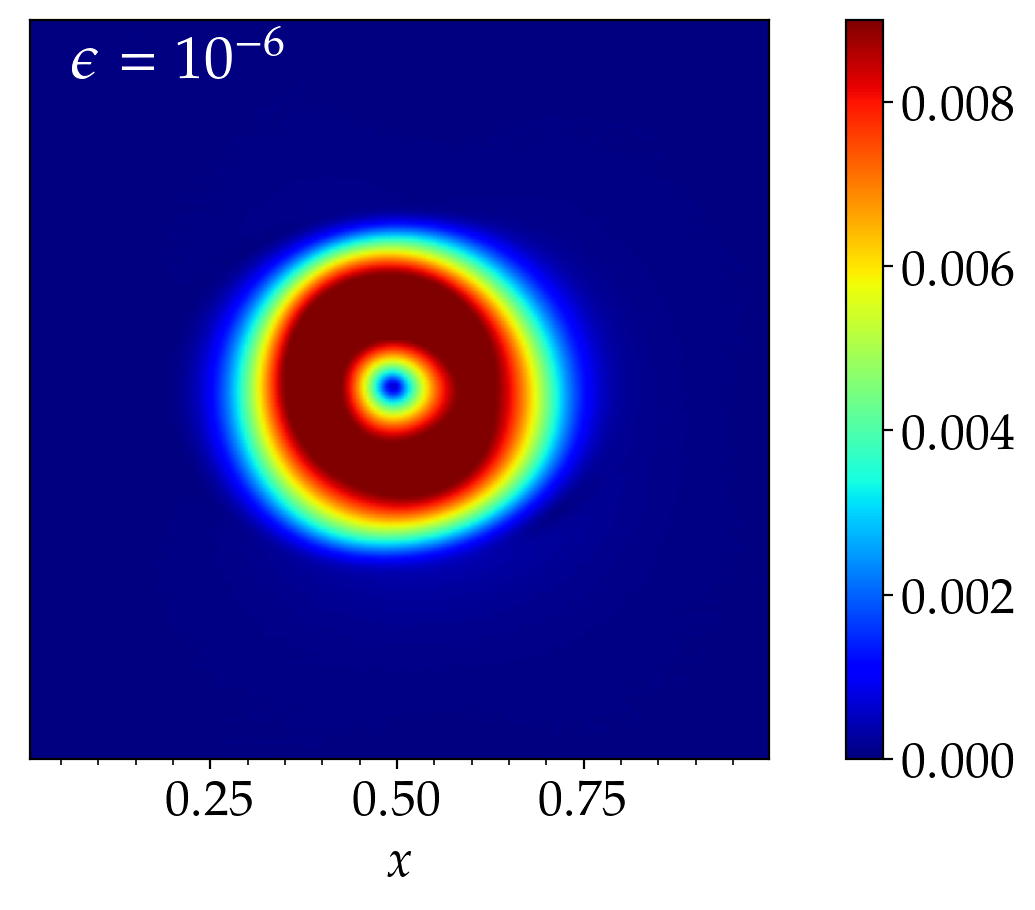}
  \end{tabular}
  \caption{Pseudocolour plots of the Mach number at time $T=5/3$ for
    $\veps = 10^{-1}, 10^{-2}, 10^{-3}, 10^{-4}, 10^{-5}, 10^{-6}$.} 
    \label{fig:ML_vort_machnumpcol}
\end{figure}

\begin{figure}[htbp]
  \centering
  \includegraphics[height=0.21\textheight]{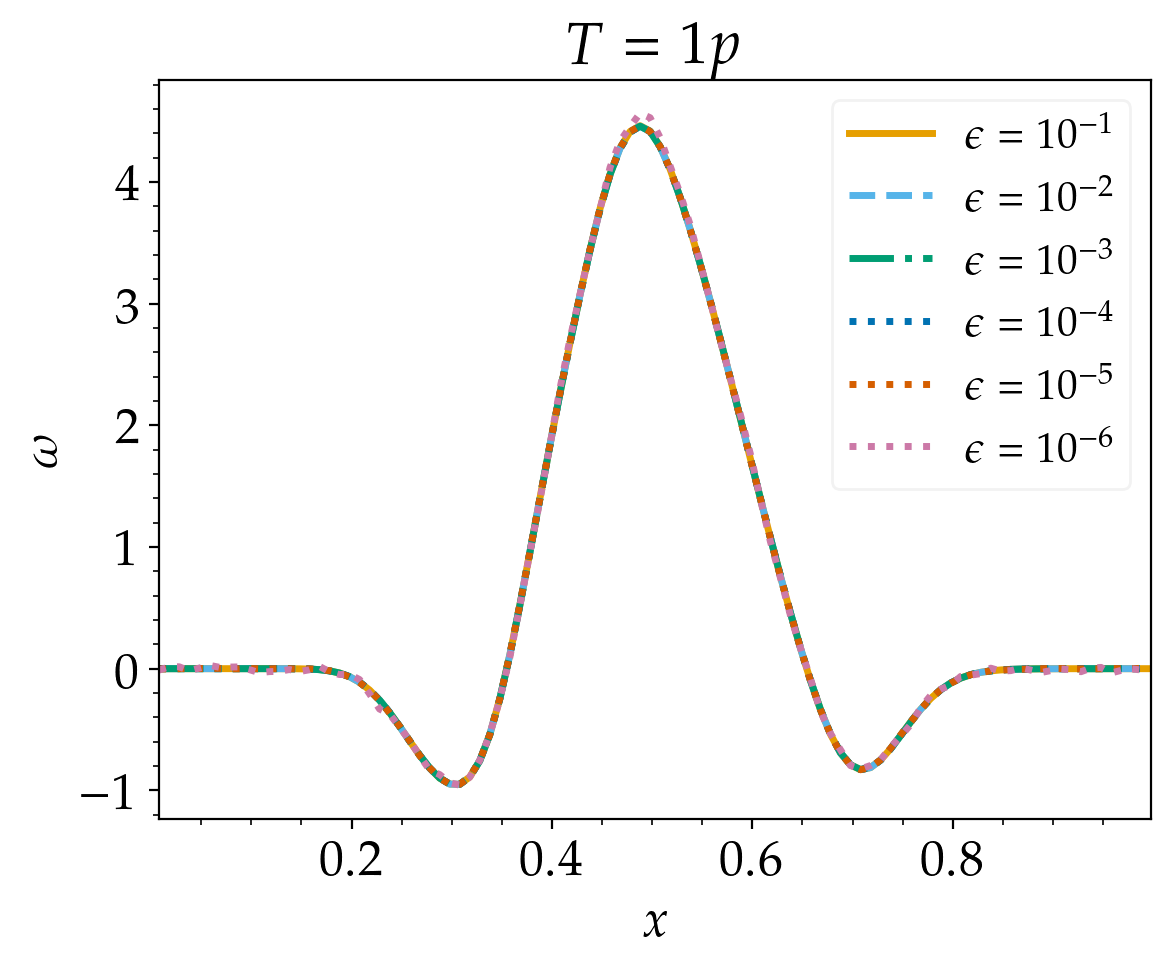}
  \includegraphics[height=0.21\textheight]{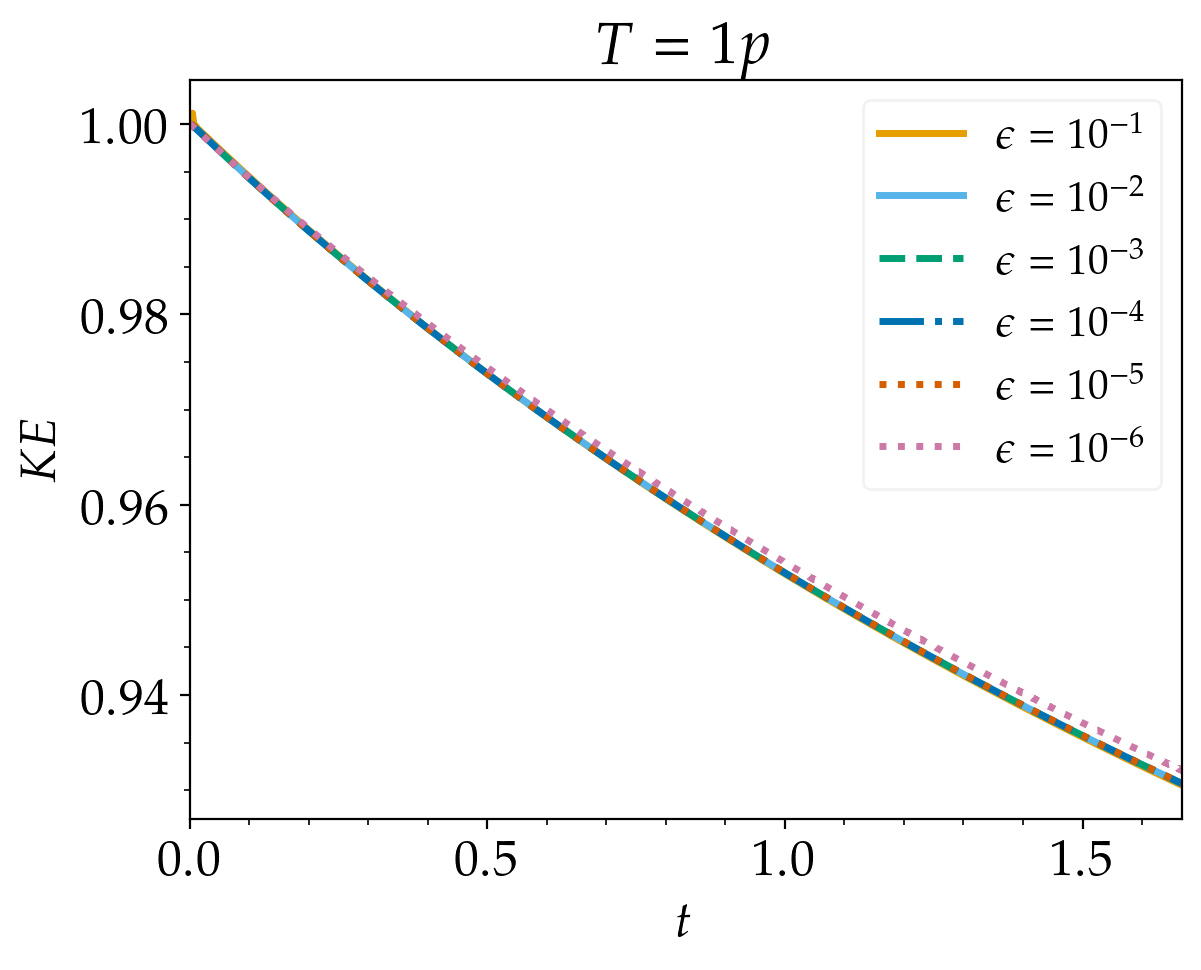}
  \caption{Cross-sections of the vorticity (left) and relative kinetic
    energy (right) for $\veps = 10^{-1}, 10^{-2}, 10^{-3}, 10^{-4},
    10^{-5},10^{-6}$.}   
  \label{fig:ML_vort_vorticity}
\end{figure}

\subsection{Cylindrical Explosion}
\label{sec:explosion}
Here we consider a 2D cylindrical explosion problem from \cite{DLV17}
with the isothermal pressure law $p(\rho)=\rho$, and the purpose of
this case study is twofold. First, we want to illustrate the shock
capturing capabilities of the present scheme in compressible regimes
$(\veps=1)$ and second, its ability to capture the incompressible
limit when $\veps\to0$. The initial density reads  
\begin{equation*}
    \rho(0,x,y) = 
    \begin{cases}
        1 + \veps^2, & \mbox{if} \ r^2\leq\frac{1}{4},
        \\
        1, & \mbox{otherwise},
    \end{cases}
\end{equation*}
where $r=\sqrt{x^2+y^2}$. The initial velocity field is taken as
\begin{equation*}
    (u,v)(0,x,y) =
    -\frac{\alpha(x,y)}{\rho(0,x,y)}\Big(\frac{x}{r},\frac{y}{r}\Big)\mcal{X}_{r>10^{-15}}, 
\end{equation*}
where $\alpha(x,y)=\max\{0,1-r\}(1-e^{-16r^2})$. We consider the
computational domain $[-1,1]\times [-1,1]$ and take $100\times100$
grid points. The boundaries are periodic everywhere. In
Figure~\ref{fig:cyl_exp_eps1p0_den} we present the surface plots of
the density which clearly indicates the scheme's shock capturing
capabilities. Also, we compare an 1D cut along the $x$-direction of the 
cylindrically symmetric density profile with a reference solution
obtained by an explicit Rusanov scheme with mesh resolution $\Delta x
= 1/500$, $\Delta t = 1/10000$ to validate the correctness of the
shock speed. Furthermore, since the problem is cylindrically
symmetric, we compare the radial velocity with the quasi-exact
reference solution obtained by the 1D Euler equations in radial
coordinates \cite{Tor09} on $200$ mesh points in
Figure~\ref{fig:cyl_exp_eps1p0_vel}. 

\begin{figure}[htbp]
  \centering
  \begin{tabular}{ccc}
  \includegraphics[height=0.18\textheight]{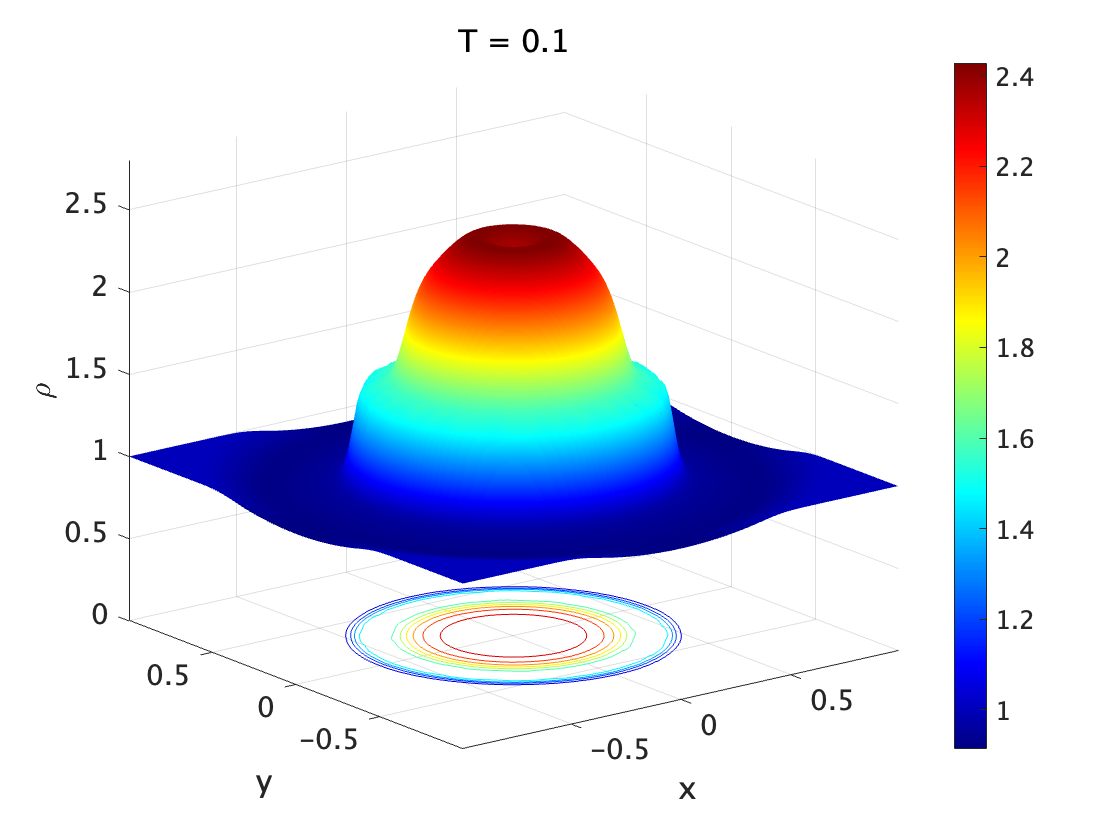} &
  \includegraphics[height=0.18\textheight]{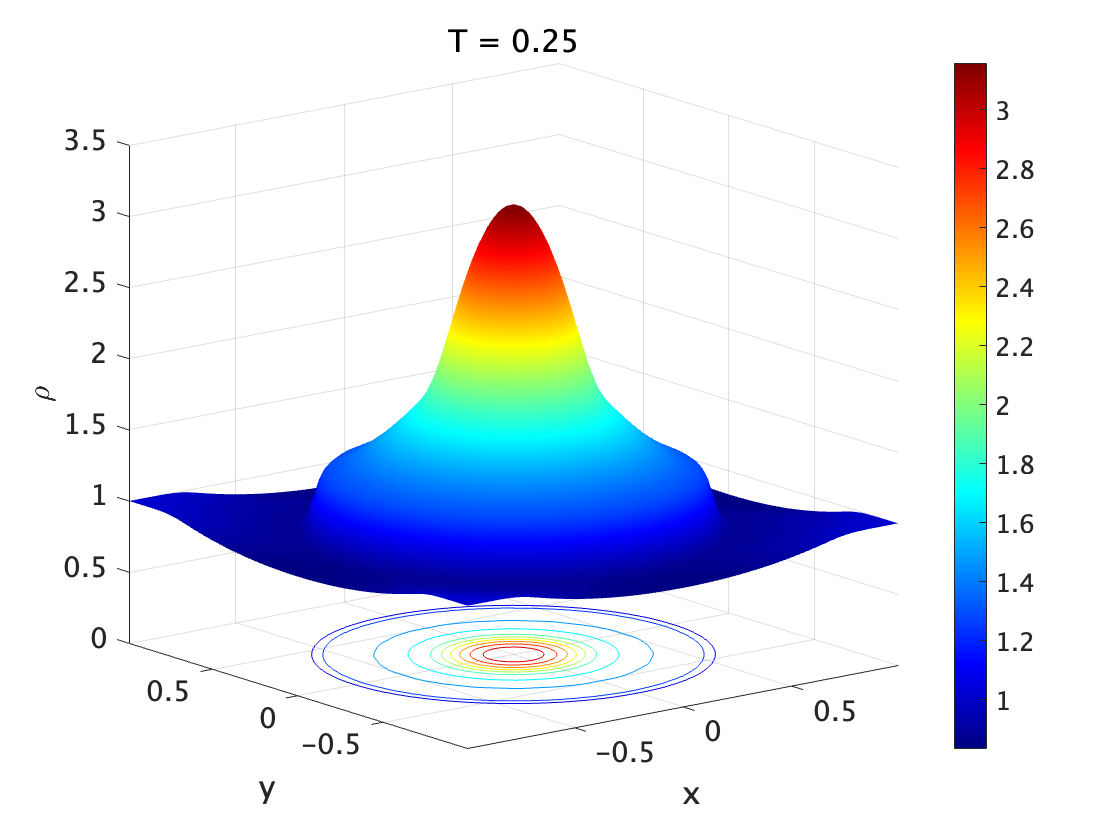} &
  \includegraphics[height=0.18\textheight]{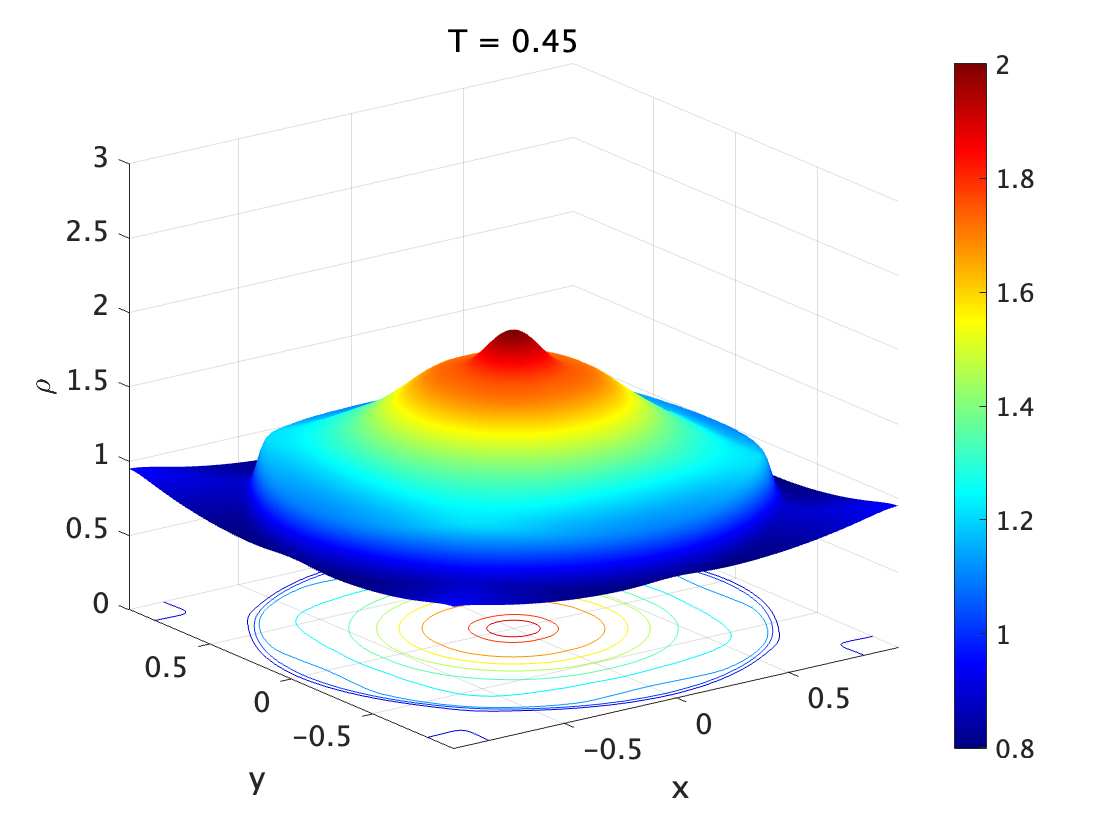} \\
  \includegraphics[height=0.18\textheight]{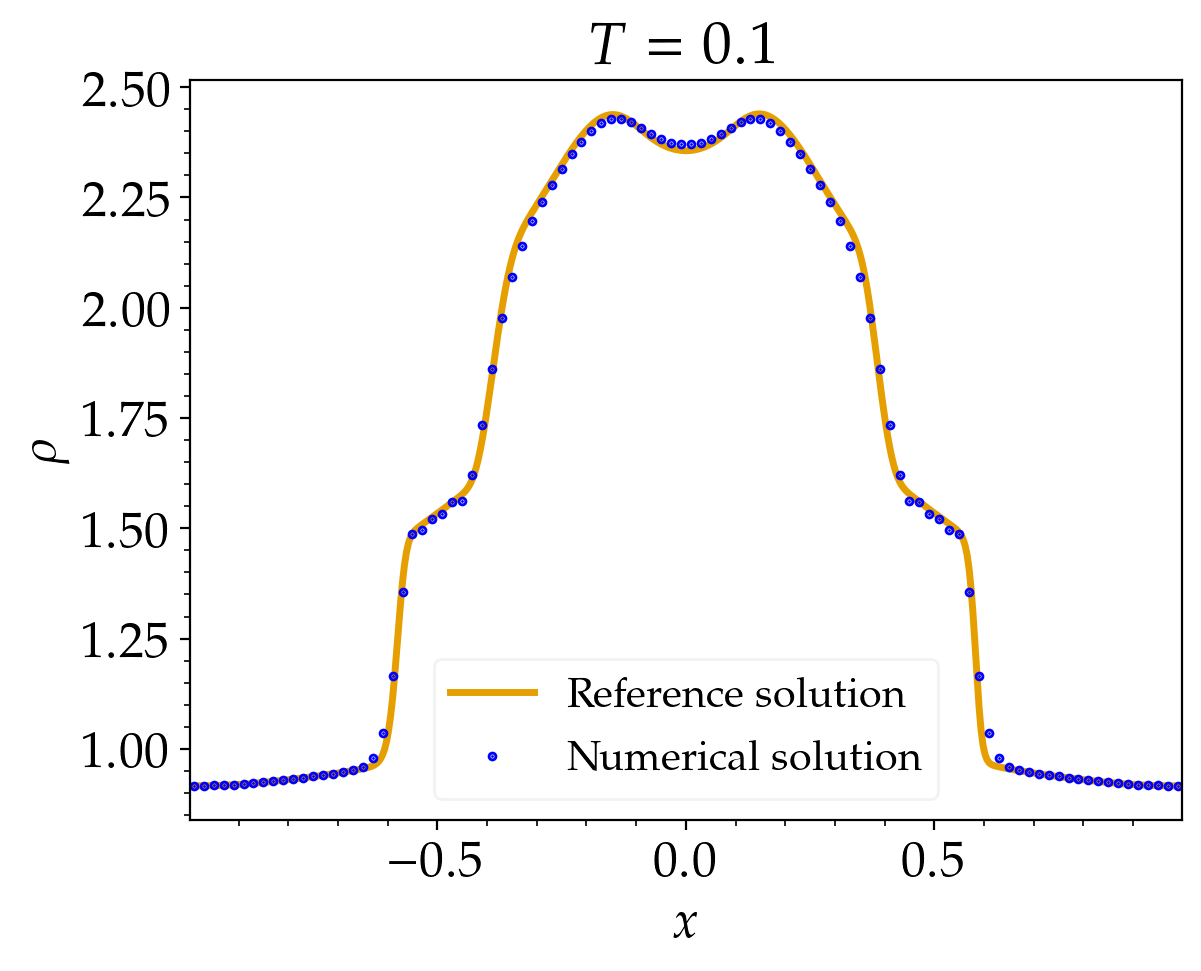} &
  \includegraphics[height=0.18\textheight]{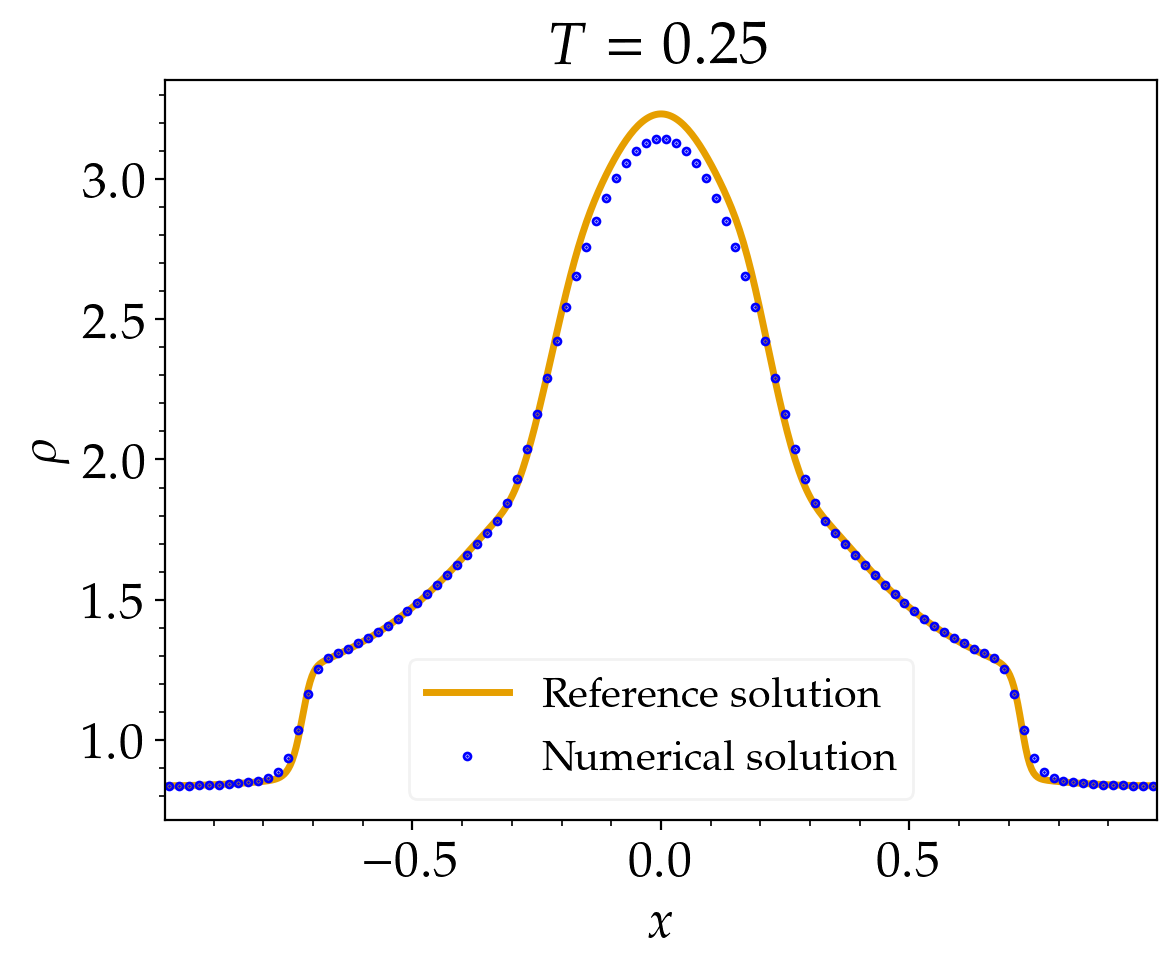} &
  \includegraphics[height=0.18\textheight]{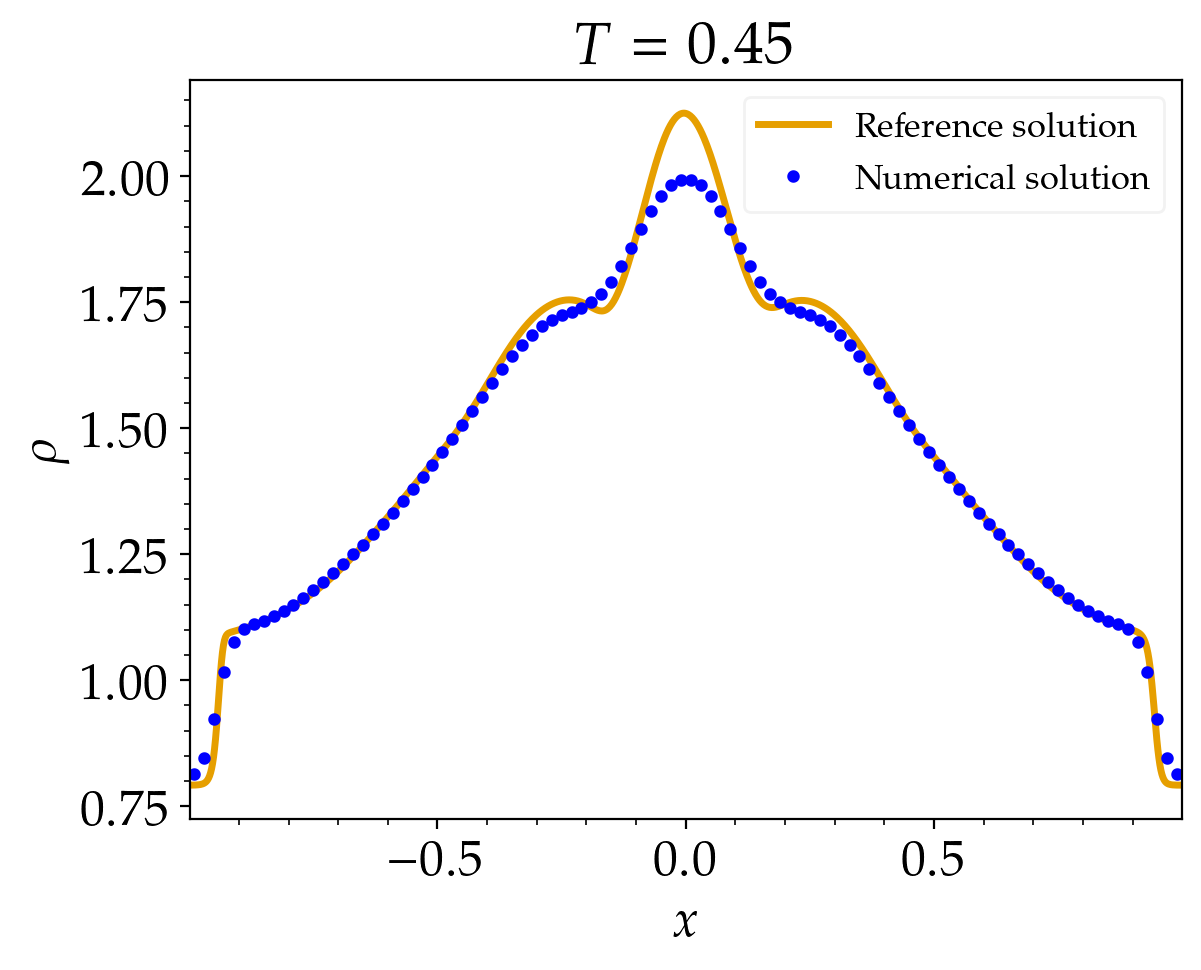}
  \end{tabular}
  \caption{Surface plots of the density (top) and comparison of cuts
    along the $x$-direction (bottom) for the cylindrical explosion
    problem at times $T=0.1,0.25,0.45$ for $\veps=1$.} 
  \label{fig:cyl_exp_eps1p0_den}
\end{figure}
\begin{figure}[htbp]
  \centering
  \includegraphics[height=0.18\textheight]{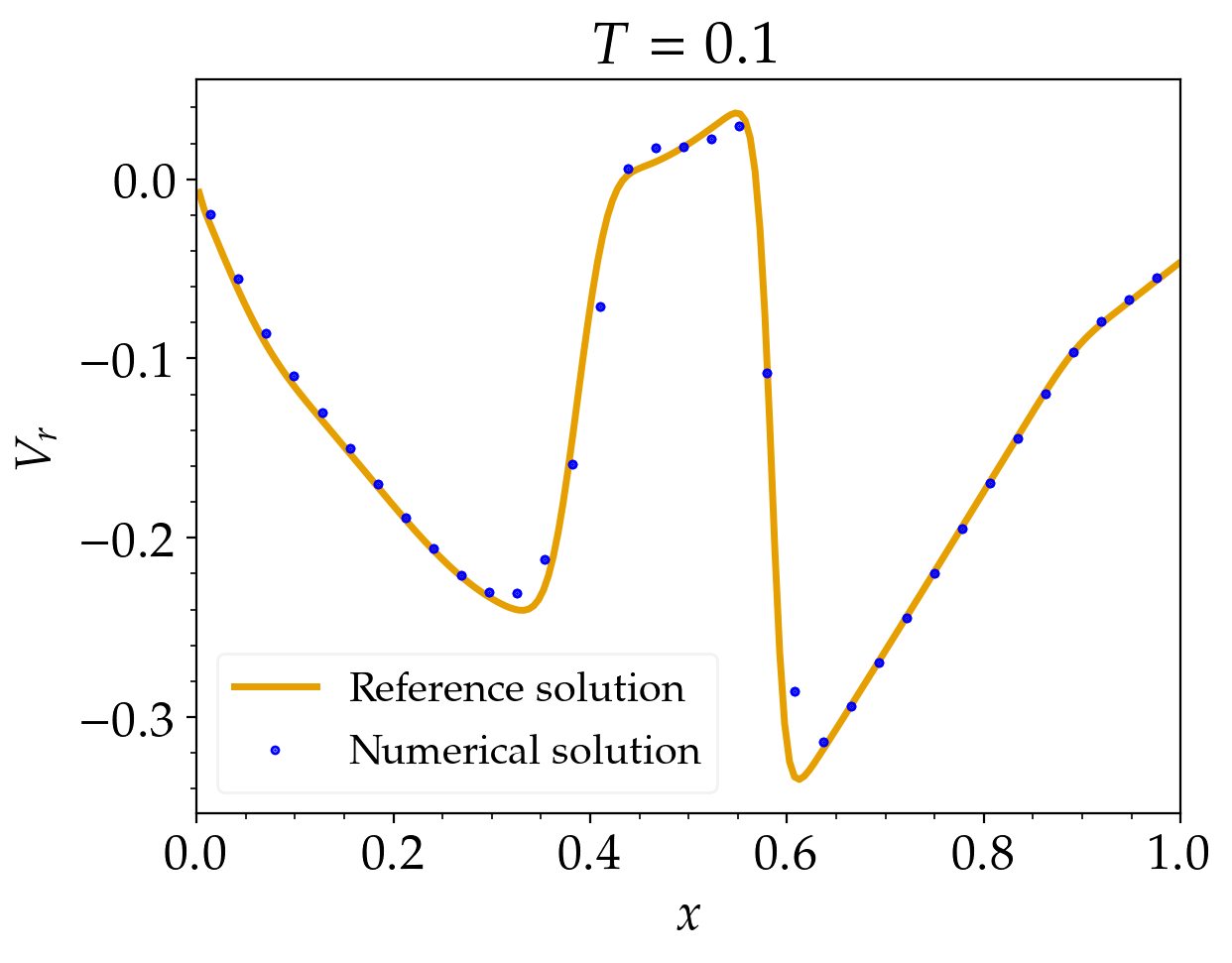}
  \includegraphics[height=0.18\textheight]{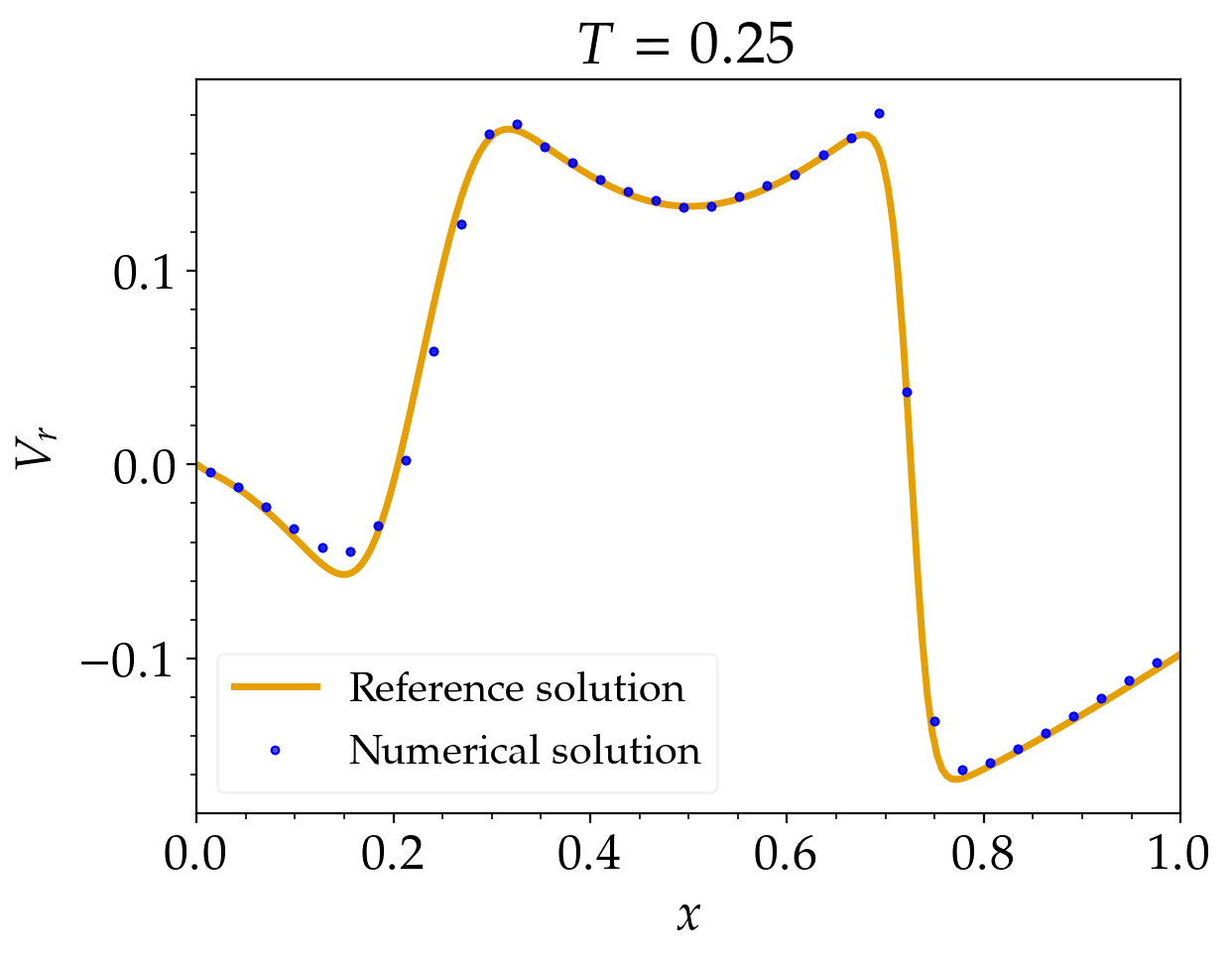}
  \includegraphics[height=0.18\textheight]{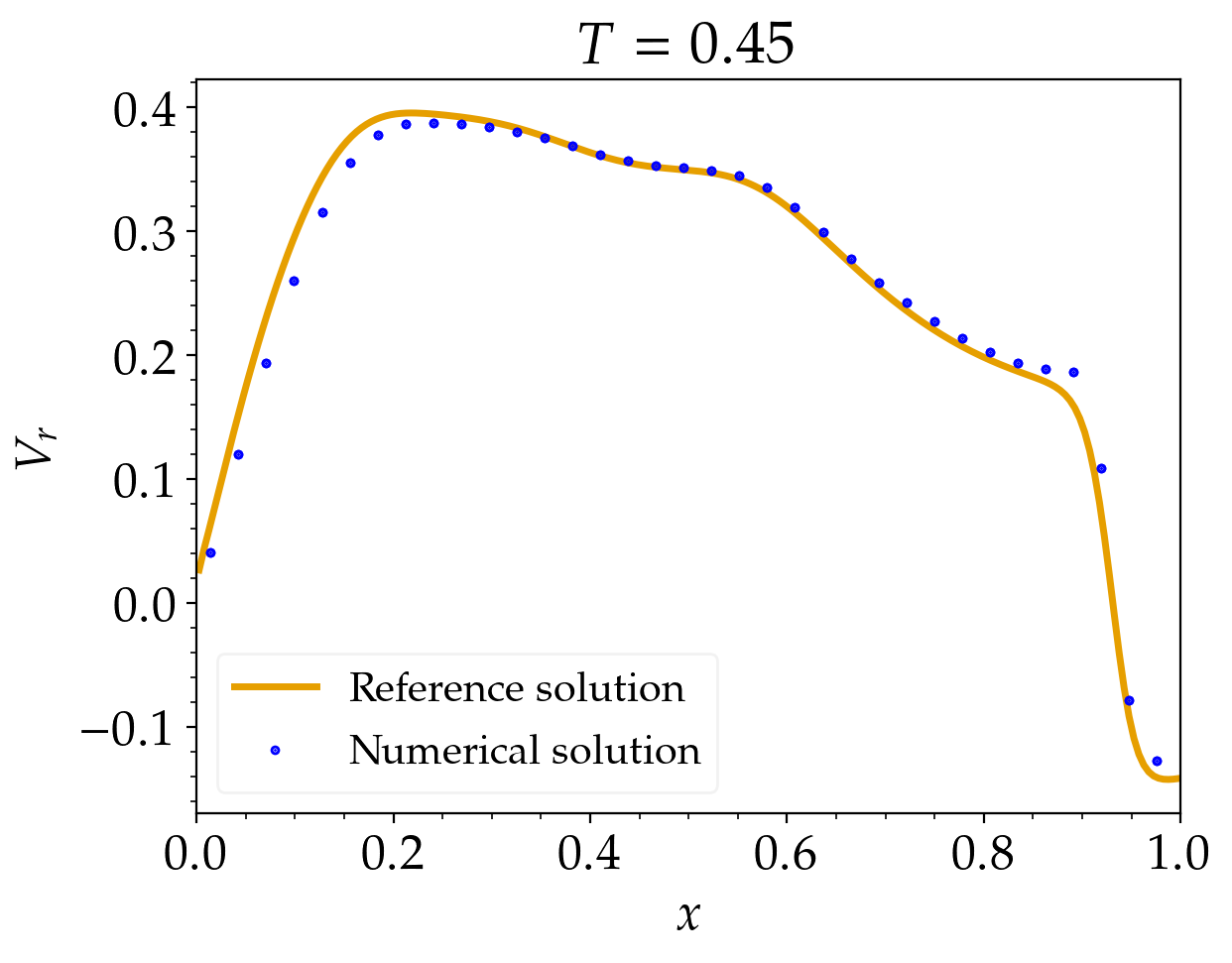}
  \caption{1D cuts of the radial velocity compared to a quasi-exact
    radial solution for the cylindrical explosion problem at times
    $T=0.1,0.25,0.45$ for $\veps=1$.} 
  \label{fig:cyl_exp_eps1p0_vel}
\end{figure}
Finally, in Figure~\ref{fig:cyl_exp_eps0p0001_den} we give the surface plot of
the density deviation from the constant state $1$ computed at time
$T=0.05$ in the incompressible regime for $\veps= 10^{-4}$. We clearly
note that the numerical solution has converged to a constant density
$\rho=1$ with an error $10^{-9}$ and the divergence errors are of the
order $10^{-4}$. The minimum and maximum values of $\eta$ observed for
this problem are $1.5$ and $6.6$, respectively.  
\begin{figure}[htbp]
  \centering
  \includegraphics[height=0.19\textheight]{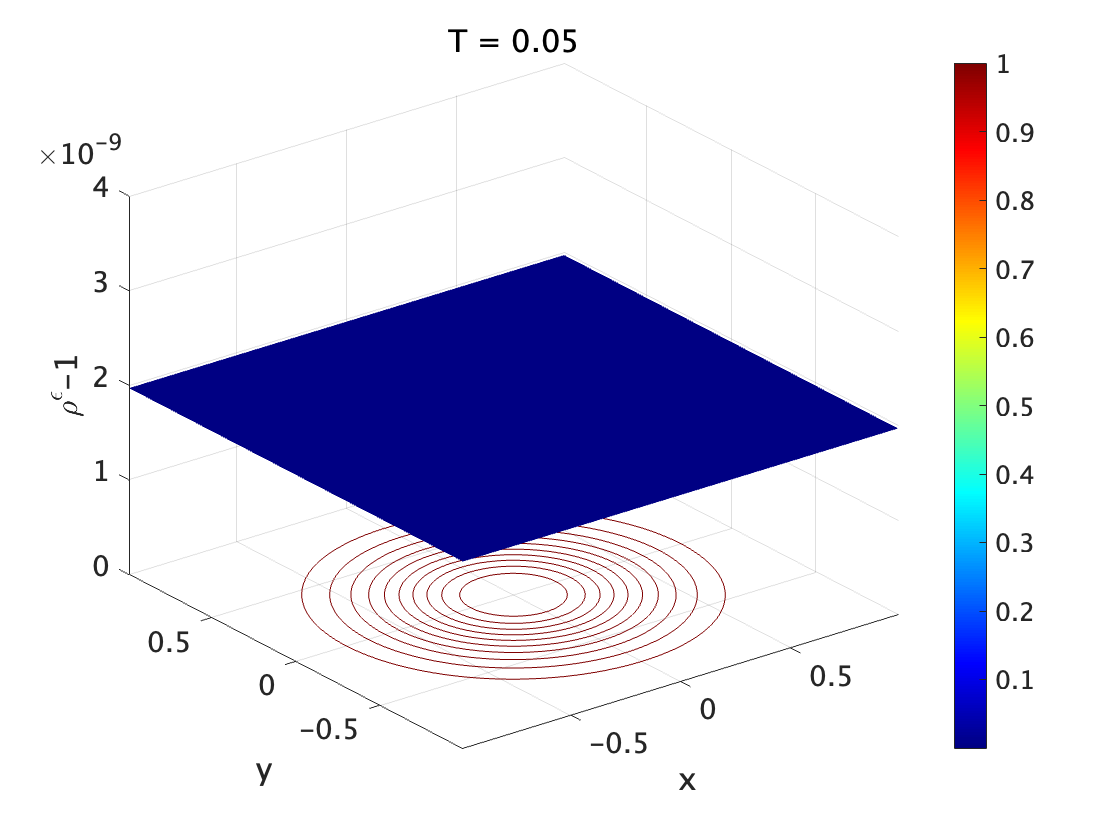}
  \includegraphics[height=0.19\textheight]{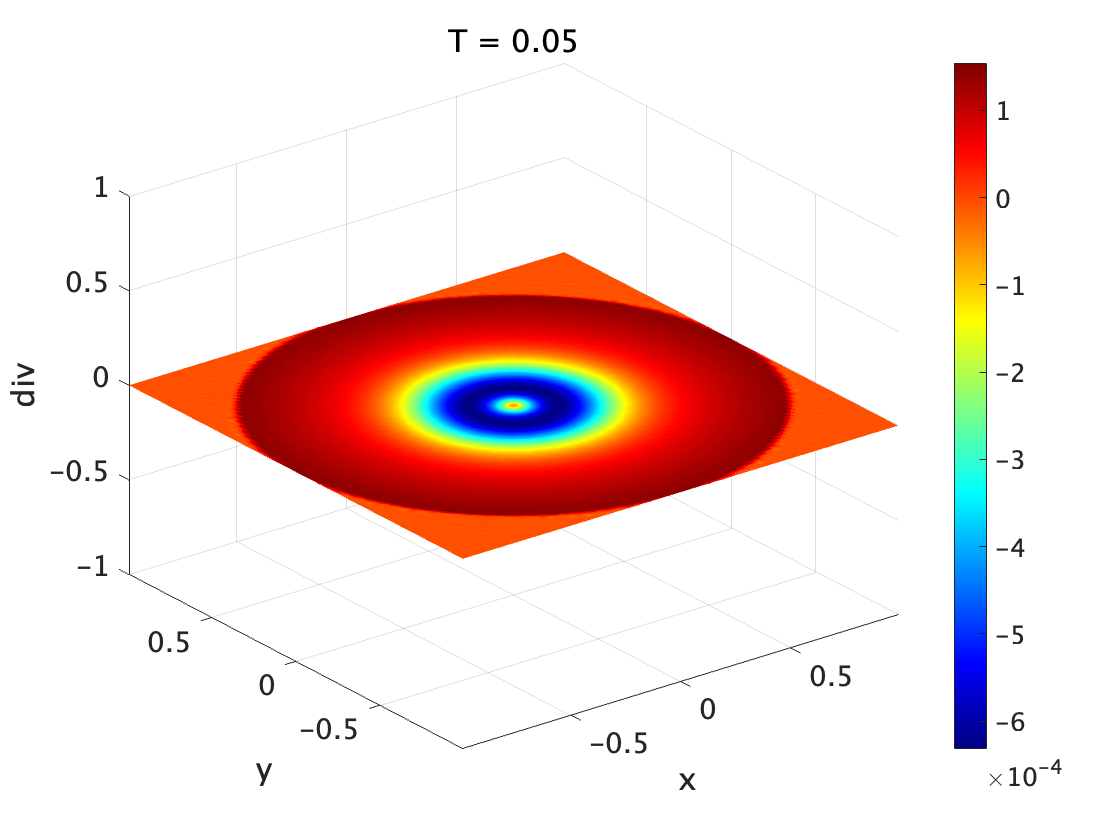}
  \caption{Surface plots of the deviation of density from $1$ and the
    divergence of the velocity for the cylindrical explosion problem
    at time $T=0.05$ for $\veps= 10^{-4}$.}  
  \label{fig:cyl_exp_eps0p0001_den}
\end{figure}

\subsection{Incompressible Problem}
\label{sec:incomp}

We consider the following incompressible initial data from \cite{ES94}
wherein  
\begin{equation*}
  \rho(0,x,y) = 1, \; u(0,x,y) = -\sin x\cos y, \; v(0,x,y) = \cos x\sin y.
\end{equation*}
The pressure law is taken as $p(\rho)=\rho^2$. For this problem an
exact solution is available. We take the computational domain
$[0,2\pi]\times[0,2\pi]$ which is divided into $32 \times 32$ mesh
points. In Figure~\ref{fig:incomp_vort_comp} we compare the exact
vorticity and the numerical vorticity at time $T=2$ for $\veps =
0.01$. The figure clearly indicates the convergence of the numerical
solution to the incompressible solution even on a coarse mesh. In
Table~\ref{tab:incomp1_vort} the relative errors between the exact and 
computed vorticities along with the rates of convergence are
tabulated. We clearly notice first order convergence in all three
norms $L^1,L^2$ and $L^\infty$. We observe $\eta\sim 1.5$ for this
test case. 

\begin{figure}[htbp]
    \centering
    \includegraphics[height=0.21\textheight]{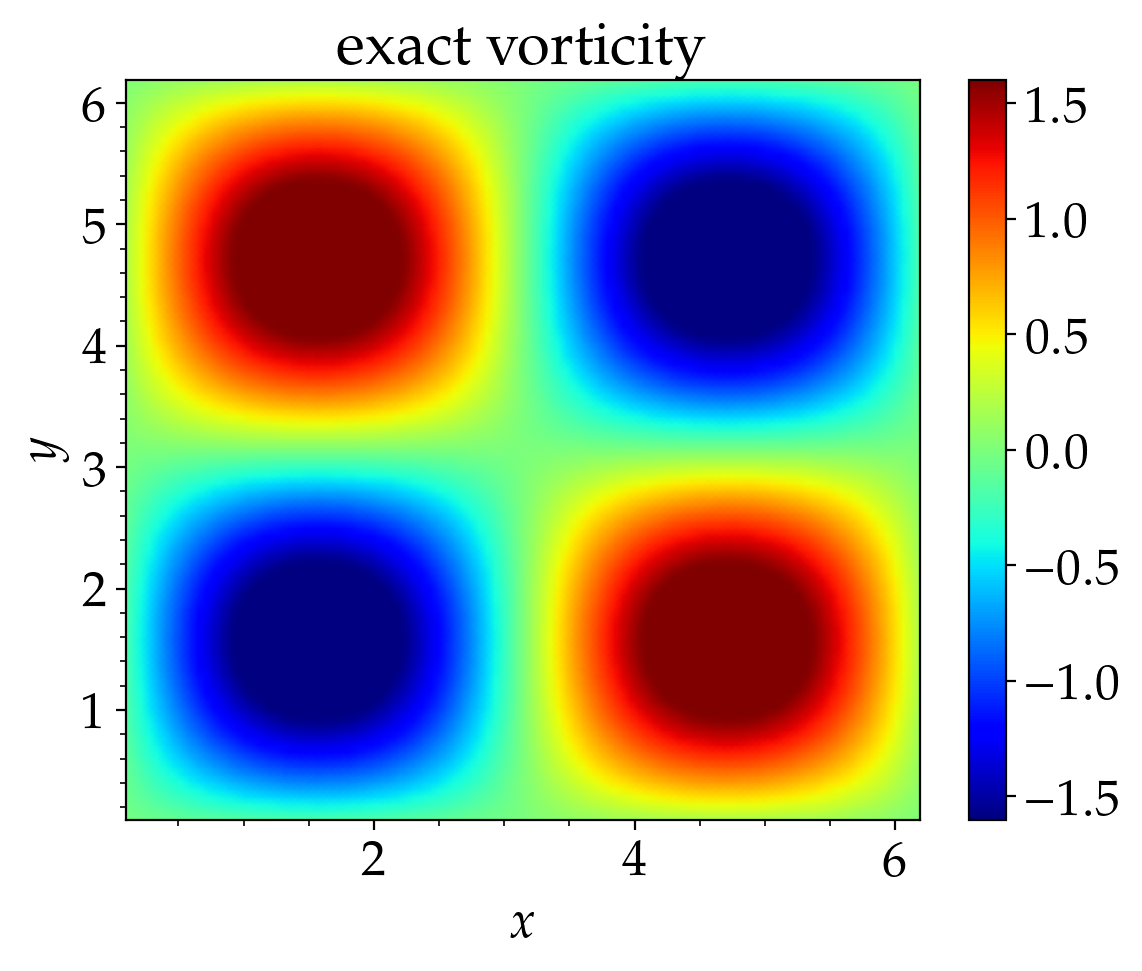}
    \includegraphics[height=0.21\textheight]{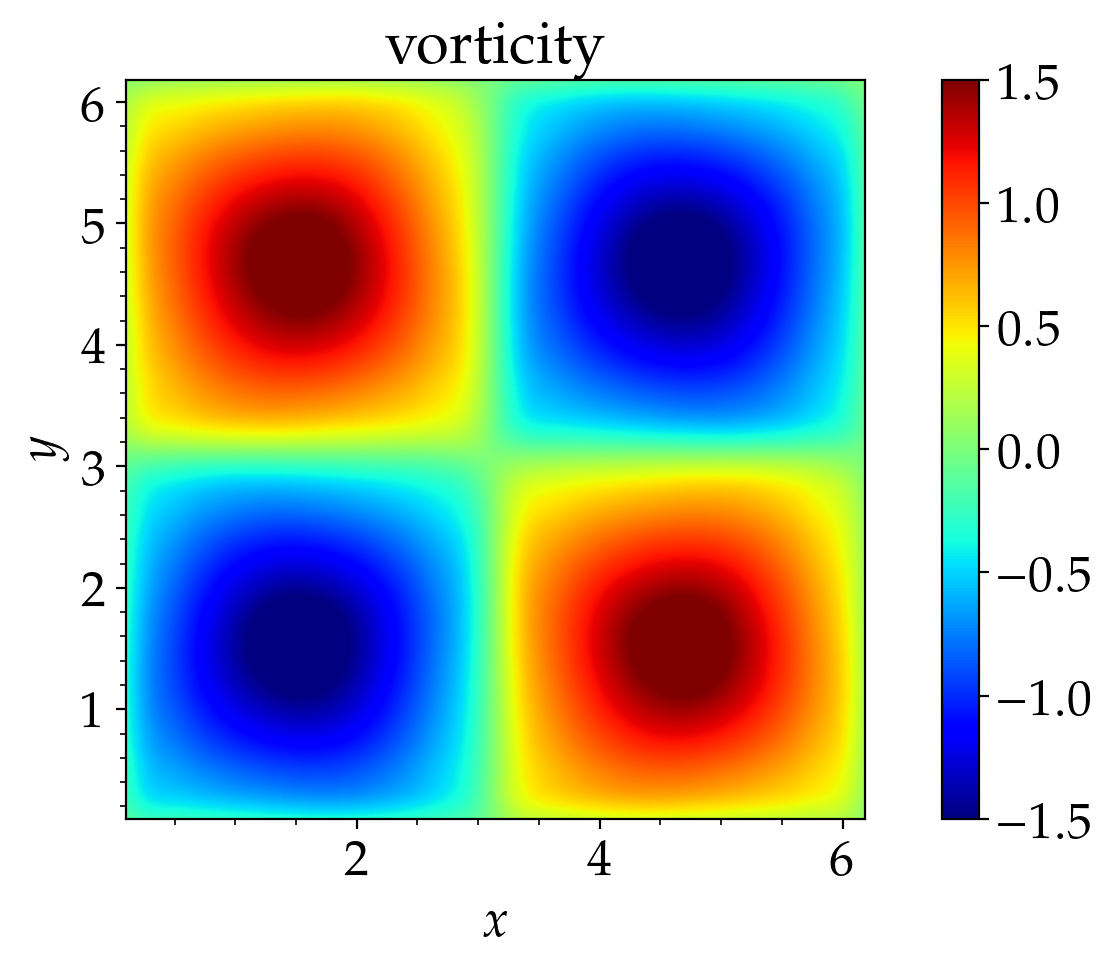}
    \caption{Comparison of the exact and numerical vorticity for the
      incompressible problem for $\veps = 0.01$ at time $T=2$.} 
    \label{fig:incomp_vort_comp}
\end{figure}

\begin{table}[htbp]
    \centering
    \begin{tabular}{|c|c|c|c|c|c|c|}
    \hline
      Grid & $L^1$ error & Rate & $L^2$ error & Rate & $L^\infty$ error & Rate \\
    \hline
      $16\times 16$   &  $0.306722$  &  -                & $0.323460$
                                              &  -   & $0.374421$  &  - \\
      \hline
      $32\times 32$   &  $0.190525$  &  $0.68688$  & $0.192703$  &  $0.747172$ &
                                                                                 $0.202888$
                                                                        &
                                                                          $0.883933$ \\ 
      \hline
      $64\times 64$   &  $0.120086$  &  $0.665848$ & $0.114076$  &  $0.756426$
                                                     &  $0.111018$
                                                                        & $0.869871$ \\ 
    \hline
      $128\times 128$ &  $0.068017$  & $0.820077$ & $0.063206$  & $0.851839$ &
                                                                     $0.058597$  & $0.921893$\\ 
      \hline
      $256\times256$  &  $0.036152$  & $0.911807$ & $0.033302$  & $0.924480$ &
                                                                     $0.030259$  & $0.953451$ \\ 
      \hline
    \end{tabular}
    \caption{The relative errors in vorticity and rates of convergence
      for the incompressible problem in $L^1,L^2$ and $L^\infty$ norms
      for $\veps=0.01$.}  
    \label{tab:incomp1_vort}
\end{table}

\subsection{Shear Flow Problem}
\label{sec:shear_flow}
We consider the shear flow problem from \cite{ES94} with the following
initial data: 
\begin{equation*}
  \rho(0,x,y) = \pi/15, \;  
  u(0,x,y) = \begin{cases}
    \tanh \frac{y-\frac{\pi}{2}}{\pi/15}, & \mbox{if} \ y\leq \pi,
    \\
    \tanh \frac{\frac{3\pi}{2}-y}{\pi/15}, & \mbox{if} \ y> \pi,
  \end{cases} \;
  v(0,x,y) = 0.05\sin x. 
\end{equation*}
The goal of this test problem is two-fold. First, we show the
convergence of the numerical densities to the constant
incompressible density as $\veps \to 0$. To this end, we calculate the 
differences between the computed density and the constant value
$\pi/15$ in $L^2$-norm and tabulate them against $\veps$ in
Table~\ref{tab:sl_rho}. We notice that the errors decay as
$\mcal{O}(\veps^2)$ as $\veps\to 0$ in accordance with the theoretical
results \cite{KM82}. Second, in order to establish numerically the
AP property, we simulate this problem in an incompressible regime by
setting a very small value $\veps = 10^{-6}$. The computational domain
$[0,2\pi]\times[0,2\pi]$ is divided into $256\times 256$ mesh
points. The plots of the vorticity obtained using the present scheme
and the one computed using the numerical solution of the limiting
incompressible scheme
\eqref{eq:dis_incomp_mas}-\eqref{eq:dis_incomp_mom} at times $T=4, 6,
8, 10$ are presented in Figure~\ref{fig:shear_flow_vorticity} and
Figure~\ref{fig:shear_flow_lim_vorticity}, respectively. The plots
clearly show that the scheme can resolve complex vortex structures of
the incompressible flow very well and the numerical solution is
indistinguishable from that of the limiting scheme. The value
$\eta\sim 7.16$ is observed for this test case. 
\begin{table}[htbp]
    \centering
    \begin{tabular}{|c|c|c|c|c|c|c|}
    \hline
    $\veps$ & $10^{-1}$ & $10^{-2}$ & $10^{-3}$ & $10^{-4}$ & $10^{-5}$ & $10^{-6}$ \\
    \hline
    $\norm{\rho^\veps-\pi/15}_{L^2} = 0.6\times$ &  $10^{-2}$  &  $10^{-4}$  &  $10^{-6}$  &  $10^{-8}$ & $10^{-10}$ & $10^{-12}$ \\
    \hline
    \end{tabular}
    \caption{The $L^2$ errors in density for the shear flow problem.} 
    \label{tab:sl_rho}
\end{table} 
\begin{figure} 
  \centering
  \includegraphics[height=0.18\textheight]{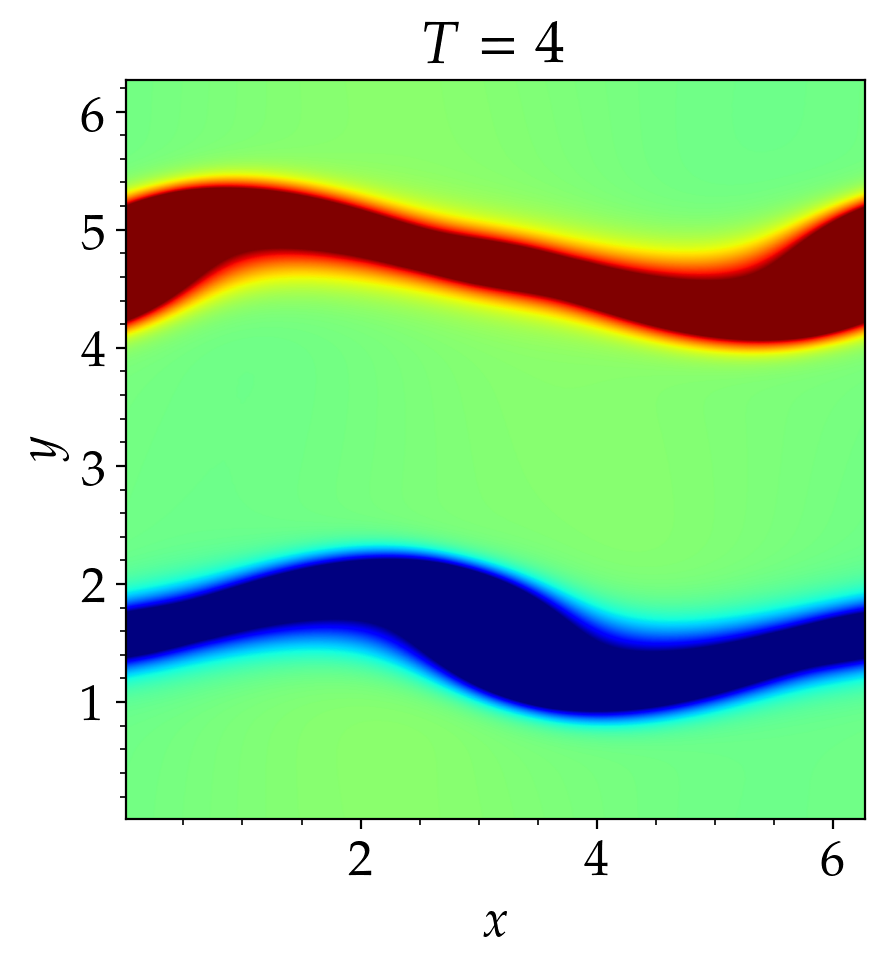}
  \includegraphics[height=0.18\textheight]{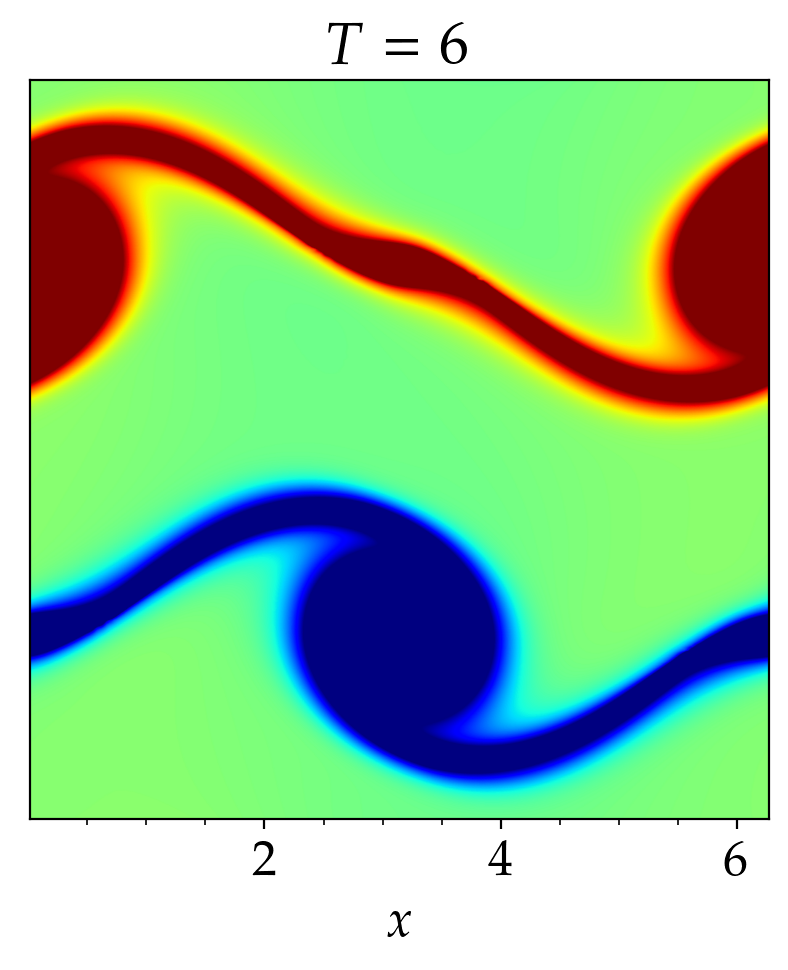}
  \includegraphics[height=0.18\textheight]{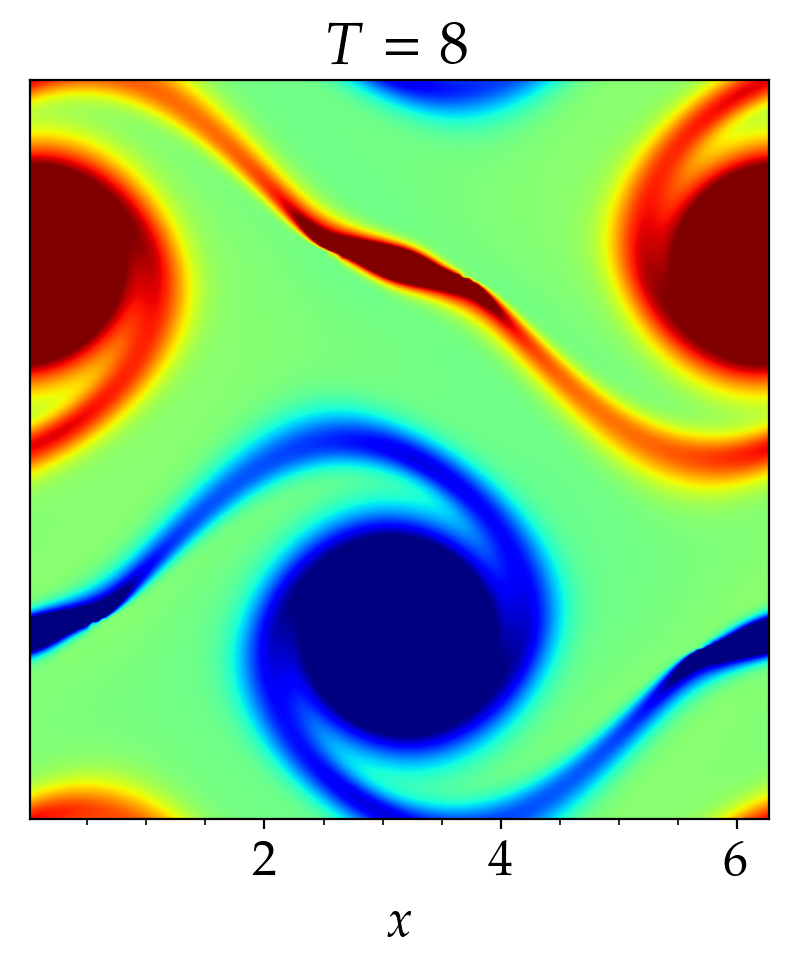}
  \includegraphics[height=0.18\textheight]{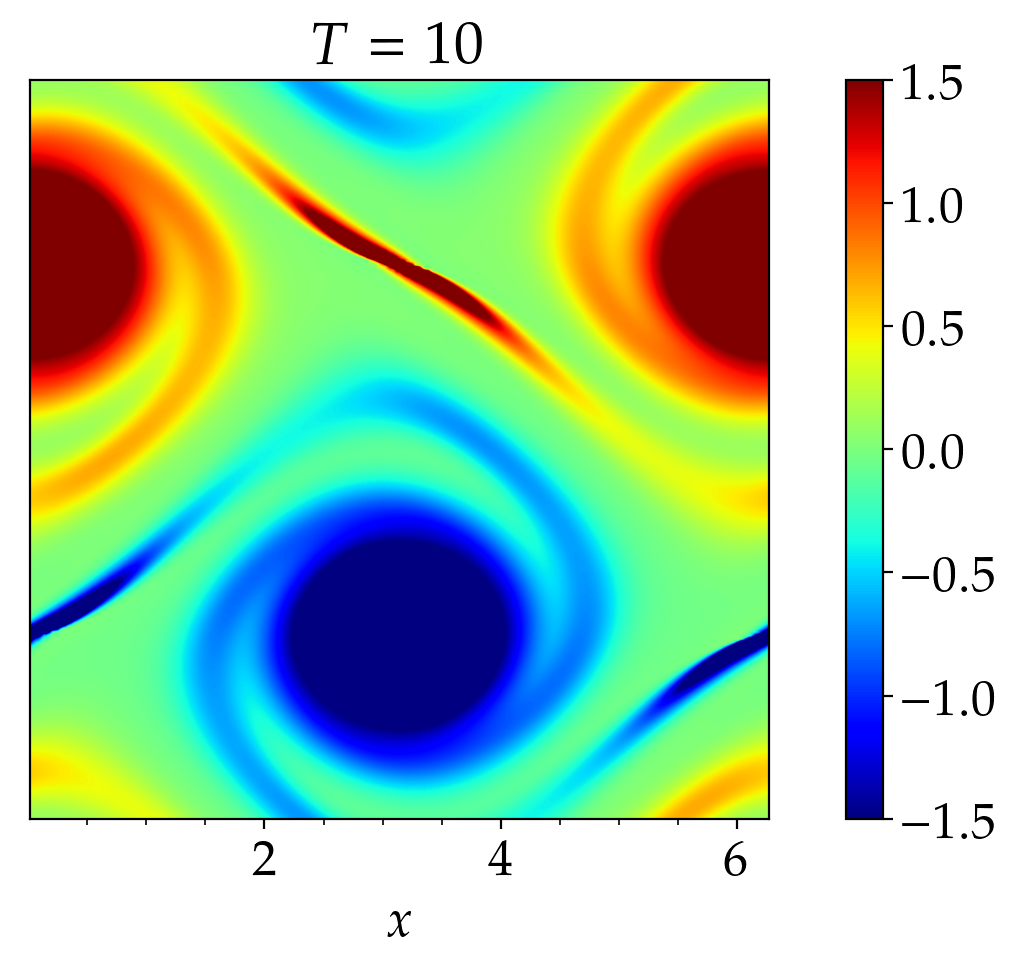}
  \caption{Vorticity for $\veps = 10^{-6}$ at times $T= 4, 6, 8$ and $10$
    for the shear flow problem.} 
  \label{fig:shear_flow_vorticity}
\end{figure}

\begin{figure} 
  \centering  
  \includegraphics[height=0.18\textheight]{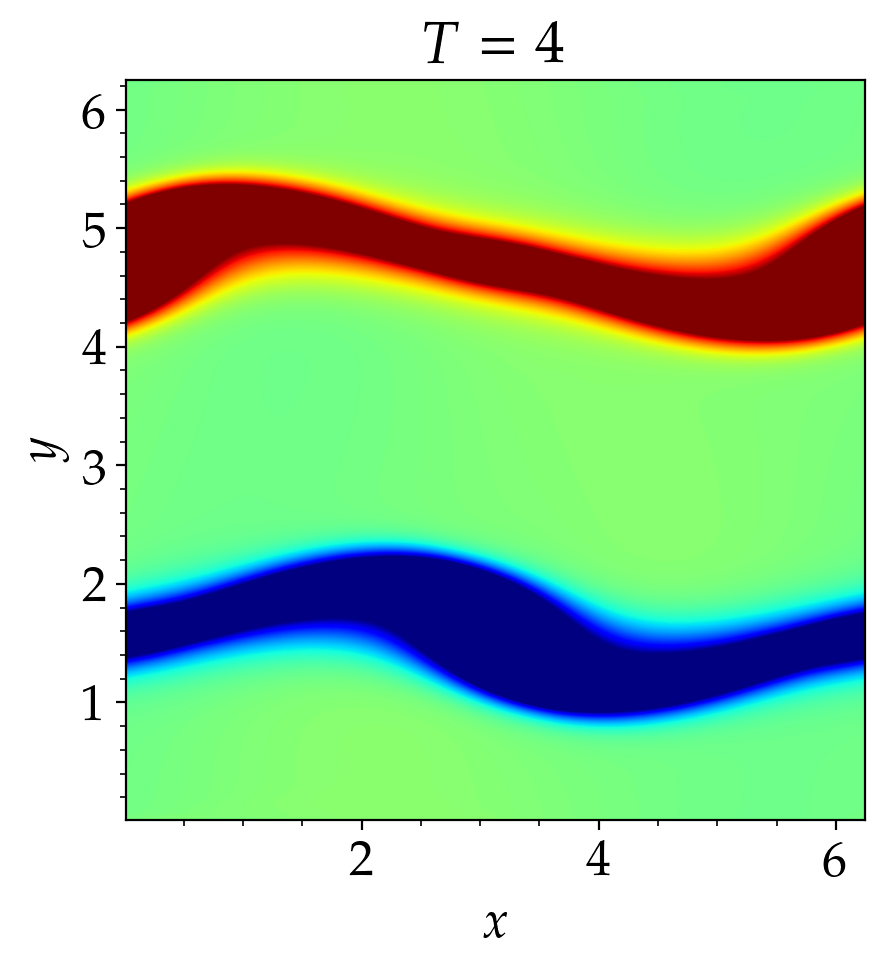}
  \includegraphics[height=0.18\textheight]{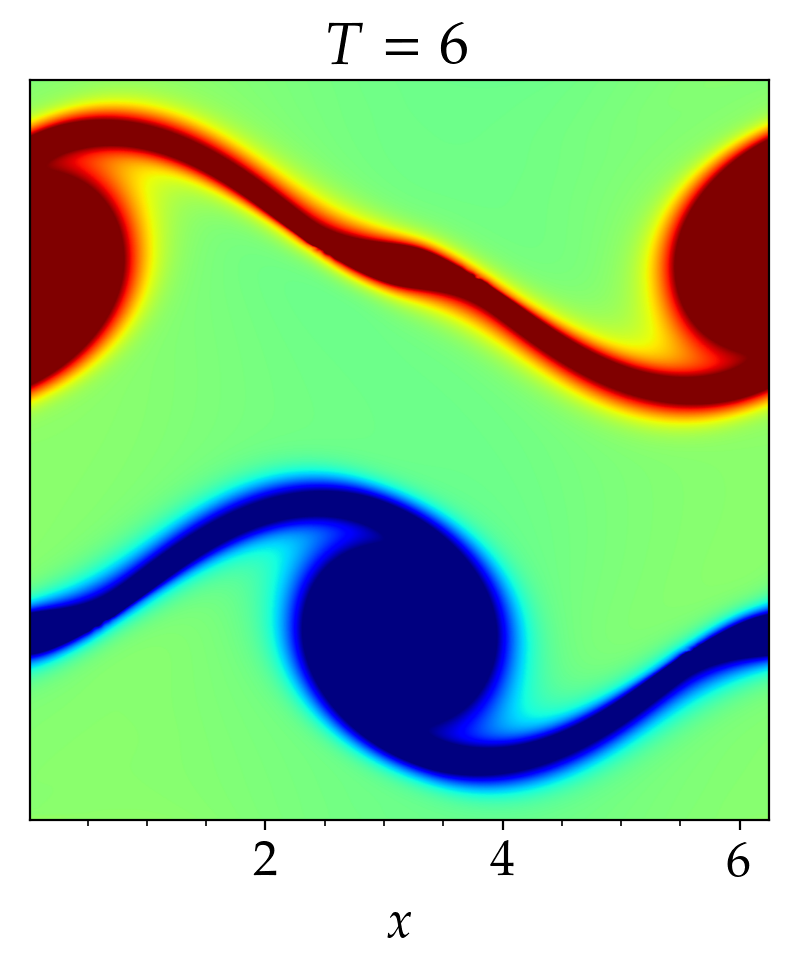}
  \includegraphics[height=0.18\textheight]{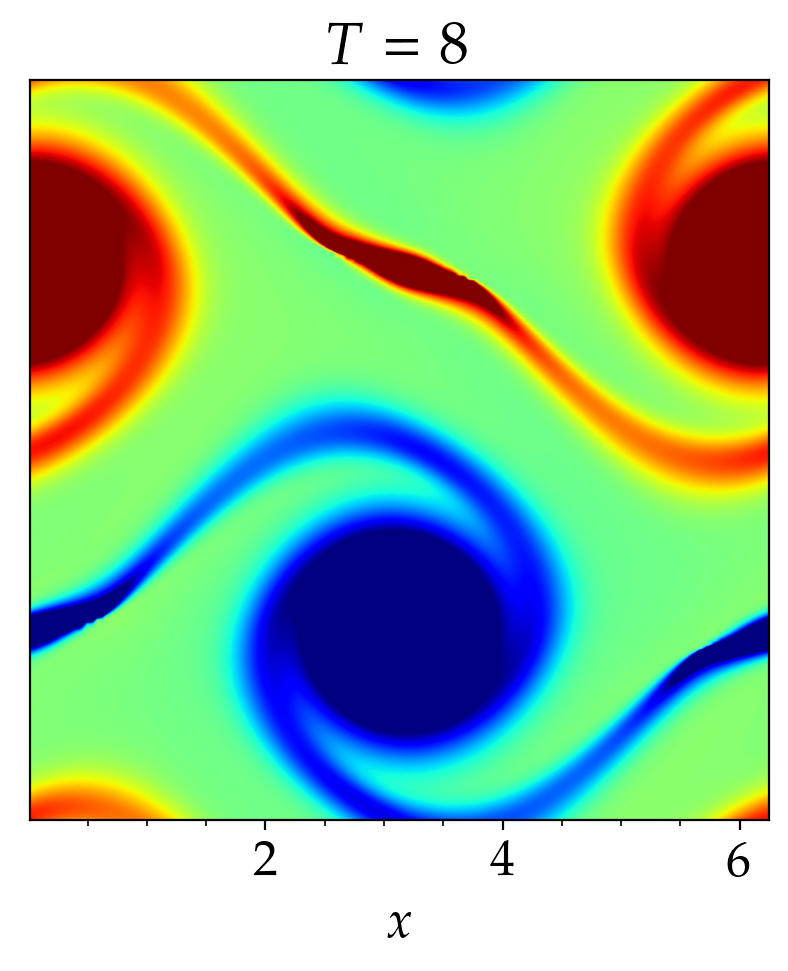}
  \includegraphics[height=0.18\textheight]{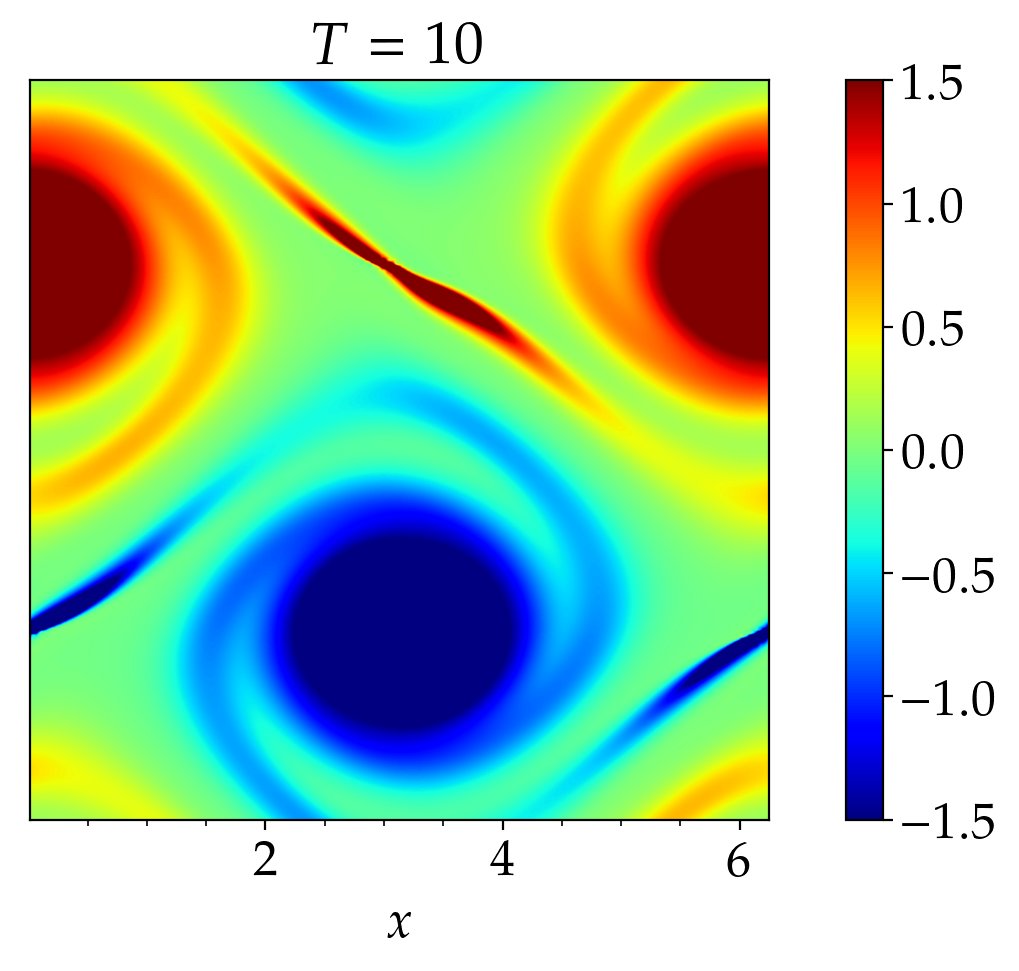}
  \caption{Vorticity at times $T= 4, 6, 8, 10$ for the shear flow
    problem computed using the limiting scheme
    \eqref{eq:dis_incomp_mas}-\eqref{eq:dis_incomp_mom}.}   
  \label{fig:shear_flow_lim_vorticity}
\end{figure}

\section{Concluding Remarks}
\label{sec:conclusion}
An AP and entropy stable scheme for the barotropic Euler system in the
zero Mach number limit is designed, analysed and implemented. The
entropy stability is achieved by introducing a velocity shift
proportional to the pressure gradient in the convective fluxes, which
is responsible for the dissipation of mechanical energy. The numerical
scheme obtained by using semi-implicit in time and upwind in space
finite volume discretisations, possesses several attractive features:
the positivity of density, the entropy stability and the consistency
with continuous equations when the grid parameters go to zero. The
scheme is stable under a CFL condition which permits large time-steps
at low Mach numbers. The apriori entropy stability estimate leads to a
bound on the second order pressure which in turn allows the passage to
the zero Mach number limit to get a consistent, velocity stabilised
scheme for the incompressible limit system. The results of numerical
case studies clearly validate the positivity preservation, the correct shock
capturing capabilities in the compressible regimes, the ability to
simulate weakly compressible flows even for ill-prepared initial data
and the AP property in the zero Mach number limit with exact
convergence rates with respect to the vanishing Mach number. Numerical 
experiments indicate a first order convergence for the scheme.

\section*{Data Availability and Declaration}
\label{sec:data-availability}

Enquiries about data availability should be directed to the
authors. The authors declare to have no conflict of interest. 

\bibliographystyle{abbrv}
\bibliography{references}

\end{document}